\documentclass[12pt, a4paper]{article}
\usepackage{}
\usepackage{amsfonts}
\usepackage{mathbbold}
\usepackage{bbm}
\usepackage{mathrsfs}
\usepackage{latexsym}
\usepackage{graphicx}
\usepackage{amssymb}

\newtheorem{theorem}{Theorem}[section]
\newtheorem{lemma}[theorem]{Lemma}
\newtheorem{corollary}[theorem]{Corollary}

\title{{\Large \bf  The domination number and the least $Q$-eigenvalue II\thanks{Supported by NSFC
(No. 11771376), ``333" Project of  Jiangsu (2016) and JSNSFC (BK2012245).}}}
\author{Guanglong Yu$^a$\thanks{E-mail addresses:
yglong01@163.com.}
~ Mingqing Zhai$^b$\thanks{Supported by KPPAT of Anhui (JXBJZD 2016082).} ~ Chao Yan$^c$\thanks{Supported by NSFCU of Jiangsu (16KJB110011).} ~ Shu-guang Guo$^a$ ~
\\ ~ \\
{\footnotesize $^a$Department of Mathematics, Yancheng Teachers
University,}\\ {\footnotesize  Yancheng, 224002, Jiangsu, China}\\
{\footnotesize $^b$ School of Mathematics and Finance, Chuzhou University,}\\
{\footnotesize Chuzhou, 239000, Anhui, China}\\
{\footnotesize $^c$Department of Mathematics, Pujiang Institute, Nanjing Tech University,}\\
{\footnotesize Nanjing, 211101, Jiangsu, China}}
\date{}

\begin{document}
%\openup 1.0\jot
\maketitle

\begin{abstract}
Denote by $L_{g, l}$ the $lollipop$ $graph$ obtained by attaching a pendant path $\mathbb{P}=v_{g}v_{g+1}\cdots v_{g+l}$ ($l\geq 1$) to a cycle $\mathbb{C}=v_{1}v_{2}\cdots v_{g}v_{1}$ ($g\geq 3$). A $\mathcal {F}_{g, l}$-$graph$ of order $n\geq g+1$ is defined to be the graph obtained by attaching $n-g-l$ pendent vertices to some of the nonpendant vertices of $L_{g, l}$ in which each vertex other than $v_{g+l-1}$ is attached at most one pendant vertex. A $\mathcal {F}^{\circ}_{g, l}$-graph is a $\mathcal {F}_{g, l}$-$graph$ in which $v_{g}$ is attached with pendant vertex. Denote by $q_{min}$ the $least$ $Q$-$eigenvalue$ of a graph. In this paper, we proceed on considering the domination number, the least $Q$-eigenvalue of a graph as well as their relation. Further results obtained are as follows:

$\mathrm{(i)}$ some results about the changing of the domination number under the structural perturbation of a graph are represented;

$\mathrm{(ii)}$ among all nonbipartite unicyclic graphs of order $n$, with both domination number $\gamma$ and girth $g$ ($g\leq n-1$), the minimum $q_{min}$ attains at a $\mathcal {F}_{g, l}$-graph for some $l$;

$\mathrm{(iii)}$ among the  nonbipartite graphs of order $n$ and with given domination number which contain a $\mathcal {F}^{\circ}_{g, l}$-graph as a subgraph, some lower bounds for $q_{min}$ are represented;

$\mathrm{(iv)}$ among the nonbipartite graphs of order $n$ and with given domination number $\frac{n}{2}$, $\frac{n-1}{2}$, the minimum $q_{min}$ is completely determined respectively;

$\mathrm{(v)}$ among the nonbipartite graphs of order $n\geq 4$, and with both domination number $\frac{n+1}{3}<\gamma\leq \frac{n}{2}$ and odd-girth (the length of the shortest odd cycle) at most $5$, the minimum $q_{min}$ is completely determined.


\end{abstract}
\baselineskip 18.6pt

\section{Introduction}

\ \ \ The concept of graph dominating set has been frequently used for studying ad hoc networks \cite{YWWY}, the efficiency of multicast/broadcast routing \cite{EWB}, and the power management
\cite{YWWY}. As a result, a comprehensive
study of issues relevant to dominating set of a graph has become an active topic \cite{YWWY}.

All graphs considered in this paper are connected, undirected and
simple, i.e., no loops or multiple edges are allowed.
We denote by $\parallel S\parallel$ the $cardinality$ of a set $S$,
and denote by $G=G[V(G)$, $E(G)]$ a graph with vertex set
$V(G)=\{v_1, v_2, \ldots, v_n\}$ and edge set $E(G)$ where $\parallel V(G)\parallel= n$ is the $order$
and $\parallel E(G)\parallel= m$ is the $size$.

In a graph, if vertices $v_{i}$ and $v_{j}$ are adjacent (denoted by $v_{i}\sim v_{j}$), we say that they $dominate$ each other.
A vertex set $D$ of a graph $G$ is said to be a $dominating$ $set$ if every vertex of $V(G)\setminus D$ is
adjacent to (dominated by) at least one vertex in $D$. The $domination$ $number$ $\gamma(G)$ ($\gamma$, for short) is the
minimum cardinality of all dominating sets of $G$. For a graph $G$, a dominating set is called a $minimal$ $dominating$ $set$ if its cardinality is $\gamma(G)$. A well known result about $\gamma(G)$ is that for a graph $G$
of order $n$ containing no isolated vertex, $\gamma\leq \frac{n}{2}$ \cite{ORE}. In the study of a real-world network, what about the structure of the
network with fixed domination number and how about
the changing of the domination number under some structural perturbations of the network are very significant problems. In this paper, for a nonbipartite graph with both
order $n$ and domination number $\gamma$, we present some results
about the changing of the domination number under some structural perturbations.

Recall
that $Q(G)= D(G) + A(G)$ is called the $signless$ $Laplacian$ $matrix$ (or $Q$-$matrix$) of $G$, where $D(G)=
\mathrm{diag}(d_{1}, d_{2},
\ldots, d_{n})$ with $d_{i}= d_{G}(v_{i})$ being the degree of
vertex $v_{i}$ $(1\leq i\leq n)$, and $A(G)$ is the adjacency matrix of $G$. The signless
Laplacian has attracted the
attention of many researchers and it is
being promoted by many researchers \cite{CCRS}, \cite{D.P.S}-\cite{LOA}, \cite{WF}, \cite{ZJ2}.

The least eigenvalue of $Q(G)$,
denote by $q_{min}(G)$ or $q_{min}$, is called the $least$ $Q$-$eigenvalue$ of $G$. Noting that $Q(G)$ is positive semi-definite,
we have $q_{min}(G)\geq 0.$
From \cite{D.P.S}, we know that, for a connected
graph $G$, $q_{min}(G)= 0$ if and only if $G$ is bipartite.
Consequently, in \cite{DR}, $q_{min}$ was studied as a measure of nonbipartiteness of a graph. One can note
that there are quite a few results about $q_{min}$. In \cite{CCRS}, D.M. Cardoso et al. determined the graphs
with the the minimum $q_{min}$ among all the connected nonbipartite
graphs with a prescribed number of vertices. In \cite{LOA}, L. de Lima et al. surveyed some known results about $q_{min}$ and also presented some
new results. In \cite{FF}, S. Fallat, Y. Fan investigated the relations
between $q_{min}$ and some parameters reflecting
the graph bipartiteness. In \cite{WF}, Y. Wang, Y. Fan investigated $q_{min}$
of a graph under some perturbations, and minimized $q_{min}$ among the connected graphs with fixed order
which contains a given nonbipartite graph as an induced subgraph.

In \cite{YFT} and \cite{YGZW}, the authors first considered the relation between $q_{min}$ of a graph and its domination number. Among all the nonbipartite graphs with both order $n\geq 4$ and domination number $\gamma\leq \frac{n+1}{3}$, they characterized the graphs with the minimum $q_{min}$. A remaining open problem is that how about $q_{min}$ of the connected nonbipartite graph on $n$ vertices with domination number $\frac{n+1}{3}<\gamma\leq \frac{n}{2}$. In this paper, we proceed on considering this problem. Further results about the domination number, the least $Q$-eigenvalue of a graph as well as their relation are represented.

The layout of this paper is as follows: section 2 introduces some notations and working lemmas; section 3 represents some results to characterize the relation between the domination number and the structural perturbation of a graph; section 4 represents some results about the minimum $q_{min}$ among uncyclic graphs; section 5 represents some results about the minimum $q_{min}$ among all general graphs; section 6 poses an open problem.

\section{Preliminary}

\ \ \ \ In this section, we introduce some notations and some working lemmas.

In a graph $G$, we say that a vertex $v$ is $dominated$ by a vertex set
$S$ if $v\in S$ or $v$ is adjacent to a vertex in $S$.  A vertex is called a $p$-$dominator$ (or $support$ $vertex$) if it dominates a pendant vertex. For a vertex $v\in V(G)$, we denote by $N_{G}(v)$ its $neighbor$ $set$ in $G$, and let $N_{G}[v]=N_{G}(v)\cup\{v\}$ be the $close$ $neighbor$ $set$; for a subgraph $K$ in $G$, we denote by $N_{K}(v)$ the $neighbor$ $set$ of $v$ in $K$.

Denote by $P_n$, $C_{n}$, $K_{n}$, $S_{n}$ a $path$, a $n$-cycle (of length $n$), a $complete$ graph, a $star$ of order $n$ respectively, and denote by
 $K_{r,s}$ the $complete$ $bipartite$ graph with partite
sets of order $r$ and $s$ respectively. Denote by $S^{+}_{n}$ the graph obtained by adding an additional edge between two pendant vertices of $S_{n}$. For a path $P$ and a cycle
$C$, we denote by $L(P)$, $L(C)$ their $lengths$ respectively. If $k$ is odd, we say $C_{k}$ an $odd$ $cycle$. The $odd$-$girth$ for a nonbipartite graph $G$,
denoted by $g_{o}(G)$ or $g_{o}$, is the length of the shortest odd cycle in this graph.  $G-v_{i}v_{j}$
 denotes the graph obtained from $G$ by deleting the edge $v_{i}v_{j}\in
 E(G)$, and let $G-v_{i}$
 denote the graph obtained from $G$ by deleting the vertex $v_{i}$ and the edges incident with $v_{i}$.
 Similarly,  $G+v_{i}v_{j}$ is the graph obtained from $G$ by adding an edge $v_{i}v_{j}$ between its two nonadjacent vertices $v_{i}$
 and $v_{j}$. Given an edge set $E$, $G-E$ denotes the graph obtained by deleting all the edges in $E$ from $G$; given an vertex set $S$, $G-S$ denotes the graph obtained by deleting all the vertices in $S$ from $G$ and the edges incident with any vertex in $S$.

A connected graph $G$ of order $n$ is called a $unicyclic$ graph if
$\|E(G)\|=n$. For $S\subseteq V(G)$, let $G[S]$
denote the subgraph induced by $S$.
Denoted by $d_{G}(v_{i}, v_{j})$ the $distance$ between two vertices
$v_{i}$ and $v_{j}$ in a graph $G$.

Let $G_1$ and $G_2$ be two disjoint graphs, and let $v_1\in V(G_1),$  $v_2\in V(G_2)$. The $coalescence$ of $G_1$ and $G_2$,
denoted by $G_1(v_1)\diamond G_2(v_2)$, is obtained from $G_1$, $G_2$ by identifying $v_1$ with $v_2$ and forming a new vertex
$u$. The graph $G_1(v_1)\diamond G_2(v_2)$ can also be written as $G_1(u)\diamond G_2(u)$, where for $i=1, 2$, $G_{i}$ can be trivial (that is, $G_{i}$ is only one vertex). In $G_1(v_1)\diamond G_2(v_2)$, we say that $G_1$ is attached to $G_2$ at $u$ (or $G_1$ is attached to $u$ simply), or  $G_2$ is attached to $G_1$ at $u$ (or $G_2$ is attached to $u$ simply). If a connected
graph $G$ can be expressed
in the form $G = G_1(u)\diamond G_2(u)$, where both $G_1$ and $G_2$ are connected, then for $i=1$, $2$, $G_i$ is
called a
$branch$ of $G$ with root $u$. In $G = G_1(u)\diamond G_2(u)$, if $G_{2}$ is a path where $u$ is one end vertex of $G_{2}$, then we say that $G_{2}$ is a $pendant$ $path$ starting from $G_{1}$.

For a graph $G$ of order $n$, let $X=(x_1, x_2, \ldots, x_n)^T \in R^n$ be defined on $V(G)$, i.e.,
each vertex $v_i$ is mapped to the entry $x_i$; let $|x_i|$ denote the $absolute$ $value$ of $x_i$.
One can find
that $X^TQ(G)X =\sum_{v_iv_j\in E(G)}(x_{i} + x_{j})^2.$
In addition, for an arbitrary unit vector $X\in R^n$, $q_{min}(G) \le X^TQ(G)X$,
with equality if and only if $X$ is an eigenvector corresponding to $q_{min}(G)$.
A branch $H$ of $G$ is called a $zero$ $branch$ with respect to $X$ if $x_{i} = 0$  for all $v_{i} \in V(H)$; otherwise, it
is called a $nonzero$ $branch$ with respect to $X$.

\begin{lemma}{\bf \cite{D.R.S}} \label{le02,01} %------
Let $G$ be a graph on $n$ vertices and $m$ edges, and let $e$ be an
edge of $G$. Let $q_{1}\geq q_{2}\geq \cdots \geq q_{n}$ and
$s_{1}\geq s_{2}\geq \cdots \geq s_{n}$ be the $Q$-eigenvalues of  $G$
and $G-e$ respectively. Then $0\leq s_{n}\leq q_{n}\leq \cdots \leq
s_{2}\leq q_{2}\leq s_{1}\leq q_{1}.$
\end{lemma}

\begin{lemma}{\bf \cite{WF}}\label{le02,02} %--
Let $G$ be a connected graph which contains a bipartite branch $H$ with root $v_{s}$, and let $X$ be an eigenvector of $G$
corresponding
to $q_{min}(G)$.

{\normalfont (i)} If $x_{s} = 0$, then $H$ is a zero branch of G with respect to $X$;

{\normalfont (ii)} If $x_{s}\neq 0$, then $x_{p}\neq 0$ for every vertex $v_{p}\in V(H)$. Furthermore, for every vertex $v_{p}\in V(H)$,
$x_{p}x_{s}$ is either positive or negative depending on whether $v_{p}$ is or is not in the same part of the bipartite graph $H$ as
$v_{s}$; consequently, $x_{p}x_{t} < 0$ for each edge $v_{p}v_{t} \in E(H)$.
\end{lemma}

\begin{lemma}{\bf \cite{WF}}\label{le02,03} %------
Let $G$ be a connected nonbipartite graph of order $n$, and let $X$ be an eigenvector of $G$ corresponding to $q_{min}(G)$.
$T$ is a tree which is a nonzero branch of $G$ with respect to $X$ and with root $v_{s}$. Then $|x_{t}| < |x_{p}|$
 whenever $v_{p},$ $v_{t}$
are vertices of $T$ such that $v_{t}$ lies on the unique path from $v_{s}$ to $v_{p}$.
\end{lemma}

\begin{lemma}{\bf \cite{YGX}}\label{le0,04} %------
 Let $G = G_1(v_2) \diamond T(u)$ and $G^* = G_1(v_1)\diamond T(u)$, where $G_1$ is a connected nonbipartite
graph containing two distinct vertices $v_1, v_2$, and $T$ is a nontrivial tree. If there exists an
 eigenvector $X=(\,x_1$, $x_2$, $\ldots$, $x_k$, $\ldots)^T$ of $G$ corresponding to $q_{min}(G)$ such that
 $|x_1| > |x_2|$ or $|x_1| = |x_2| > 0$, then $q_{min}(G^*)<q_{min}(G)$.
\end{lemma}

Let $k\geq 3$ be odd, and let $\mathcal {C}=v_1v_2\cdots v_kv_1$ be a cycle of length $k$. For $j=1$, $2$, $\ldots$, $t$,
each $T_{j}$ is a nontrivial tree. Let $\mathcal {C}^{(T_{1}, T_{2}, \ldots, T_{t};i_{1}, i_{2}, \ldots, i_{t})}_{k}$
denote the graph obtained by identifying the vertex $u_{j}$ of $T_{j}$ and the vertex $v_{i_{j}}$ of $\mathcal {C}$,
where $1\leq j\leq t$ and for $1\leq l< f\leq t$, $i_{l}\neq i_{f}$.

\begin{lemma}{\bf \cite{YGZW}}\label{le0,05} %------
Let $k< n$ be odd and $\mathcal {C}^{(T_{1}, T_{2}, \ldots, T_{t};i_{1}, i_{2}, \ldots, i_{t})}_{k}$ be of order $n$.
$X=(\,x_1$, $x_2$, $\ldots$, $x_k$, $x_{k+1}$, $x_{k+2}$, $\ldots$, $x_{n-1}$, $x_{n}\,)^T$ is a
unit eigenvector corresponding to $q_{min}(\mathcal {C}^{(T_{1}, T_{2}, \ldots, T_{t};i_{1}, i_{2}, \ldots, i_{t})}_{k})$.
Then $\max\{|x_{i_{j}}|\, |\, 1\leq j\leq t\}>0$.
\end{lemma}

\begin{lemma}{\bf \cite{KDAS}}\label{le0,06} %------
Let $G$ be a connected graph of order $n$. Then
$q_{min}<\delta$, where $\delta$ is the minimal vertex degree of $G$.
\end{lemma}

\begin{lemma}{\bf \cite{YGZW}}\label{le0,07} %------
Let $G$ be a nonbipartite graph with domination number $\gamma(G)$. Then $G$ contains a nonbipartite
unicyclic spanning subgraph $H$
with both $g_{o}(H)=g_{o}(G)$ and $\gamma(H)=\gamma(G)$.
\end{lemma}

\begin{lemma}{\bf \cite{YGZW}}\label{le0,08} %------
Suppose a graph $G$ contains pendant vertices. Then

$\mathrm{(i)}$ there must be a minimal dominating set of $G$ containing
all of its $p$-dominators but no any pendant vertex;

$\mathrm{(ii)}$ if $v$ is a $p$-dominator of $G$ and at least two pendant vertices are adjacent to $v$, then any
minimal dominating set of $G$ contains $v$ but no any pendant vertex adjacent to $v$.
\end{lemma}

\begin{lemma}{\bf \cite{MHM}}\label{le0,09} %------
(i) For a path $P_{n}$, we have
$\gamma(P_{n})=\lceil\frac{n}{3}\rceil$.

(ii) For a cycle $C_{n}$, we have
$\gamma(C_{n})=\lceil\frac{n}{3}\rceil$.
\end{lemma}

We define the corona $G$ of graphs $G_{1}$ and $G_{2}$ as follows. The $corona$ $G = G_{1}\circ G_{2}$ is the graph
formed from one copy of $G_{1}$ and $\parallel V (G_{1})\parallel$ copies of $G_{2}$ where the $i$th vertex of $G_{1}$ is adjacent to
every vertex in the $i$th copy of $G_{2}$.

\begin{lemma}{\bf \cite{CPNH}}\label{le0,10} %------
Let $G$ be a graph of order $n$. $\gamma(G) =\frac{n}{2}$ if
and only if the components of $G$ are the cycle $C_{4}$ or the corona $H\circ K_{1}$ for any connected graph $H$.
\end{lemma}

\setlength{\unitlength}{0.5pt}
\begin{center}
\begin{picture}(370,127)
\put(7,101){\circle*{4}}
\put(7,34){\circle*{4}}
\qbezier(7,101)(7,68)(7,34)
\put(64,62){\circle*{4}}
\qbezier(7,34)(35,48)(64,62)
\qbezier(7,101)(35,82)(64,62)
\put(168,62){\circle*{4}}
\qbezier(64,62)(116,62)(168,62)
\put(337,61){\circle*{4}}
\put(182,62){\circle*{4}}
\put(204,62){\circle*{4}}
\put(193,62){\circle*{4}}
\put(116,62){\circle*{4}}
\put(215,62){\circle*{4}}
\put(264,62){\circle*{4}}
\qbezier(215,62)(239,62)(264,62)
\put(331,112){\circle*{4}}
\qbezier(264,62)(297,87)(331,112)
\put(335,93){\circle*{4}}
\qbezier(264,62)(299,78)(335,93)
\put(337,71){\circle*{4}}
\put(336,81){\circle*{4}}
\put(340,45){\circle*{4}}
\qbezier(264,62)(302,54)(340,45)
\put(65,67){$v_{3}$}
\put(2,107){$v_{2}$}
\put(1,19){$v_{1}$}
\put(113,68){$v_{4}$}
\put(248,47){$v_{3+k}$}
\put(333,116){$v_{3+k+1}$}
\put(339,92){$v_{3+k+2}$}
\put(345,38){$v_{n}$}
\put(120,-12){Fig. 2.1. $C_{3,\, k}^*$}
\end{picture}
\end{center}

Denote by $C_{3,\, k}^*$ the graph obtained by attaching a $C_{3}$ to an end vertex of a path of length $k$ and attaching $n-3-k$ pendant vertices to the other end vertex of this path (see Fig. 2.1).

\begin{lemma}{\bf \cite{YGZW}}\label{le0,11} %------
$\gamma(C_{3,\, k}^*)=\gamma(P_{k+3})$.
\end{lemma}

\section{Domination number and the structure of a graph}

\begin{lemma}\label{le03,01} %------
Suppose $G=H(u)\diamond S_{k}(u)$ where $H$ is a connected simple graph of order at
least $2$, $S_{k}$ ($k\geq 3$) is a star with center $v$ and and $u$ is a pendant vertex of $S_{k}$ (see Fig. 3.1).
Then $\gamma(G)-1\leq\gamma(H)\leq \gamma(G)$.
\end{lemma}

\setlength{\unitlength}{0.7pt}
\begin{center}
\begin{picture}(251,105)
\qbezier(0,64)(0,79)(15,91)\qbezier(15,91)(30,102)(52,102)\qbezier(52,102)(73,102)(88,91)\qbezier(88,91)(104,79)(104,64)
\qbezier(104,64)(104,48)(88,36)\qbezier(88,36)(73,25)(52,25)\qbezier(52,26)(30,26)(15,36)\qbezier(15,36)(0,48)(0,64)
\put(103,61){\circle*{4}}
\put(178,61){\circle*{4}}
\qbezier(103,61)(140,61)(178,61)
\put(237,105){\circle*{4}}
\qbezier(178,61)(207,83)(237,105)
\put(251,80){\circle*{4}}
\qbezier(178,61)(214,71)(251,80)
\put(242,24){\circle*{4}}
\qbezier(178,61)(210,43)(242,24)
\put(247,41){\circle*{4}}
\put(249,51){\circle*{4}}
\put(250,63){\circle*{4}}
\put(171,76){$v$}
\put(109,74){$u$}
\put(49,65){$H$}
\put(101,-9){Fig. 3.1. $G$}
\end{picture}
\end{center}

\begin{proof}
Firstly, we prove that $\gamma(G)-1\leq\gamma(H)$. Otherwise, suppose $\gamma(H)\leq \gamma(G)-2$, and suppose $D$ is
a dominating set of $H$. Then $D^{'}=D\cup \{v\}$ is a dominating set of $G$. This implies that $\gamma(G)\leq |D^{'}|
\leq \gamma(G)-1$, which is a contradiction. The result follows as desired.

Secondly, we prove that $\gamma(H)\leq \gamma(G)$. By Lemma \ref{le0,08}, graph $G$ has a dominating set $D$ containing all $p$-dominators but no any pendant vertex. Because $v$ is a $p$-dominator, $v\in D$. If $u\notin D$, then $(D\setminus \{v\})\cup \{u\}$ is a dominating set of $H$; if $u\in D$, then
$D\setminus \{v\}$ is a dominating set of $H$. This implies $\gamma(H)\leq \gamma(G)$. \ \ \ \ \ $\Box$
\end{proof}

Recall that a $lollipop$ $graph$ $L_{g, l}$ is $\mathbb{P}(v_{g})\diamond \mathbb{C}(v_{g})$ where $\mathbb{P}=v_{g}v_{g+1}\cdots v_{g+l}$ is a pendant path  with
length $l\geq 1$, $\mathbb{C}=v_{1}v_{2}\cdots v_{g}v_{1}$ is a $g$-cycle. For given $g$ and $l$, a graph of order $n$ is called a $F_{g, l}$-$graph$ if it
is obtained by attaching $n-g-l$ pendant vertices to some nonpendant vertices of a $L_{g, l}$. If $l=1$, a $F_{g, l}$-$graph$ is also called a $sunlike$ graph. For example, the graph $G$ shown in Fig. 3.2 is
a $F_{7, 6}$-$graph$.

\setlength{\unitlength}{0.7pt}
\begin{center}
\begin{picture}(507,143)
\qbezier(14,74)(14,90)(32,102)\qbezier(32,102)(50,115)(76,115)\qbezier(76,115)(101,115)(119,102)
\qbezier(119,102)(138,90)(138,74)
\qbezier(138,74)(138,57)(119,45)\qbezier(119,45)(101,33)(76,33)\qbezier(76,33)(50,33)(32,45)\qbezier(32,45)(14,57)(14,74)
\put(137,72){\circle*{4}}
\put(507,72){\circle*{4}}
\qbezier(137,72)(322,72)(507,72)
\put(31,46){\circle*{4}}
\put(0,31){\circle*{4}}
\qbezier(31,46)(15,39)(0,31)
\put(72,34){\circle*{4}}
\put(56,113){\circle*{4}}
\put(35,143){\circle*{4}}
\qbezier(56,113)(45,128)(35,143)
\put(402,72){\circle*{4}}
\put(343,72){\circle*{4}}
\put(17,87){\circle*{4}}
\put(112,107){\circle*{4}}
\put(143,135){\circle*{4}}
\qbezier(112,107)(127,121)(143,135)
\put(210,72){\circle*{4}}
\put(210,103){\circle*{4}}
\qbezier(210,72)(210,88)(210,103)
\put(278,72){\circle*{4}}
\put(278,98){\circle*{4}}
\qbezier(278,72)(278,85)(278,98)
\put(444,71){\circle*{4}}
\put(444,40){\circle*{4}}
\qbezier(444,71)(444,56)(444,40)
\put(201,-9){Fig. 3.2. a $F_{7, 6}$-$graph$}
\put(116,43){\circle*{4}}
\end{picture}
\end{center}

In a $F_{g, l}$-graph if each $p$-dominator other than $v_{g+l-1}$ is  attached with exactly one pendant vertex, then this graph is called a $\mathcal {F}_{g, l}$-$graph$. In the following paper, for unity, for a $\mathcal {F}_{g, l}$-$graph$, $\mathbb{C}$ and $\mathbb{P}$ are expressed as above.

\begin{theorem}\label{th03,02} %------
Among all nonbipartite unicyclic graphs of order $n$, and with both domination number $\gamma$ and girth $g$ ($g\leq n-1$), the
minimum $q_{min}$ attains at a $\mathcal {F}_{g, l}$-graph $G$ for some $l$. Moreover, for this graph $G$, suppose that $X=(x_1$, $x_2$, $x_3$, $\ldots$, $x_n)^T$ is a unit
eigenvector corresponding to $q_{min}(G)$. Then we have that $| x_{g}|> 0$, and $|x_{g+l-1}|=\max\{|x_{i}|\mid v_{i}$ is a $p$-dominator$\,\}$.
\end{theorem}

\begin{proof}
Suppose that graph $G$ is a nonbipartite unicyclic of order $n$ whose $q_{min}(G)$ attains the minimum.

Firstly, we prove that $G$ is a $F_{g, l}$-graph for some $l$. We give a proof by contradiction. Suppose that
$G$ is not a $F_{g, l}$-graph for any $l$. Then $G \cong \mathcal {C}^{(T_{1}, T_{2}, \ldots, T_{t};\, i_{1}, i_{2}, \ldots, i_{t})}_{g}$
for some $\mathcal {C}$ and $t$. Suppose that $X=(x_1$, $x_2$, $x_3$, $\ldots$, $x_n)^T$ is a unit eigenvector of $G$
corresponding to $q_{min}(G)$. By Lemma $\ref{le0,05}$, we know that $\max\{|x_{i_{j}}|\, |\, 1\leq j\leq t\}>0$. Without
loss of generality, suppose that $|x_{i_{1}}|=\max\{|x_{i_{j}}|\, |\, 1\leq j\leq t\}$. By Lemma $\ref{le02,02}$, with respect to $X$, we see that
among all $p$-dominators, the one with the largest corresponding entry must be in one nonzero $T_{j}$ (that is, $T_{j}$ a nonzero branch).  Without
loss of generality, we suppose that among all $p$-dominators,
the entry $x_{f-1}$ corresponding to the $p$-dominator $v_{f-1}$ in $T_{1}$ attains the maximum. Note that $|x_{i_{1}}|>0$. By Lemmas $\ref{le02,03}$, it follows that
$|x_{f-1}|>0$. Suppose that $v_{f}$ is a pendant vertex attached to $v_{f-1}$. Denote by $P_{i_{1}, f}$ the path from $v_{i_{1}}$ to $v_{f}$
in $T_{1}$.

\setlength{\unitlength}{0.7pt}
\begin{center}
\begin{picture}(616,301)
\put(73,114){\circle*{4}}
\put(60,143){\circle*{4}}
\qbezier(73,114)(66,129)(60,143)
\put(153,223){\circle*{4}}
\put(177,237){\circle*{4}}
\qbezier(153,223)(165,230)(177,237)
\put(50,164){\circle*{4}}
\put(53,158){\circle*{4}}
\put(56,152){\circle*{4}}
\put(35,194){\circle*{4}}
\put(38,188){\circle*{4}}
\put(41,182){\circle*{4}}
\put(211,253){\circle*{4}}
\put(214,246){\circle*{4}}
\put(216,239){\circle*{4}}
\put(0,273){\circle*{4}}
\put(31,202){\circle*{4}}
\qbezier(0,273)(15,238)(31,202)
\put(3,285){$v_{f}$}
\put(22,241){$v_{f-1}$}
\put(26,168){$v_{k}$}
\put(160,245){$v_{z}$}
\put(197,276){\circle*{4}}
\qbezier(177,237)(187,257)(197,276)
\put(206,264){\circle*{4}}
\qbezier(177,237)(191,251)(206,264)
\put(218,227){\circle*{4}}
\qbezier(177,237)(197,232)(218,227)
\put(186,292){$v_{z_{1}}$}
\put(210,270){$v_{z_{2}}$}
\put(224,223){$v_{z_{w}}$}
\put(111,203){\circle*{4}}
\put(91,194){\circle*{4}}
\put(101,198){\circle*{4}}
\put(46,173){\circle*{4}}
\put(77,189){\circle*{4}}
\qbezier(46,173)(61,181)(77,189)
\put(121,208){\circle*{4}}
\qbezier(153,223)(137,216)(121,208)
\put(151,209){$v_{z-1}$}
\put(118,195){$v_{z-2}$}
\put(72,101){$v_{i_{1}}$}
\qbezier(43,83)(43,98)(63,109)\qbezier(63,109)(84,120)(114,120)\qbezier(114,120)(143,120)(164,109)
\qbezier(164,109)(185,98)(185,83)\qbezier(185,83)(185,67)(164,56)\qbezier(164,56)(143,45)(114,45)
\qbezier(114,46)(84,46)(63,56)\qbezier(63,56)(42,67)(43,83)
\put(89,24){$P$ in $T_{1}$}
\qbezier(363,83)(363,98)(382,109)\qbezier(382,109)(401,120)(429,120)\qbezier(429,120)(456,120)(475,109)
\qbezier(475,109)(495,98)(495,83)\qbezier(495,83)(495,67)(475,56)\qbezier(475,56)(456,46)(429,46)
\qbezier(429,46)(401,46)(382,56)\qbezier(382,56)(363,67)(363,83)
\put(374,138){\circle*{4}}
\put(390,113){\circle*{4}}
\qbezier(374,138)(382,126)(390,113)
\put(361,161){\circle*{4}}
\put(365,154){\circle*{4}}
\put(369,147){\circle*{4}}
\put(15,240){\circle*{4}}
\put(542,197){\circle*{4}}
\put(336,207){\circle*{4}}
\put(314,245){\circle*{4}}
\put(354,175){\circle*{4}}
\qbezier(314,245)(334,210)(354,175)
\put(475,108){\circle*{4}}
\put(499,139){\circle*{4}}
\qbezier(475,108)(487,124)(499,139)
\put(505,146){\circle*{4}}
\put(510,153){\circle*{4}}
\put(516,161){\circle*{4}}
\put(522,169){\circle*{4}}
\put(568,232){\circle*{4}}
\qbezier(522,169)(545,201)(568,232)
\put(545,268){\circle*{4}}
\qbezier(568,232)(556,250)(545,268)
\put(574,268){\circle*{4}}
\qbezier(568,232)(571,250)(574,268)
\put(610,232){\circle*{4}}
\qbezier(568,232)(589,232)(610,232)
\put(583,256){\circle*{4}}
\put(590,249){\circle*{4}}
\put(598,242){\circle*{4}}
\put(308,259){$v_{f}$}
\put(342,212){$v_{f-1}$}
\put(390,101){$v_{i_{1}}$}
\put(462,95){$v_{i_{2}}$}
\put(567,218){$v_{z}$}
\put(516,205){$v_{z-1}$}
\put(492,176){$v_{z-2}$}
\put(518,274){$v_{z_{1}}$}
\put(569,282){$v_{z_{2}}$}
\put(617,227){$v_{z_{w}}$}
\put(404,24){$P$ in $T_{2}$}
\put(190,-9){Fig. 3.3. two cases for $P$}
\end{picture}
\end{center}

Because $G$ is not a $\mathcal {F}_{g, l}$-$graph$ for any $l$, there exists a path $P$ of length at least $2$ which has two cases as follows.

{\bf Case 1} $P$ is in $T_{1}$, and $P$ is a longest path from one vertex $v_{k}$ of $P_{i_{1}, f}$ to another pendant vertex of $T_{1}$ which is not dominated by $v_{f-1}$, where $v_{k}\neq v_{f-1}$ (see ``$P$ in $T_{1}$" in Fig. 3.3).
Suppose $P=v_{k}\cdots v_{z-1}v_{z}v_{z_{1}}$, where $v_{z_{1}}$ is a pendant vertex, $v_{z}\neq v_{f-1}$ and $L(P)\geq 2$. And suppose that
$v_{z_{1}}$, $v_{z_{2}}$, $\ldots$, $v_{z_{w}}$ are all the pendant vertices attached to $v_{z}$. Let $H=G-\{v_{z}$, $v_{z_{1}}$,
$v_{z_{2}}$, $\ldots$, $v_{z_{w}}\}$. Note that $L(P)\geq 2$. By Lemma \ref{le03,01}, $\gamma(G)-1\leq\gamma(H)\leq \gamma(G)$.

If $\gamma(H)=\gamma(G)-1$, let $G^{'}=H+ v_{f-1}v_{z}+\sum^{w}_{j=1} v_{f}v_{z_{j}}$. By Lemma \ref{le0,08}, we
know that there is a minimal dominating set $D_{H}$ of $H$ containing $v_{f-1}$ but no $v_{f}$.
Then $D_{H}\cup \{v_{f}\}$ is a dominating set of $G^{'}$. This implies $\gamma(G^{'})\leq \gamma(G)$. If $\gamma(G^{'})
\leq \gamma(G)-1$, then by Lemma \ref{le0,08}, we know that there is a minimal dominating set $D_{G^{'}}$ of $G^{'}$ containing $v_{f-1}$ and $v_{f}$ but no any $v_{z_{j}}$ for $j=1$, $2$, $\ldots$, $w$. Note
that $D_{G^{'}}\setminus \{v_{f}\}$ is a dominating set of $H$. This telles us that $\gamma(H)\leq\gamma(G)-2$, which
contradicts $\gamma(H)=\gamma(G)-1$. Thus, $\gamma(G^{'})= \gamma(G)$.

Now, let $Y=(y_1$, $y_2$, $y_3$, $\ldots$, $y_n)^T$ satisfies that
$$\left\{\begin{array}{cc}
                                y_{z}=-\mathrm{sgn}(x_{f-1})(|x_{z}|+(|x_{f-1}|-|x_{z-1}|));  & \\
                                y_{z_{j}}=-\mathrm{sgn}(x_{f})(|x_{z_{j}}|+(|x_{f}|-|x_{z}|)),\ \ \ \ \ \ \,  & j=1, 2, \ldots, w; \\
                                y_{i}=x_{i},\hspace{6.0cm}  \ & others.
                              \end{array}\right.$$
Note that $T_{1}$ is a nonzero branch and $|x_{f-1}|\geq |x_{z}|$. By Lemma \ref{le02,03}, it follows that $|x_{z-1}|<|x_{z}|\leq|x_{f-1}|< |x_{f}|$. Then $Y^{T}Y-X^{T}X> 0$, but $Y^{T}Q(G^{'})Y=X^{T}Q(G)X$. Using Rayleigh quotient follows that $q_{min}(G^{'})< q_{min}(G)$, which contradicts the minimality
of $q_{min}(G)$.

If $\gamma(H)=\gamma(G)$, let $G^{'}=H+ v_{f-1}v_{z}+\sum^{w}_{j=1} v_{f-1}v_{z_{j}}$. By Lemma \ref{le0,08}, we know that there is a minimal dominating set of $H$ containing all $p$-dominators but no any pendant vertices, and there is a minimal dominating set of $G^{'}$ containing all $p$-dominators but no any pendant vertex. Note that any such minimal dominating set of
$H$ which contains $v_{f-1}$ is also a dominating set of $G^{'}$,
and note that any such minimal dominating set of $G^{'}$ which contains $v_{f-1}$ is also a dominating set of $H$. This implies $\gamma(H)=\gamma(G^{'})=\gamma(G)$.

Similar to the proof for the above assumption that $\gamma(H)=\gamma(G)-1$, it is proved that $q_{min}(G^{'})<q_{min}(G)$, which contradicts the minimality of the $q_{min}$ of $G$.

{\bf Case 2} $P$ is a longest path in a $T_{j}$ ($j\neq 1$) which is from $v_{i_{j}}$ to a pendant vertex of $T_{j}$. For convenience, we let $j=2$  (see ``$P$ in $T_{2}$" in Fig. 3.3). Note that $G$ is unicyclic. $P$ has no common vertex with the path $P_{i_{1}, f}$. Suppose $P=v_{i_{2}}\cdots v_{z-1}v_{z}v_{z_{1}}$, where $v_{z_{1}}$ is a pendant vertex, $L(P)\geq 2$, and suppose
$v_{z_{1}}$, $v_{z_{2}}$, $\ldots$, $v_{z_{w}}$ are all the pendant vertices attached to $v_{z}$ (see ``$P$ in $T_{2}$" in Fig. 3.3). Let $H=G-\{v_{z}$, $v_{z_{1}}$,
$v_{z_{2}}$, $\ldots$, $v_{z_{w}}\}$. Note that $L(P)\geq 2$. By Lemma \ref{le03,01}, $\gamma(G)-1\leq\gamma(H)\leq \gamma(G)$. If $\gamma(H)=\gamma(G)-1$, let $G_{1}=H+ v_{f-1}v_{z}+
\sum^{w}_{j=1} v_{f}v_{z_{j}}$; if $\gamma(H)=\gamma(G)$, let $G_{2}=H+ v_{f-1}v_{z}+\sum^{w}_{j=1} v_{f-1}v_{z_{j}}$.
Note that $T_{1}$ is a nonzero branch, and note that if $T_{2}$ is a zero branch, then we get $|x_{z-1}|=|x_{z}|<|x_{f-1}|< |x_{f}|$; if $T_{2}$ is a nonzero branch, then we get $|x_{z-1}|<|x_{z}|\leq|x_{f-1}|< |x_{f}|$. Similar to Case 1, we can prove that for $i=1$, $2$, $\gamma(G_{i})= \gamma(G)$ and  $q_{min}(G_{i})<q_{min}(G)$ which contradicts the minimality of $q_{min}(G)$.

The above proof implies that $G$ is a $F_{g, l}$-graph for some $l$.

Next, we prove that  $G$ is a $\mathcal {F}_{g, l}$-graph. Otherwise, suppose $G$ is not a $\mathcal {F}_{g, l}$-graph. Then there exists a $p$-dominator other than $v_{g+l-1}$ whose attached pendant vertex is more than one. For such a $p$-dominator, let only one pendant vertex be kept and other pendant vertices be transferred to be attached to $v_{g+l-1}$. Then we get a $\mathcal {F}_{g, l}$-graph $\mathbb{F}$. By Lemma \ref{le0,04}, it is proved that $q_{min}(\mathbb{F})<q_{min}(G)$, which contradicts the minimality of $q_{min}(G)$. Thus, it follows that $G$ is a $\mathcal {F}_{g, l}$-graph for some $l$.

From the above proof, we see that if $X=(x_1$, $x_2$, $x_3$, $\ldots$, $x_n)^T$ is a unit
eigenvector corresponding to $q_{min}(G)$, then $| x_{g}|> 0$, and $|x_{g+l-1}|=\max\{|x_{i}|\mid v_{i}$ is a $p$-dominator$\,\}$.
The result follows as desired. \ \ \ \ \ $\Box$

\end{proof}

\setlength{\unitlength}{0.7pt}
\begin{center}
\begin{picture}(594,189)
\put(54,161){\circle*{4}}
\put(61,48){\circle*{4}}
\qbezier(54,161)(0,105)(61,48)
\put(67,161){\circle*{4}}
\put(77,161){\circle*{4}}
\put(87,161){\circle*{4}}
\put(400,110){\circle*{4}}
\put(410,110){\circle*{4}}
\put(420,110){\circle*{4}}
\put(35,136){\circle*{4}}
\put(7,150){\circle*{4}}
\qbezier(35,136)(21,143)(7,150)
\put(207,160){\circle*{4}}
\put(246,161){\circle*{4}}
\put(258,48){\circle*{4}}
\qbezier(246,161)(334,128)(258,48)
\put(338,111){\circle*{4}}
\put(293,111){\circle*{4}}
\put(566,134){\circle*{4}}
\put(570,126){\circle*{4}}
\put(574,118){\circle*{4}}
\put(29,103){\circle*{4}}
\put(40,73){\circle*{4}}
\put(109,47){\circle*{4}}
\put(218,19){\circle*{4}}
\put(103,161){\circle*{4}}
\qbezier(103,161)(174,161)(246,161)
\put(137,161){\circle*{4}}
\put(137,189){\circle*{4}}
\qbezier(137,161)(137,175)(137,189)
\put(279,142){\circle*{4}}
\put(282,78){\circle*{4}}
\put(436,110){\circle*{4}}
\put(579,110){\circle*{4}}
\qbezier(436,110)(507,110)(579,110)
\put(470,110){\circle*{4}}
\put(470,143){\circle*{4}}
\qbezier(470,110)(470,127)(470,143)
\put(524,110){\circle*{4}}
\put(564,144){\circle*{4}}
\qbezier(524,110)(544,127)(564,144)
\put(282,147){$v_{1}$}
\put(243,170){$v_{2}$}
\put(128,148){$v_{r_{1}}$}
\put(40,132){$v_{r_{j-1}}$}
\put(105,58){$v_{r_{j}}$}
\put(22,70){$v_{a}$}
\put(45,37){$v_{a+1}$}
\put(-2,101){$v_{a-1}$}
\put(274,110){$v_{g}$}
\put(329,122){$v_{g+1}$}
\put(210,-9){Fig. 3.4. $\mathcal {G}$}
\put(218,48){\circle*{4}}
\qbezier(218,48)(218,34)(218,19)
\put(384,111){\circle*{4}}
\put(288,73){$v_{g-1}$}
\put(565,95){$v_{g+l}$}
\put(109,17){\circle*{4}}
\qbezier(109,47)(109,32)(109,17)
\put(209,58){$v_{r_{t}}$}
\put(145,48){\circle*{4}}
\qbezier(61,48)(103,48)(145,48)
\put(191,48){\circle*{4}}
\qbezier(258,48)(224,48)(191,48)
\put(180,47){\circle*{4}}
\put(170,47){\circle*{4}}
\put(160,47){\circle*{4}}
\qbezier(293,111)(338,111)(384,111)
\end{picture}
\end{center}

Let $\mathcal {G}$ be a $\mathcal {F}_{g, l}$-graph of order $n$ (see Fig. 3.4). In $\mathcal {G}$, we denote by
$v_{r_{1}}$, $v_{r_{2}}$, $\ldots$, $v_{r_{t}}$ all the $p$-dominators on $\mathbb{C}$, where $1\leq r_{1}< r_{2}< \cdots< r_{t}\leq g$.
Along $\mathbb{P}$, from $v_{g}$ to $v_{g+l}$, suppose $v_{g+s}$ is the first $p$-dominator,
where $0\leq s\leq l-1$. Obviously, if $s=0$, then $r_{t}=g$. Hereafter, we use $\mathcal {F}^{\circ}_{g, l}$-graph to denote a $\mathcal {F}_{g, l}$-graph with $s=0$. If $i=1$, let $r_{i-1}=r_{t}$; if $i=t$, let $r_{t+1}=r_{1}$.  For $1\leq i\leq t-1$, suppose that $v_{\tau_{i}}$ is the pendant vertex attached to $v_{r_{i}}$; if $l\geq 2$ and $r_{t}=g$, suppose that $v_{\tau_{g}}$ is the pendant vertex attached to $v_{g}$; if $l= 1$, then $r_{t}=g$, and then suppose $v_{\tau^{1}_{g}}$, $v_{\tau^{2}_{g}}$, $\ldots$, $v_{\tau^{y}_{g}}$ are all the pendant vertices attached to vertex $v_{g}$. Let $v_{a}$ be a vertex of $\mathbb{C}$. For $1\leq i\leq g$, if $v_{a}=v_{1}$, we let $v_{a-i}=v_{g-i+1}$. For $1\leq a\leq g-2$, let $\mathcal {G}_{1}=\mathcal {G}-v_{a+1}v_{a+2}+v_{a+1}v_{a-1}$, $\mathcal {G}_{2}=\mathcal {G}_{1}-
v_{a-1}v_{a-2}+v_{a-1}v_{a+2}$ (see Fig. 3.5); for $a=g-1$, let $\mathcal {G}_{1}=\mathcal {G}-v_{g}v_{1}+v_{g}v_{g-2}$, $\mathcal {G}_{2}=\mathcal {G}_{1}-
v_{g-2}v_{g-3}+v_{g-2}v_{1}$; for $a=g$, let $\mathcal {G}_{1}=\mathcal {G}-v_{2}v_{1}+v_{1}v_{g-1}$, $\mathcal {G}_{2}=\mathcal {G}_{1}-
v_{g-1}v_{g-2}+v_{g-1}v_{2}$. If $l=1$, $r_{1}\geq 4$ and $r_{t}=g$, let $\mathscr{X}_{t}=\mathcal {G}-v_{2}v_{3}+v_{2}v_{g}-\sum\limits^{y}_{i=1}v_{g}v_{\tau^{i}_{g}}+
\sum\limits^{y}_{i=1}v_{3}v_{\tau^{i}_{g}}$; if $l\geq 2$, $r_{1}\geq 4$ and $r_{t}=g$, $\mathscr{X}_{t}=\mathcal {G}-v_{2}v_{3}+v_{2}v_{g}-v_{g}v_{\tau_{g}}+
v_{3}v_{\tau_{g}} -v_{g}v_{g+1}+ v_{g+1}v_{3}$ (see Fig. 3.5); if $r_{t}=g$, and for $1\leq i\leq t-1$, $r_{i+1}-r_{i}\geq 4$, let $\mathscr{X}_{i}=\mathcal {G}-v_{r_{i}+2}v_{r_{i}+3}+v_{r_{i}+2}v_{r_{i}}-v_{r_{i}}v_{\tau_{i}}+
v_{r_{i}+3}v_{\tau_{i}}$.
If $r_{1}\geq 4$ and $r_{t}=g$, $\mathscr{G}_{1}= \mathcal {G}-v_{r_{1}}v_{\tau_{1}}-v_{r_{1}}v_{r_{1}+1}+v_{r_{1}}v_{r_{1}-2}+
v_{1}v_{\tau_{1}}$ (see Fig. 3.5); for all $i=2$, $\ldots$, $t-1$, if $r_{i}-r_{i-1}\geq 4$, let $\mathscr{G}_{i}= \mathcal {G}-v_{r_{i}}v_{\tau_{i}}-v_{r_{i}}v_{r_{i}+1}+v_{r_{i}}v_{r_{i}-2}+
v_{r_{i-1}+1}v_{\tau_{i}}$; if $r_{t}< g$ and $r_{t}-r_{t-1}\geq 4$, $\mathscr{G}_{t}= \mathcal {G}-v_{r_{t}}v_{\tau_{t}}+
v_{r_{t-1}+1}v_{\tau_{t}}-v_{r_{t}}v_{r_{t}+1}+v_{r_{t}}v_{r_{t}-2}$; if $l=1$, $r_{t}= g$ and $g-r_{t-1}\geq 4$, $\mathscr{G}_{t}= \mathcal {G}-\sum\limits^{y}_{i=1}v_{g}v_{\tau^{i}_{g}}+
\sum\limits^{y}_{i=1}v_{r_{t-1}+1}v_{\tau^{i}_{g}}-v_{g}v_{1}+v_{g}v_{g-2}$; if $l\geq 2$, $r_{t}= g$ and $g-r_{t-1}\geq 4$, $\mathscr{G}_{t}= \mathcal {G}-v_{g}v_{\tau_{g}}+
v_{r_{t-1}+1}v_{\tau_{g}}-v_{g}v_{1}+v_{g}v_{g-2} -v_{g}v_{g+1}+ v_{g+1}v_{r_{t-1}+1}$.

\setlength{\unitlength}{0.7pt}
\begin{center}
\begin{picture}(510,309)
\put(12,259){\circle*{4}}
\put(12,214){\circle*{4}}
\qbezier(12,259)(12,237)(12,214)
\put(47,236){\circle*{4}}
\qbezier(12,259)(29,248)(47,236)
\qbezier(12,214)(29,225)(47,236)
\put(82,236){\circle*{4}}
\qbezier(47,236)(64,236)(82,236)
\put(99,235){\circle*{4}}
\put(91,235){\circle*{4}}
\put(108,235){\circle*{4}}
\put(118,235){\circle*{4}}
\put(147,235){\circle*{4}}
\qbezier(118,235)(132,235)(147,235)
\put(147,254){\circle*{4}}
\qbezier(147,235)(147,245)(147,254)
\put(147,262){\circle*{4}}
\put(147,270){\circle*{4}}
\put(147,278){\circle*{4}}
\put(147,287){\circle*{4}}
\put(129,309){\circle*{4}}
\qbezier(147,287)(138,298)(129,309)
\put(167,308){\circle*{4}}
\qbezier(147,287)(157,298)(167,308)
\put(156,308){\circle*{4}}
\put(140,308){\circle*{4}}
\put(148,308){\circle*{4}}
\put(171,235){\circle*{4}}
\qbezier(147,235)(159,235)(171,235)
\put(188,235){\circle*{4}}
\put(180,235){\circle*{4}}
\put(196,235){\circle*{4}}
\put(204,235){\circle*{4}}
\put(234,235){\circle*{4}}
\qbezier(204,235)(219,235)(234,235)
\put(510,259){\circle*{4}}
\put(510,213){\circle*{4}}
\put(478,234){\circle*{4}}
\put(446,234){\circle*{4}}
\put(340,234){\circle*{4}}
\put(332,234){\circle*{4}}
\put(349,234){\circle*{4}}
\put(359,234){\circle*{4}}
\put(388,234){\circle*{4}}
\put(388,253){\circle*{4}}
\put(388,261){\circle*{4}}
\put(388,269){\circle*{4}}
\put(388,277){\circle*{4}}
\put(388,286){\circle*{4}}
\put(409,308){\circle*{4}}
\put(397,307){\circle*{4}}
\put(381,307){\circle*{4}}
\put(389,307){\circle*{4}}
\put(412,234){\circle*{4}}
\put(429,234){\circle*{4}}
\put(421,234){\circle*{4}}
\put(437,234){\circle*{4}}
\put(292,234){\circle*{4}}
\put(322,234){\circle*{4}}
\qbezier(510,259)(510,236)(510,213)
\qbezier(510,259)(494,247)(478,234)
\qbezier(510,213)(494,224)(478,234)
\qbezier(478,234)(462,234)(446,234)
\qbezier(359,234)(373,234)(388,234)
\qbezier(388,234)(388,244)(388,253)
\put(370,308){\circle*{4}}
\qbezier(388,286)(379,297)(370,308)
\qbezier(388,286)(398,297)(409,308)
\qbezier(388,234)(400,234)(412,234)
\qbezier(292,234)(307,234)(322,234)
\put(13,99){\circle*{4}}
\put(13,54){\circle*{4}}
\put(48,76){\circle*{4}}
\put(82,76){\circle*{4}}
\put(98,75){\circle*{4}}
\put(90,75){\circle*{4}}
\put(107,75){\circle*{4}}
\put(119,75){\circle*{4}}
\put(170,75){\circle*{4}}
\put(170,94){\circle*{4}}
\put(170,102){\circle*{4}}
\put(170,110){\circle*{4}}
\put(170,118){\circle*{4}}
\put(170,127){\circle*{4}}
\put(152,149){\circle*{4}}
\put(190,148){\circle*{4}}
\put(179,148){\circle*{4}}
\put(163,148){\circle*{4}}
\put(171,148){\circle*{4}}
\put(194,75){\circle*{4}}
\put(211,75){\circle*{4}}
\put(203,75){\circle*{4}}
\put(219,75){\circle*{4}}
\put(227,75){\circle*{4}}
\put(257,75){\circle*{4}}
\qbezier(13,99)(13,77)(13,54)
\qbezier(13,99)(30,88)(48,76)
\qbezier(13,54)(30,65)(48,76)
\qbezier(48,76)(65,76)(82,76)
\qbezier(119,75)(144,75)(170,75)
\qbezier(170,75)(170,85)(170,94)
\qbezier(170,127)(161,138)(152,149)
\qbezier(170,127)(180,138)(190,148)
\qbezier(170,75)(182,75)(194,75)
\qbezier(227,75)(242,75)(257,75)
\put(308,98){\circle*{4}}
\put(308,53){\circle*{4}}
\put(355,75){\circle*{4}}
\put(483,74){\circle*{4}}
\put(483,94){\circle*{4}}
\put(483,102){\circle*{4}}
\put(483,110){\circle*{4}}
\put(483,118){\circle*{4}}
\put(483,127){\circle*{4}}
\put(465,149){\circle*{4}}
\put(503,148){\circle*{4}}
\put(492,148){\circle*{4}}
\put(476,148){\circle*{4}}
\put(484,148){\circle*{4}}
\put(398,75){\circle*{4}}
\put(423,74){\circle*{4}}
\put(415,74){\circle*{4}}
\put(431,74){\circle*{4}}
\put(449,74){\circle*{4}}
\qbezier(308,98)(308,76)(308,53)
\qbezier(308,98)(331,87)(355,75)
\qbezier(308,53)(331,64)(355,75)
\qbezier(483,74)(483,84)(483,94)
\qbezier(483,127)(474,138)(465,149)
\qbezier(483,127)(493,138)(503,148)
\put(129,251){\circle*{4}}
\qbezier(147,235)(138,243)(129,251)
\put(405,251){\circle*{4}}
\qbezier(388,234)(396,243)(405,251)
\put(191,92){\circle*{4}}
\qbezier(170,75)(180,84)(191,92)
\put(458,93){\circle*{4}}
\qbezier(483,74)(470,84)(458,93)
\qbezier(355,75)(376,75)(398,75)
\qbezier(449,74)(466,74)(483,74)
\put(8,269){$v_{a}$}
\put(1,203){$v_{a+1}$}
\put(41,245){$v_{a-1}$}
\put(66,223){$v_{a-2}$}
\put(141,221){$v_{g}$}
\put(222,223){$v_{a+2}$}
\put(192,244){$v_{a+3}$}
\put(137,181){$\mathcal {G}_{1}$}
\put(281,244){$v_{a-2}$}
\put(312,222){$v_{a-3}$}
\put(383,221){$v_{g}$}
\put(504,268){$v_{a}$}
\put(501,203){$v_{a+1}$}
\put(455,243){$v_{a-1}$}
\put(435,220){$v_{a+2}$}
\put(379,182){$\mathcal {G}_{2}$}
\put(119,98){\circle*{4}}
\qbezier(119,75)(119,87)(119,98)
\put(4,108){$v_{r_{1}-1}$}
\put(5,41){$v_{r_{1}}$}
\put(42,84){$v_{r_{1}-2}$}
\put(164,62){$v_{g}$}
\put(114,62){$v_{1}$}
\put(239,63){$v_{r_{1}+1}$}
\put(110,108){$v_{\tau_{1}}$}
\put(350,83){$v_{g}$}
\put(302,107){$v_{1}$}
\put(302,40){$v_{2}$}
\put(486,62){$v_{3}$}
\put(443,62){$v_{4}$}
\put(433,95){$v_{\tau_{g}}$}
\put(152,251){$v_{g+1}$}
\put(355,251){$v_{g+1}$}
\put(136,92){$v_{g+1}$}
\put(490,90){$v_{g+1}$}
\put(127,21){$\mathscr{G}_{1}$}
\put(404,22){$\mathscr{X}_{t}$}
\put(177,-9){Fig. 3.5. $\mathcal {G}_{1}$, $\mathcal {G}_{2}$, $\mathscr{G}_{1}$, $\mathscr{X}_{t}$}
\end{picture}
\end{center}

\begin{lemma}\label{le03,03} %------
Let $\mathcal {G}$ be a $\mathcal {F}^{\circ}_{g, l}$-graph with $g\geq 5$ and order $n\geq g+1$, and let $r_{i-1}+1\leq a\leq r_{i}$ where $1\leq i\leq t$. Then

$\mathrm{(i)}$ if $a=r_{i-1}+1$, $r_{i}-r_{i-1}\geq 4$, then $\gamma(\mathcal {G})=\gamma(\mathscr{X}_{i-1})$;

$\mathrm{(ii)}$ if $r_{i-1}+2\leq a\leq r_{i}-3$, then $\gamma(\mathcal {G})\leq\max\{\gamma(\mathcal {G}_{1}), \gamma(\mathcal {G}_{2})\}$;

$\mathrm{(iii)}$ if $r_{i}\geq 4$, $a= r_{i}-1$, then $\gamma(\mathcal {G})=\gamma(\mathscr{G}_{i})$;

$\mathrm{(iv)}$ for other cases of $a$, $\gamma(\mathcal {G})\leq\gamma(\mathcal {G}_{1})$.

\end{lemma}

\begin{proof} Note that $r_{t}= g$ now. Without loss of generality, we suppose that $1\leq a\leq r_{1}$.
By Lemma \ref{le0,08}, there is a minimal dominating set $D$ of $\mathcal {G}$
which contains all $p$-dominators but no any pendant vertex. And further, we can get a minimal dominating set $\mathcal {D}$ of $\mathcal {G}$ from $D$, which not only satisfies that it contains all $p$-dominators but no any pendant vertex, and for all $i=1$, $2$, $\ldots$, $t$, it contains no $v_{r_{i}+1}$ if it is not a $p$-dominator, no $v_{r_{i}-1}$ if it is not a $p$-dominator either. If $D$ already satisfies this, we let $\mathcal {D}=D$. If $D$ does not satisfy that,
then may as well, we suppose that $v_{1}$ is not a $p$-dominator but it is in $D$. We say that $v_{2}\notin D$. Otherwise, if $v_{2}\in D$, then $D\setminus \{v_{1}\}$ is also
a dominating set of $\mathcal {G}$, but its cardinality is less than $D$, which contradicts the choice of $D$. Thus,
let $D_{1}=(D\setminus \{v_{1}\})\cup \{v_{2}\}$, which is also a dominating set of $\mathcal {G}$ with cardinality
$\gamma(\mathcal {G})$. Proceeding like this, we can get such a minimal dominating set $\mathcal {D}$ of $\mathcal {G}$
from $D_{1}$.

If $r_{1}\geq 3$, then we let
$P=v_{2}v_{3}\cdots v_{r_{1}-2}$ (if $r_{1}=1, 2$, then $V(P)=\emptyset$). Here, $\mathcal {D}^{1}=\mathcal {D}\cap V(P)$ is a dominating set of $P$.
Thus, $\parallel\mathcal {D}^{1}\parallel\geq \gamma(P)$. Note that by Lemma \ref{le0,08}, for $P$, there is a
minimal dominating set $\mathcal {R}$ of $P$, which contains no any of $v_{2}$ and $v_{r_{1}-2}$.
Note that $\mathcal {D}\setminus \mathcal {D}^{1}$ is a dominating set of $\mathcal {G}-V(P)$. So, $\mathcal {D}^{'}=\mathcal {R}\cup (\mathcal {D}\setminus \mathcal {D}^{1})$ is also a dominating set of $\mathcal {G}$. Then
$\parallel\mathcal {D}\parallel\geq\parallel\mathcal {D}^{'}\parallel$.
Note that $\parallel\mathcal {D}\parallel=\gamma(\mathcal {G})$. It follows that $\parallel\mathcal {D}\parallel=\parallel\mathcal {D}^{'}\parallel$ and
$\parallel \mathcal {R}\parallel=\parallel\mathcal {D}^{1}\parallel=\gamma(P)$. Let $\mathcal {S}=V(\mathcal {G})\setminus V(P)$, $\mathcal {G}^{\ast}=\mathcal {G}-V(P)$. Like $\parallel\mathcal {D}^{1}\parallel$, we can prove that $\parallel\mathcal {D}\setminus \mathcal {D}^{1}\parallel=\gamma(\mathcal {G}^{\ast})$.
Then combined with Lemma \ref{le0,11}, it follows that
$\parallel\mathcal {D}\parallel=\gamma(P)+\gamma(\mathcal {G}^{\ast})=\lceil\frac{r_{1}-3}{3}\rceil+\gamma(\mathcal {G}^{\ast})$.

To prove the lemma, next, we consider 5 cases as follows.

{\bf Case 1} $a= r_{1}$. For $\mathcal {G}_{1}$, by Lemma \ref{le0,08}, there is a minimal dominating set $D$ of $\mathcal {G}_{1}$
which contains all $p$-dominators but no any pendant vertex. Thus, $v_{r_{1}}\in D$, and both $v_{r_{1}+1}$ and $v_{r_{1}-1}$ can be dominated by $v_{r_{1}}$. Therefore, $D$ is also a dominating set of $\mathcal {G}$, and then $\gamma(\mathcal {G})\leq\gamma(\mathcal {G}_{1})$.

{\bf Case 2} $r_{1}\geq2$, $a= r_{1}-1$.

If $r_{1}=2$, for $\mathcal {G}_{1}$, by Lemma \ref{le0,08}, there is a minimal dominating set $D$ of $\mathcal {G}_{1}$ which contains all $p$-dominators but no any pendant vertex. Thus, $v_{r_{1}}\in D$, $v_{g}\in D$, and $v_{1}$ can be dominated by both $v_{g}$ and $v_{r_{1}}$. Therefore, $D$ is also a dominating set of $\mathcal {G}$, and then $\gamma(\mathcal {G})\leq\gamma(\mathcal {G}_{1})$.

If $r_{1}=3$, for $\mathcal {G}_{1}$, by Lemma \ref{le0,08}, there is a minimal dominating set $D$ of $\mathcal {G}_{1}$ which contains all $p$-dominators but no any pendant vertex. Thus, $v_{r_{1}}\in D$, $v_{g}\in D$, $v_{1}$ can be dominated by $v_{g}$, and $v_{2}$ can be dominated by $v_{r_{1}}$. Therefore, $D$ is also a dominating set of $\mathcal {G}$, and then $\gamma(\mathcal {G})\leq\gamma(\mathcal {G}_{1})$.

Next we consider the case that $r_{1}\geq 4$.

If $r_{2}=r_{1}+1$, we let $S=\{v_{2}$, $v_{3}$, $\ldots$, $v_{r_{1}}$, $v_{\tau_{1}}\}$, $S^{'}=V(\mathcal {G})\setminus S$. As the narrations in the first two paragraphs for $\mathcal {G}$, we get a minimal dominating set $D$ of $\mathcal {G}$ with $D=D_{1}\cup D_{2}$, where $D_{1}$ is a minimal dominating set of $\mathcal {G}[S]$ with cardinality $\gamma(\mathcal {G}[S])=1+\lceil\frac{r_{1}-3}{3}\rceil$, and $D_{2}$ is a minimal dominating set of $\mathcal {G}[S^{'}]$. Thus, we get that $\gamma(\mathcal {G})=1+\lceil\frac{r_{1}-3}{3}\rceil+\gamma(\mathcal {G}[S^{'}])$. Let $B=\{v_{3}$, $\ldots$, $v_{r_{1}-1}$, $v_{r_{1}}\}$, $B^{'}=V(\mathcal {G})\setminus (B\cup\{v_{1}$, $v_{2}$, $v_{\tau_{1}}\})$. Similarly,
for $\mathscr{G}_{1}$, we can get a minimal dominating set $\mathbb{D}$ of $\mathscr{G}_{1}$
with $\mathbb{D}=\mathbb{D}_{1}\cup \{v_{1}\}\cup\mathbb{D}_{2}$, where $\mathbb{D}_{1}\cup \{v_{1}\}$ is a minimal dominating set of $\mathscr{G}_{1}[B\cup\{v_{1}$, $v_{2}$, $v_{\tau_{1}}\}]$, and $\mathbb{D}_{2}$ is a minimal dominating set of $\mathscr{G}_{1}[B^{'}]$ which containing all of its $p$-dominators but no any pendant vertex. Note that any minimal dominating set of $\mathscr{G}_{1}[B^{'}]$ containing all $p$-dominators but no any pendant vertex is also a dominating set of $\mathcal {G}[S^{'}]$, and conversely, any minimal dominating set of $\mathcal {G}[S^{'}]$ containing all $p$-dominators but no any pendant vertex is also a dominating set of $\mathscr{G}_{1}[B^{'}]$. So,  $\gamma(\mathscr{G}_{1}[B^{'}])=\gamma(\mathcal {G}[S^{'}])$. By Lemma \ref{le0,11}, it follows that $\parallel\mathbb{D}_{1}\cup \{v_{1}\}\parallel=1+\lceil\frac{r_{1}-3}{3}\rceil$. Thus $\gamma(\mathcal {G})=\gamma(\mathscr{G}_{1})$.

Similar to the case that $r_{2}=r_{1}+1$, we can prove that if $r_{2}=r_{1}+2$, then $\gamma(\mathcal {G})=\gamma(\mathscr{G}_{1})$; if $r_{2}=r_{1}+3$, then $\gamma(\mathcal {G})+1=\gamma(\mathscr{G}_{1})$.

If $r_{2}\geq r_{1}+4$, we let $S=\{v_{2}$, $v_{3}$, $\ldots$, $v_{r_{2}-2}$, $v_{\tau_{1}}\}$, $S^{'}=V(\mathcal {G})\setminus S$. Similar to the above case that $r_{2}=r_{1}+1$, for $\mathcal {G}$, we get that $\gamma(\mathcal {G})=1+\lceil\frac{r_{1}-3}{3}\rceil+\lceil\frac{r_{2}-r_{1}-3}{3}\rceil+\gamma(\mathcal {G}[S^{'}])$. Let $B=\{v_{1}$, $v_{2}$, $v_{3}$, $\ldots$, $v_{r_{2}-2}$, $v_{\tau_{1}}\}$, $B^{'}=V(\mathcal {G})\setminus B$. Similarly, for $\mathscr{G}_{1}$, we get that $\gamma(\mathscr{G}_{1}[B^{'}])=\gamma(\mathcal {G}[S^{'}])$, and $\gamma(\mathscr{G}_{1})=1+\lceil\frac{r_{1}-3}{3}\rceil+\lceil\frac{r_{2}-r_{1}-2}{3}\rceil+\gamma(\mathcal {G}_{1}[B^{'}])$. It follows that $\gamma(\mathcal {G})\leq\gamma(\mathscr{G}_{1})$.

{\bf Case 3} $a= r_{1}-2$. For $\mathcal {G}_{1}$, we claim that any minimal dominating set $D$ of $\mathcal {G}_{1}$ contains at most one of $v_{r_{1}-2}$ and $v_{r_{1}-1}$. Otherwise,
suppose both $v_{r_{1}-2}$ and $v_{r_{1}-1}$ are in $D$. Then $D\setminus \{v_{r_{1}-1}\}$ is also a dominating
set of $\mathcal {G}_{1}$, but its cardinality is less than that of $D$, which contradicts the choice of
$D$. So, our claim holds.

Similar to Case 2,
we can get a minimal dominating set $\mathbb{D}$ of $\mathcal {G}_{1}$, which contains all $p$-dominators, but no $v_{1}$ and no any pendant vertex. We suppose that $\mathbb{D}$ contains no $v_{r_{1}-1}$. Note that $v_{r_{1}-1}$ can be dominated by $v_{r_{1}}$. Therefore, $\mathbb{D}$ is also a dominating set of $\mathcal {G}$ (if $\mathbb{D}$ contains $v_{r_{1}-1}$, we let $\mathbb{D}^{'}=(\mathbb{D}\setminus \{v_{r_{1}-1}\})\cup \{v_{r_{1}-2}\}$. Then $\mathbb{D}^{'}$ is also a dominating set of $\mathcal {G}$), and then $\gamma(\mathcal {G})\leq\gamma(\mathcal {G}_{1})$.

{\bf Case 4} $2\leq a\leq r_{1}-3$. For $\mathcal {G}_{1}$, let $B=\{v_{2}$, $v_{3}$, $\cdots$, $v_{a}$, $v_{a+1}\}$ and $\mathbb{P}=v_{a+2}v_{a+3}\cdots v_{r_{1}-2}$. Similar to the narrations for $\mathcal {G}$  in the first two paraphs and Case 2, we get
$\gamma(\mathcal {G}_{1})=\gamma(\mathcal {G}[B])+\gamma(\mathbb{P})+\gamma(\mathcal {G}^{\ast})=\lceil\frac{a-1}{3}\rceil+\lceil\frac{r_{1}-a-3}{3}\rceil+\gamma(\mathcal {G}^{\ast})$, where $\mathcal {G}^{\ast}=\mathcal {G}-(B\cup V(\mathbb{P}))$.

For $\mathcal {G}_{2}$, if $4\leq a\leq r_{1}-3$, let $\mathbb{P}=v_{2}v_{3}\cdots v_{a-2}$ and $B=
\{v_{a-1}$, $v_{a}$, $v_{a+1}$, $v_{a+2}$, $\cdots$ $v_{r_{1}-2}\}$. Similarly,
we get that $\gamma(\mathcal {G}_{2})=\gamma(\mathbb{P})+\gamma(\mathcal {G}[B])+\gamma(\mathcal {G}^{\ast})=\lceil\frac{a-3}{3}\rceil+\lceil\frac{r_{1}-a-1}{3}\rceil+\gamma(\mathcal {G}^{\ast})$.

For $\mathcal {G}_{2}$, if $a=2, 3$, let $B=
\{v_{a-1}$, $v_{a}$, $v_{a+1}$, $v_{a+2}$, $\cdots$ $v_{r_{1}-2}\}$, let $B^{'}=V(\mathcal {G}_{2})\setminus B$. As Case 2, we get $\gamma(\mathcal {G}^{\ast})=\gamma(\mathcal {G}_{2}[B^{'}])$, and get that
$\gamma(\mathcal {G}_{2})=\gamma(\mathcal {G}[B])+\gamma(\mathcal {G}^{\ast})=\lceil\frac{r_{1}-1-a}{3}\rceil+\gamma(\mathcal {G}^{\ast})$.

From above discussions, for $4\leq a\leq r_{1}-3$, $$\gamma(\mathcal {G})=\lceil\frac{r_{1}-2-1}{3}\rceil+\gamma(\mathcal {G}^{\ast}),$$
$$\gamma(\mathcal {G}_{1})=\lceil\frac{a-1}{3}\rceil+\lceil\frac{r_{1}-a-3}{3}\rceil+\gamma(\mathcal {G}^{\ast}),\ \gamma(\mathcal {G}_{2})=\lceil\frac{a-3}{3}\rceil+\lceil\frac{r_{1}-a-1}{3}\rceil+\gamma(\mathcal {G}^{\ast}).$$
Note that for positive integers $u\geq 1$, $w\geq 2$, $\lceil\frac{u+w}{3}\rceil\leq \max\{\lceil\frac{u-1}{3}\rceil+\lceil\frac{w}{3}\rceil$,
$\lceil\frac{u+1}{3}\rceil+\lceil\frac{w-2}{3}\rceil\}$. Let $u=r_{1}-2-a$, $w=a-1$. Then we get that $\gamma(\mathcal {G})
\leq\max\{\gamma(\mathcal {G}_{1}), \gamma(\mathcal {G}_{2})\}$.

For $a=3$, $$\gamma(\mathcal {G})=\lceil\frac{r_{1}-3}{3}\rceil+\gamma(\mathcal {G}^{\ast}),\ \gamma(\mathcal {G}_{1})=
\lceil\frac{2}{3}\rceil+\lceil\frac{r_{1}-6}{3}\rceil+\gamma(\mathcal {G}^{\ast}),\ \gamma(\mathcal {G}_{2})=\lceil\frac{r_{1}-4}{3}\rceil+\gamma(\mathcal {G}^{\ast}).$$
As above proof, we let $u=r_{1}-5$, $w=2$. Then we get that $\gamma(\mathcal {G})\leq\max\{\gamma(\mathcal {G}_{1}),
\gamma(\mathcal {G}_{2})\}$.

For $a=2$, noting that $\gamma(\mathcal {G})=\lceil\frac{r_{1}-3}{3}\rceil+
\gamma(\mathcal {G}^{\ast}),$ $$\gamma(\mathcal {G}_{1})=\lceil\frac{1}{3}\rceil+\lceil\frac{r_{1}-5}{3}\rceil+
\gamma(\mathcal {G}^{\ast})=1+\lceil\frac{r_{1}-5}{3}\rceil+\gamma(\mathcal {G}^{\ast}),\ \gamma(\mathcal {G}_{2})=\lceil\frac{r_{1}-3}{3}\rceil+\gamma(\mathcal {G}^{\ast}),$$ and comparing with the results about the above case that $a=3$,
we get that $\gamma(\mathcal {G})\leq\max\{\gamma(\mathcal {G}_{1}), \gamma(\mathcal {G}_{2})\}$.

{\bf Case 5} $a=1$. If $r_{1}=1$, $2$, $3$, then for any minimal dominating set of $\mathcal {G}_{1}$
which contains all $p$-dominators but no any pendant vertex is also a dominating set of $\mathcal {G}$. Combining with Lemma \ref{le0,08} get that $\gamma(\mathcal {G})\leq\gamma(\mathcal {G}_{1})$.

Next, suppose $r_{1}\geq 4$. Let $\mathbb{P}_{1}=v_{2}v_{3}\cdots v_{r_{1}-2}$ and $\mathbb{P}_{2}=
v_{r_{t-1}+2}v_{r_{t-1}+3}\cdots v_{g-2}$, where if $g-3-r_{t-1}\leq 0$, then $V(\mathbb{P}_{2})=\emptyset$. Without loss of generality, we suppose $w=g-3-r_{t-1}\geq 1$. Let $u=r_{1}-3$, $B=V(\mathcal {G})\setminus (V(\mathbb{P}_{1})\cup V(\mathbb{P}_{2}))$, $B^{'}=(B\setminus\{v_{1}$, $v_{g-1}$, $v_{g}\})\cup \{v_{3}$, $v_{4}\}$ if $r_{1}\geq 5$; $B^{'}=(B\setminus\{v_{1}$, $v_{g-1}$, $v_{g}\})\cup \{v_{3}\}$ if $r_{1}= 4$. Let $\mathcal {A}=
V(\mathbb{P}_{2})\cup \{v_{g-1}$, $v_{g}$, $v_{1}$, $v_{2}\}$. For $r_{1}=4, 5, 6$, as Case 2, we can prove that $\gamma(\mathcal {G})=\gamma(\mathbb{P}_{1})+\gamma(\mathbb{P}_{2})+\gamma(\mathcal {G}[B])=\lceil\frac{u}{3}\rceil+\lceil\frac{w}{3}\rceil+\gamma(\mathcal {G}[B])=1+\lceil\frac{w}{3}\rceil+\gamma(\mathcal {G}[B])$, and $\gamma(\mathscr{X}_{t})=\gamma(\mathscr{X}_{t}[\mathcal {A}])+\gamma(\mathscr{X}_{t}[B^{'}])=
\lceil\frac{w+3}{3}\rceil+\gamma(\mathscr{X}_{t}[B^{'}])$. Note that $\gamma(\mathscr{X}_{t}[B^{'}])=\gamma(\mathcal {G}[B])$ and $1+\lceil\frac{w}{3}\rceil=\lceil\frac{w+3}{3}\rceil$. Therefore, $\gamma(\mathcal {G})=\gamma(\mathscr{X}_{t})$. For $r_{1}\geq 7$, we let $\mathbb{P}=v_{5}v_{6}\cdots v_{r_{1}-2}$. In a same way, we can prove that $\gamma(\mathcal {G})=\gamma(\mathscr{X}_{t})$, where $\gamma(\mathcal {G})=\gamma(\mathbb{P}_{1})+\gamma(\mathbb{P}_{2})+\gamma(\mathcal {G}[B])=\lceil\frac{u}{3}\rceil+\lceil\frac{w}{3}\rceil+\gamma(\mathcal {G}[B])$, $\gamma(\mathscr{X}_{t})=\gamma(\mathbb{P})+\gamma(\mathscr{X}_{t}[\mathcal {A}])+\gamma(\mathscr{X}_{t}[B^{'}])=
\lceil\frac{u-3}{3}\rceil+\lceil\frac{w+3}{3}\rceil+\gamma(\mathscr{X}_{t}[B^{'}])$, and $\gamma(\mathscr{X}_{t}[B^{'}])=\gamma(\mathcal {G}[B])$.

From the above proof, it follows that if $1\leq a\leq r_{1}$, the lemma holds.
Similarly, we can prove the lemma holds for any case that $r_{i-1}+1\leq a\leq r_{i}$.
This completes the proof. \ \ \ \ \ $\Box$
\end{proof}

\setlength{\unitlength}{0.7pt}
\begin{center}
\begin{picture}(633,315)
\put(5,290){\circle*{4}}
\put(5,213){\circle*{4}}
\qbezier(5,290)(5,252)(5,213)
\put(26,252){\circle*{4}}
\qbezier(5,290)(15,271)(26,252)
\qbezier(5,213)(15,233)(26,252)
\put(80,252){\circle*{4}}
\put(48,252){\circle*{4}}
\put(363,253){\circle*{4}}
\put(105,252){\circle*{4}}
\put(96,252){\circle*{4}}
\put(88,252){\circle*{4}}
\put(184,252){\circle*{4}}
\put(249,253){\circle*{4}}
\put(355,253){\circle*{4}}
\qbezier(249,253)(302,253)(355,253)
\qbezier(105,252)(144,252)(184,252)
\put(130,252){\circle*{4}}
\put(152,252){\circle*{4}}
\put(152,279){\circle*{4}}
\qbezier(152,252)(152,266)(152,279)
\put(213,252){\circle*{4}}
\qbezier(184,252)(198,252)(213,252)
\put(224,253){\circle*{4}}
\put(233,253){\circle*{4}}
\put(242,253){\circle*{4}}
\put(329,253){\circle*{4}}
\put(301,252){\circle*{4}}
\put(275,253){\circle*{4}}
\put(301,281){\circle*{4}}
\qbezier(301,252)(301,267)(301,281)
\put(373,253){\circle*{4}}
\put(382,253){\circle*{4}}
\put(466,253){\circle*{4}}
\put(439,253){\circle*{4}}
\put(416,253){\circle*{4}}
\put(439,283){\circle*{4}}
\put(390,253){\circle*{4}}
\put(492,253){\circle*{4}}
\qbezier(390,253)(441,253)(492,253)
\qbezier(439,253)(439,268)(439,283)
\put(71,252){\circle*{4}}
\put(594,253){\circle*{4}}
\put(555,253){\circle*{4}}
\put(633,253){\circle*{4}}
\qbezier(555,253)(594,253)(633,253)
\put(502,253){\circle*{4}}
\put(511,253){\circle*{4}}
\put(520,253){\circle*{4}}
\put(528,253){\circle*{4}}
\put(536,253){\circle*{4}}
\put(544,253){\circle*{4}}
\qbezier(26,252)(48,252)(71,252)
\put(39,128){\circle*{4}}
\put(39,51){\circle*{4}}
\put(63,90){\circle*{4}}
\put(113,90){\circle*{4}}
\put(158,90){\circle*{4}}
\put(284,90){\circle*{4}}
\put(571,111){\circle*{4}}
\put(575,105){\circle*{4}}
\put(578,99){\circle*{4}}
\put(229,90){\circle*{4}}
\put(407,90){\circle*{4}}
\put(171,90){\circle*{4}}
\put(180,90){\circle*{4}}
\put(189,90){\circle*{4}}
\put(199,90){\circle*{4}}
\put(209,90){\circle*{4}}
\put(218,90){\circle*{4}}
\qbezier(39,128)(39,90)(39,51)
\qbezier(39,128)(51,109)(63,90)
\qbezier(39,51)(51,71)(63,90)
\qbezier(229,90)(318,90)(407,90)
\qbezier(63,90)(110,90)(158,90)
\put(1,298){$v_{2}$}
\put(8,250){$v_{3}$}
\put(40,240){$v_{4}$}
\put(1,200){$v_{1}$}
\put(149,240){$v_{a_{1}}$}
\put(290,240){$v_{a_{2}}$}
\put(410,240){$v_{a_{k}}$}
\put(330,75){$v_{\varepsilon-k-1}$}
\put(266,76){$v_{\varepsilon-k-2}$}
\put(32,137){$v_{2}$}
\put(64,98){$v_{3}$}
\put(30,39){$v_{1}$}
\put(242,-9){Fig. 3.6. $\mathcal {H}^{k}_{1}$, $\mathcal {H}^{k}_{2}$}
\put(288,176){$\mathcal {H}^{k}_{1}$}
\put(292,19){$\mathcal {H}^{k}_{2}$}
\put(627,239){$v_{\varepsilon}$}
\put(577,239){$v_{\varepsilon-1}$}
\put(343,90){\circle*{4}}
\put(536,90){\circle*{4}}
\put(472,90){\circle*{4}}
\put(584,90){\circle*{4}}
\put(419,90){\circle*{4}}
\put(428,90){\circle*{4}}
\put(437,90){\circle*{4}}
\put(445,90){\circle*{4}}
\put(453,90){\circle*{4}}
\put(461,90){\circle*{4}}
\qbezier(472,90)(528,90)(584,90)
\put(528,74){$v_{\varepsilon-1}$}
\put(582,76){$v_{\varepsilon}$}
\put(606,282){\circle*{4}}
\qbezier(594,253)(600,268)(606,282)
\put(623,261){\circle*{4}}
\put(618,267){\circle*{4}}
\put(613,272){\circle*{4}}
\put(5,213){\circle*{4}}
\qbezier(5,213)(5,252)(5,290)
\qbezier(5,213)(5,252)(5,290)
\qbezier(5,213)(5,252)(5,290)
\qbezier(5,213)(5,252)(5,290)
\put(5,213){\circle*{4}}
\qbezier(5,213)(5,252)(5,290)
\qbezier(5,213)(5,252)(5,290)
\put(568,122){\circle*{4}}
\qbezier(536,90)(552,106)(568,122)
\put(472,119){\circle*{4}}
\qbezier(472,90)(472,105)(472,119)
\put(407,119){\circle*{4}}
\qbezier(407,90)(407,105)(407,119)
\put(343,119){\circle*{4}}
\qbezier(343,90)(343,105)(343,119)
\put(392,75){$v_{\varepsilon-k}$}
\put(461,75){$v_{\varepsilon-2}$}
\end{picture}
\end{center}

Let $\mathcal {H}^{k}_{1}$ be a $\mathcal {F}_{3, \varepsilon-3}$-graph of order $n\geq 4$ where there are $k\geq 0$ $p$-dominators among $v_{1}$, $v_{2}$, $\ldots$, $\varepsilon-2$ ($\varepsilon\geq3$. see Fig. 3.6). If $k\geq 1$, in $\mathcal {H}^{k}_{1}$, suppose $v_{a_{j}}s$ are $p$-dominators where $1\leq j\leq k$, $1\leq a_{1}<a_{2}<\cdots <a_{k}\leq \varepsilon-2$, and suppose $v_{\tau_{j}}$ is the pendant vertex attached to $v_{a_{j}}$. Let $\mathcal {H}^{k}_{2}=\mathcal {H}^{k}_{1}-\sum\limits_{j=1}^{k} v_{\tau_{j}}v_{a_{j}}+\sum\limits_{j=1}^{k} v_{\tau_{j}}v_{\varepsilon-2-k+j}$ (see Fig. 3.6). If $k= 0$, then $\mathcal {H}^{0}_{1}=\mathcal {H}^{0}_{2}$.

\begin{theorem}\label{th03,04} %------
$\gamma(\mathcal {H}^{k}_{1})\leq\gamma(\mathcal {H}^{k}_{2})$.
\end{theorem}

\begin{proof}
For $k= 0$, the theorem holds clearly.
So, we suppose $k\geq 1$ next.

For $\varepsilon=3$, $4$, it is clear that the result holds because $\mathcal {H}^{k}_{1}\cong\mathcal {H}^{k}_{2}$. Next, we suppose that $\varepsilon\geq 5$.

{\bf Case 1} $v_{1}$, $v_{2}$, $v_{3}$ are all $p$-dominators. Then $a_{1}=1$, $a_{2}=2$, $a_{3}=3$.

{\bf Subcase 1.1}  $a_{k}=3$. If $\varepsilon=5$, $6$, by Lemma \ref{le0,08}, it is proved that $\gamma(\mathcal {H}^{k}_{1})=\gamma(\mathcal {H}_{2})$ as Lemma \ref{le03,03}. So, we suppose that $\varepsilon\geq 7$ next. Let $\mathcal {P}=v_{5}v_{6}\cdots v_{\varepsilon-3}$, $S_{1}=\{v_{1}$, $v_{2}$, $v_{3}$, $v_{\tau_{1}}$, $v_{\tau_{2}}$, $v_{\tau_{3}}\}$, $S_{2}=V(\mathcal {H}^{k}_{1})\setminus (S_{1}\cup V(\mathcal {P}))$. As the proof of Lemma \ref{le03,03}, we get that $\gamma(\mathcal {H}^{k}_{1})=\gamma(\mathcal {H}^{k}_{1}[S_{1}])+\gamma(\mathcal {P})+\gamma(\mathcal {H}^{k}_{1}[S_{2}])=3+\lceil\frac{\varepsilon-7}{3}\rceil+1=4+\lceil\frac{\varepsilon-7}{3}\rceil$. In $\mathcal {H}^{k}_{2}$, let $B_{1}=\{v_{1}$, $v_{2}$, $v_{3}$, $\ldots$, $v_{\varepsilon-6}\}$, $B_{2}=V(\mathcal {H}^{k}_{2})\setminus B_{1}$. Similarly, we get $\gamma(\mathcal {H}^{k}_{2})=\gamma(\mathcal {H}^{k}_{2}[B_{1}])+\gamma(\mathcal {H}^{k}_{2}[B_{2}])=\lceil\frac{\varepsilon-7}{3}\rceil+4$. It follows that
$\gamma(\mathcal {H}^{k}_{1})= \gamma(\mathcal {H}^{k}_{2})$ immediately.

{\bf Subcase 1.2}  $a_{k}\geq 4$.

From $4$ to $k$, suppose $a_{z}$ is the largest one such that $a_{z}< \varepsilon-2-k+z$. Without loss of generality, we suppose $a_{k}<\varepsilon-2$. Let $S_{1}=\{v_{1}$, $v_{2}$, $v_{3}$, $\ldots$, $v_{a_{k-1}+1}\}$, $S_{2}=\{v_{\tau_{1}}$, $v_{\tau_{2}}$, $\ldots$, $v_{\tau_{k-1}}\}$, $\mathcal {P}_{1}=v_{a_{k-1}+2}v_{a_{k-1}+3}\ldots v_{a_{k}-2}$ if $a_{k-1}+2\leq a_{k}-2$ (if $a_{k-1}+2> a_{k}-2$, then $V(\mathcal {P}_{1})=\emptyset$), $\mathcal {P}_{2}=v_{a_{k}+2}v_{a_{k}+3}\ldots v_{\varepsilon-3}$ if $a_{k}+2\leq \varepsilon-3$ (if $a_{k}+2> \varepsilon-3$, then $V(\mathcal {P}_{2})=\emptyset$), $B=S_{1}\cup S_{2}$. Let $G_{1}=\mathcal {H}^{k}_{1}[B]$, $G_{2}=\mathcal {H}^{k}_{1}[N_{\mathcal {H}^{k}_{1}}[v_{a_{k}}]]$, $G_{3}=\mathcal {H}^{k}_{1}[N_{\mathcal {H}^{k}_{1}}[v_{\varepsilon-1}]]$.
As the proof of Lemma \ref{le03,03}, we get that $\gamma(\mathcal {H}^{k}_{1})=\gamma(G_{1})+\gamma(\mathcal {P}_{1})+\gamma(G_{2})+\gamma(\mathcal {P}_{2})+\gamma(G_{3})$
where if $V(\mathcal {P}_{1})=\emptyset$, then $\gamma(\mathcal {P}_{1})=0$, and if $V(\mathcal {P}_{2})=\emptyset$, then $\gamma(\mathcal {P}_{2})=0$. Let $K=\mathcal {H}^{k}_{1}- v_{a_{k}}v_{\tau_{k}}+v_{\varepsilon-2}v_{\tau_{k}}$. Let $\mathcal {P}=v_{a_{k-1}+2}v_{a_{k-1}+3}\ldots v_{\varepsilon-4}$, $\mathcal {R}=N_{K}[v_{\varepsilon-1}]\cup N_{K}[v_{\varepsilon-2}]$. Similarly, for $K$, we get that $\gamma(K)=\gamma(G_{1})+\gamma(\mathcal {P})+\gamma(K[\mathcal {R}])$.

If one of $V(\mathcal {P}_{1})$, $V(\mathcal {P}_{2})$ is empty, then $\max \{|V(\mathcal {P}_{1})|$, $|V(\mathcal {P}_{2})|\}\leq |V(\mathcal {P})|$. Note that $\gamma(G_{2})+\gamma(G_{3})=2=\gamma(K[\mathcal {R}])$, $\gamma(\mathcal {P}_{1})+\gamma(\mathcal {P}_{2})\leq \gamma(\mathcal {P})$. So, $\gamma(\mathcal {H}^{k}_{1})\leq \gamma(K)$.

If neither $V(\mathcal {P}_{1})$ nor $V(\mathcal {P}_{2})$ is empty, then $|V(\mathcal {P}_{1})|+|V(\mathcal {P}_{2})|+2= |V(\mathcal {P})|$. Note that $\gamma(G_{2})+\gamma(G_{3})=2=\gamma(K[\mathcal {R}])$, and note that for two nonnegative integer numbers $a$ and $b$, we have $\lceil\frac{a+b+2}{3}\rceil\geq \lceil\frac{a}{3}\rceil+\lceil\frac{b}{3}\rceil$. Then $\gamma(\mathcal {P}_{1})+\gamma(\mathcal {P}_{2})\leq \gamma(\mathcal {P})$. So, $\gamma(\mathcal {H}^{k}_{1})\leq \gamma(K)$.

Proceeding like this, we get graph $\mathcal {H}^{k}_{2}$ from $\mathcal {H}^{k}_{1}$ satisfying that $\gamma(\mathcal {H}^{k}_{1})\leq\gamma(\mathcal {H}^{k}_{2})$.

In a same way, we can prove that $\gamma(\mathcal {H}^{k}_{1})\leq \gamma(\mathcal {H}^{k}_{2})$ for the following cases that:
{\bf Case 2} there are two $p$-dominators among $v_{1}$, $v_{2}$, $v_{3}$;
{\bf Case 3} there is only one $p$-dominator among $v_{1}$, $v_{2}$, $v_{3}$;
{\bf Case 4} there is no $p$-dominator among $v_{1}$, $v_{2}$, $v_{3}$.
\ \ \ \ \ $\Box$
\end{proof}

In $\mathcal {H}^{k}_{2}$, for $j=1$, $2$, $\ldots$, $k$, suppose $v_{\tau_{\varepsilon-2-k+j}}$ is the pendant vertex attached to vertex $v_{\varepsilon-2-k+j}$. Suppose $v_{\omega_{1}}$, $v_{\omega_{2}}$, $\ldots$, $v_{\omega_{s}}$ are the pendant vertices attached to vertex $v_{\varepsilon-1}$. If $s\geq 2$, let $\mathcal {H}^{k}_{3}=\mathcal {H}_{2}^{k}-v_{\varepsilon-1-k}v_{\tau_{\varepsilon-1-k}}+
v_{\varepsilon-1}v_{\tau_{\varepsilon-1-k}}-\sum\limits_{j=2}^{s}v_{\varepsilon-1}v_{\omega_{j}}+
\sum\limits_{j=2}^{s}v_{\omega_{1}}v_{\omega_{j}}$.  Let
$\mathcal {H}^{k-1}_{4}=\mathcal {H}^{k}_{2}-v_{\varepsilon-1-k}v_{\tau_{\varepsilon-1-k}}+v_{\varepsilon-1}v_{\tau_{\varepsilon-1-k}}$, $\mathcal {H}^{k-2}_{5}=\mathcal {H}^{k-1}_{4}-v_{\varepsilon-k}v_{\tau_{\varepsilon-k}}+v_{\varepsilon-1}v_{\tau_{\varepsilon-k}}$.
If $\alpha\geq 1$, we denoted by $\mathscr{H}_{3,\alpha}$ the graph $\mathcal {H}^{\alpha-1}_{2}$ of order $n$ in which there are $\alpha$ $p$-dominators and $v_{\varepsilon-1}$ has only one pendant vertex; if $\alpha= 0$, we let $\mathscr{H}_{3,0}=C_{3}=v_{1}v_{2}v_{3}v_{1}$.

\begin{theorem}\label{th03,05} %------
\

$\mathrm{(i)}$ If $\varepsilon-k-1\leq 2$, then $\gamma(\mathcal {H}^{k}_{2})=k+1$ and $\gamma(\mathcal {H}^{k-1}_{4})= \gamma(\mathcal {H}^{k}_{2})-1$;

$\mathrm{(ii)}$ If $\varepsilon-k-1\geq 3$, then $\gamma(\mathcal {H}^{k}_{2})=\lceil\frac{\varepsilon-k-4}{3}\rceil+k+1$;

$\mathrm{(iii)}$ $\gamma(\mathcal {H}^{k}_{2})\leq \gamma(\mathcal {H}^{k}_{3})$;

$\mathrm{(iv)}$ If $\varepsilon-k-1\geq 3$, $\frac{\varepsilon-k-4}{3}\neq t$ where $t$ is a nonnegative integral number, then $\gamma(\mathcal {H}^{k-1}_{4})= \gamma(\mathcal {H}^{k}_{2})-1$;

$\mathrm{(v)}$ If $\varepsilon-k-1\geq 3$, $\frac{\varepsilon-k-4}{3}= t$ where $t$ is a nonnegative integral number, $\gamma(\mathcal {H}^{k-1}_{4})= \gamma(\mathcal {H}^{k}_{2})$, $\gamma(\mathcal {H}^{k-2}_{5})= \gamma(\mathcal {H}^{k}_{2})-1$.
\end{theorem}

\begin{proof}
As the proof of Lemma \ref{le03,03}, it is proved that (i) and (ii) hold. (iii)--(v) follow from (i) and (ii).
\ \ \ \ \ $\Box$
\end{proof}

\begin{corollary}\label{cl03,06} %------
\

$\mathrm{(i)}$ $\gamma(\mathscr{H}_{3,0})=1$;

$\mathrm{(ii)}$ If $\alpha\geq 1$ and $n-2\alpha\leq 2$, then $\gamma(\mathscr{H}_{3,\alpha})= \alpha$;

$\mathrm{(iii)}$ If $\alpha\geq 1$ and $n-2\alpha\geq 3$, then $\gamma(\mathscr{H}_{3,\alpha})=\lceil\frac{n-2\alpha-2}{3}\rceil+\alpha$.
\end{corollary}

\begin{proof}
It is easy to check that $\gamma(\mathscr{H}_{3,0})=1$. Then (i) follows. If $\alpha\geq 1$, let $k=\alpha- 1$. Then (ii) and (iii) follow from Theorem \ref{th03,05}.
\ \ \ \ \ $\Box$
\end{proof}

Let $G^{\ast}$ be a sunlike graph of order $n$ and with both girth $g$ and $k$ $p$-dominators $v_{1}$, $v_{2}$, $\ldots$, $v_{k}$ on $\mathbb{C}$. Same as Theorems \ref{th03,04} and \ref{th03,05}, we get the following Lemma \ref{le03,07}.

\begin{lemma}\label{le03,07} %------
Let $G$ be a sunlike graph of order $n$ and with both girth $g$ and $k$ $p$-dominators on $\mathbb{C}$. Then $\gamma(G)\leq \gamma(G^{\ast})$, where $\gamma(G^{\ast})=k+\lceil\frac{g-k-2}{3}\rceil$.
\end{lemma}

\begin{lemma}\label{le03,08} %------
Suppose that $\mathcal {G}$ is a $\mathcal {F}_{g, l}$-graph with $\gamma(\mathcal {G}) =\frac{n-1}{2}$, $g\geq 5$ and order $n\geq g+1$, and suppose there are exactly $f$ vertices of the unique cycle $\mathbb{C}$ such that none of them is $p$-dominator. Then we get

$\mathrm{(i)}$ if $f=g$, then $g=5$;

$\mathrm{(ii)}$ if $f\neq g$, then $f\leq 3$ and $f\neq2$;

$\mathrm{(iii)}$ if $f= 3$, then the three vertices are consecutive on $\mathbb{C}$, i.e., they are $v_{i-1}$, $v_{i}$, $v_{i+1}$ for some $1\leq i< g$, and each in $(V(\mathbb{C})\setminus \{v_{i-1}$, $v_{i}$, $v_{i+1}\})\cup V(\mathbb{P}-v_{g+l})$ is $p$-dominator (if $i=1$, then $v_{i-1}=v_{g}$).
\end{lemma}

\begin{proof}
Denote by $A$ the set of vertices of $\mathbb{C}$ and the pendant vertices attached to $\mathbb{C}$. Suppose $\parallel A \parallel=z$. Let $A^{'}=V(\mathcal {G})\setminus A$. Then $\gamma(\mathcal {G})\leq \gamma(\mathcal {G}[A])+\gamma(\mathcal {G}[A^{'}])$. Note that $A^{'}=\emptyset$, or $\mathcal {G}[A^{'}]$ is connected with at least $2$ vertices. Suppose $f\geq 4$.

(i) $f=g$. Then $z-f=0$. This means that there is not $p$-dominator on $\mathbb{C}$. So, $\mathcal {G}[A^{'}]$ is connected with at least $2$ vertices. Thus, if $f\geq 9$, $\gamma(\mathcal {G})\leq \lceil\frac{f}{3}\rceil+\gamma(\mathcal {G}[A^{'}])\leq \frac{n-f}{2}+\frac{f+2}{3}<\frac{n-1}{2}$. Then $f\leq 7$.

If $\gamma(\mathcal {G}[A^{'}])<\frac{n-f}{2}$, then $\gamma(\mathcal {G})<\frac{n-1}{2}$. Hence, it follows that $\gamma(\mathcal {G}[A^{'}])=\frac{n-f}{2}$. Combined with Lemma \ref{le0,10}, it follows that $\mathcal {G}[A^{'}]=P_{\frac{n-f}{2}}\circ K_{1}$. Here, suppose $P_{\frac{n-f}{2}}=v_{a_{1}}v_{a_{2}}\cdots v_{a_{t}}$ with $t=\frac{n-f}{2}$, and suppose $v_{\tau_{1}}$ is the unique pendant vertex attached to $v_{a_{1}}$. By Lemma \ref{le0,08}, $V(P_{\frac{n-f}{2}})$ is a minimal dominating set of $\mathcal {G}[A^{'}]$.

Assume that $f=7$. Note that $\mathcal {G}$ is a $\mathcal {F}_{g, l}$-graph. If $\mathcal {G}=\mathbb{C}+v_{g}v_{a_{1}}+\mathcal {G}[A^{'}]$, then $V(P_{\frac{n-f}{2}})\cup \{v_{2}, v_{5}\}$ is a dominating set of $\mathcal {G}$; if $\mathcal {G}=\mathbb{C}+v_{g}v_{\tau_{1}}+\mathcal {G}[A^{'}]$, then $(V(P_{\frac{n-f}{2}})\backslash \{v_{a_{1}}\})\cup \{v_{2}, v_{5}, v_{\tau_{1}}\}$ is a dominating set of $\mathcal {G}$. This implies that $\gamma(\mathcal {G})\leq \frac{n-7}{2}+2<\frac{n-1}{2}$ which contradicts $\gamma(\mathcal {G}) =\frac{n-1}{2}$. Thus, it follows that $g=5$.

(ii) $f\neq g$. Note that there is no the case that $z-f=1$. Then $z-f\geq2$.
By Lemma \ref{le03,07}, $\gamma(\mathcal {G}[A])\leq \gamma(\mathcal {G}^{\ast}[A])=g-f+\lceil\frac{f-2}{3}\rceil\leq \frac{z-f}{2}+\lceil\frac{f-2}{3}\rceil$, where $\mathcal {G}^{\ast}[A]$ is a sunlike graph with vertex set $A$, $\mathbb{C}$ contained in it and $g-f$ $p$-dominators $v_{1}$, $v_{2}$, $\ldots$, $v_{g-f}$ (defined as $\mathcal {G}^{\ast}$ in Lemma \ref{le03,07}). Thus, if $f\geq 4$, then $\gamma(\mathcal {G})\leq \frac{z-f}{2}+\lceil\frac{f-2}{3}\rceil+\gamma(\mathcal {G}[A^{'}])\leq \frac{n-f}{2}+\lceil\frac{f-2}{3}\rceil\leq \frac{n-f}{2}+\frac{f}{3}<\frac{n-1}{2}$. This contradicts that $\gamma(\mathcal {G}) =\frac{n-1}{2}$. Consequently, $f\leq 3$.

Suppose $f=2$ and suppose that $v_{j}$, $v_{k}$ of $\mathbb{C}$ are the exact $2$ vertices such that neither of them is $p$-dominator. Note that by Lemma \ref{le0,08}, there is a minimal dominating set $D$ of $\mathcal {G}-v_{j}-v_{k}$ which contains all $p$-dominators but no any pendant vertex. Note that the vertices of $\mathbb{C}$ other than $v_{j}$, $v_{k}$ are all $p$-dominators in both $\mathcal {G}-v_{j}-v_{k}$ and $\mathcal {G}$. Thus, each of $v_{j}$, $v_{k}$ is adjacent to at least one $p$-dominator on $\mathbb{C}$. So, $D$ is also a dominating set of $\mathcal {G}$. Note that there is no isolated vertex in $\mathcal {G}-v_{j}-v_{k}$. Then $\gamma(\mathcal {G}-v_{j}-v_{k})\leq\frac{n-2}{2}$, and then $\gamma(\mathcal {G})\leq\frac{n-2}{2}$, which contradicts $\gamma(\mathcal {G}) =\frac{n-1}{2}$. Then (ii) follows.

(iii) Suppose $v_{a}$, $v_{b}$, $v_{c}$ are the exact $3$ vertices of $\mathbb{C}$ such that none of them is $p$-dominator. If the $3$ vertices $v_{a}$, $v_{b}$, $v_{c}$ are not
consecutive, then each of them can be dominated by its adjacent $p$-dominator. Note that by Lemma \ref{le0,08}, there are a minimal dominating set $D$ of $\mathcal {G}-v_{a}-v_{b}-v_{c}$ which contains all $p$-dominators but no any pendant vertex. Thus such $D$ is also a dominating set of $\mathcal {G}$. Note that there is no isolated vertex in $\mathcal {G}-v_{a}-v_{b}-v_{c}$. So, $\gamma(\mathcal {G})\leq \parallel D \parallel=\gamma (\mathcal {G}-v_{a}-v_{b}-v_{c})\leq \frac{n-3}{2}$, which contradicts $\gamma(\mathcal {G}) =\frac{n-1}{2}$. Therefore, the $3$ vertices $v_{a}$, $v_{b}$, $v_{c}$ are
consecutive.

Suppose that the $3$ vertices are $v_{i-1}$, $v_{i}$, $v_{i+1}$ for some $1\leq i\leq g$ (here, if $i=g$, we let $v_{i+1}=v_{1}$; if $i=1$, we let $v_{i-1}=v_{g}$). Note that there is no isolated vertex in $\mathcal {G}-v_{i-1}-v_{i}-v_{i+1}$. Thus, $\gamma(\mathcal {G}-v_{i-1}-v_{i}-v_{i+1})\leq\frac{n-3}{2}$. Next, we claim that $\gamma(\mathcal {G}-v_{i-1}-v_{i}-v_{i+1})=\frac{n-3}{2}$.

{\bf Claim 1} $\gamma(\mathcal {G}-v_{i-1}-v_{i}-v_{i+1})=\frac{n-3}{2}$. Otherwise, suppose $\gamma(\mathcal {G}-v_{i-1}-v_{i}-v_{i+1})<\frac{n-3}{2}$, and suppose $D$ is a minimal dominating set of $\mathcal {G}-v_{i-1}-v_{i}-v_{i+1}$. Then $D\cup\{v_{i}\}$ is a dominating set $D$ of $\mathcal {G}$. Thus, $\mathcal {G}< 1+\frac{n-3}{2}<\frac{n-1}{2}$, which contradicts $\gamma(\mathcal {G}) =\frac{n-1}{2}$. Then the claim holds.

By Lemma \ref{le0,10}, $\mathcal {G}-v_{i-1}-v_{i}-v_{i+1}=H\circ K_{1}$ for some acyclic graph $H$ of order $\frac{n-3}{2}$.

{\bf Claim 2} For any minimal dominating set $D$ of $\mathcal {G}-v_{i-1}-v_{i}-v_{i+1}$, in $\mathcal {G}$, at least one of $v_{i-1}$, $v_{i}$, $v_{i+1}$ can not be dominated by $D$. Otherwise, $D$ is a dominating set of $\mathcal {G}$ too. Hence, $\gamma(\mathcal {G})\leq\frac{n-3}{2}$, which contradicts $\gamma(\mathcal {G}) =\frac{n-1}{2}$. Then the claim holds.

If $i=g$, then $H=H_{1}\cup H_{2}$, $H_{1}=\mathcal {G}[A]-v_{g-1}-v_{g}-v_{1}$, $H_{2}=\mathcal {G}[A^{'}]=P_{\frac{n-z}{2}}\circ K_{1}$ (if $n=z$, then $H_{2}=K_{1}$). Here, suppose $P_{\frac{n-z}{2}}=v_{a_{1}}v_{a_{2}}\cdots v_{a_{t}}$ with $t=\frac{n-z}{2}$, and suppose $v_{\tau_{1}}$ is the unique pendant vertex attached to $v_{a_{1}}$. Then $\mathcal {G}=\mathcal {G}[A]+v_{g}v_{a_{1}}+H_{2}$ or $\mathcal {G}=\mathcal {G}[A]+v_{g}v_{\tau_{1}}+H_{2}$. Note that the vertices in $(\mathbb{C}\setminus \{v_{g-1}, v_{g}, v_{1}\})\cup V(P_{\frac{n-z}{2}})$ are all $p$-dominators.  As (i), we get that $\gamma(\mathcal {G})\leq \frac{n-3}{2}<\frac{n-1}{2}$ which contradicts $\gamma(\mathcal {G}) =\frac{n-1}{2}$. This implies $i\neq g$.

If $i\neq 1, g-1$, then $H$ is connected. Combing Lemma \ref{le0,08}, we get that each in $(V(\mathbb{C})\setminus \{v_{i-1}$, $v_{i}$, $v_{i+1}\})\cup V(\mathbb{P}-v_{g+l})$ is a $p$-dominator,  where $\mathbb{P}=v_{g}v_{g+1}\cdots v_{g+l}$.

If $i= 1$, then $H=H_{1}\cup H_{2}$, $H_{1}=\mathcal {G}[A]-v_{g}-v_{1}-v_{2}$, $H_{2}=\mathcal {G}[A^{'}]=P_{\frac{n-z}{2}}\circ K_{1}$. Here, suppose $P_{\frac{n-z}{2}}=v_{a_{1}}v_{a_{2}}\cdots v_{a_{t}}$ with $t=\frac{n-z}{2}$, and suppose $v_{\tau_{1}}$ is the unique pendant vertex attached to $v_{a_{1}}$. Then $\mathcal {G}=\mathcal {G}[A]+v_{g}v_{a_{1}}+H_{2}$ or $\mathcal {G}=\mathcal {G}[A]+v_{g}v_{\tau_{1}}+H_{2}$. We say that $\mathcal {G}\neq\mathcal {G}[A]+v_{g}v_{\tau_{1}}+H_{2}$. Otherwise, suppose $\mathcal {G}=\mathcal {G}[A]+v_{g}v_{\tau_{1}}+H_{2}$. Note that $n-z$ is even now and $\mathcal {G}-\{v_{2}, v_{1}, v_{g}, v_{a_{1}}, v_{\tau_{1}}\}$ has no isolated vertex. Then for $\mathcal {G}-\{v_{2}, v_{1}, v_{g}, v_{a_{1}}, v_{\tau_{1}}\}$, it has a dominating set $\mathbb{D}$ with $\parallel\mathbb{D}\parallel\leq\frac{n-5}{2}$. Then $\mathbb{D}\cup \{v_{1}, v_{\tau_{1}}\}$ is a dominating set of $\mathcal {G}$, which contradicts $\gamma(\mathcal {G}) =\frac{n-1}{2}$. This implies that $\mathcal {G}=\mathcal {G}[A]+v_{g}v_{a_{1}}+H_{2}$. It follows that each one in $(V(\mathbb{C})\setminus \{v_{g}$, $v_{1}$, $v_{2}\})\cup V(\mathbb{P}-v_{g+l})$ is a $p$-dominator. Similarly, for $i= g-1$, we get that each one in $(V(\mathbb{C})\setminus \{v_{g-2}$, $v_{g-1}$, $v_{g}\})\cup V(\mathbb{P}-v_{g+l})$ is a $p$-dominator.
Then (iii) follows.
 \ \ \ \ \ $\Box$
\end{proof}

\begin{theorem}\label{th03,09} %------
Let $\mathscr{C}=v_{1}v_{2}v_{3}\cdots v_{n-1}v_{n}v_{1}$ be an odd cycle of $n$.

$\mathrm{(i)}$ If $n\equiv 1$ (mod 3), then let $$\mathscr{K}=\mathscr{C}-
v_{\lceil\frac{n}{2}\rceil}v_{\lceil\frac{n+2}{2}\rceil}+
v_{\lceil\frac{n}{2}\rceil}v_{\lfloor\frac{n-2}{2}\rfloor}-v_{\lceil\frac{n+2}{2}\rceil}v_{\lceil\frac{n+4}{2}\rceil}+
v_{\lceil\frac{n+2}{2}\rceil}v_{\lceil\frac{n+8}{2}\rceil}$$ where if $\lceil\frac{n+8}{2}\rceil> n$, then $v_{\lceil\frac{n+8}{2}\rceil}=v_{\lceil k\rceil}$ ($k\equiv \lceil\frac{n+8}{2}\rceil$ (mod n), $1\leq k\leq n-1$. see Fig. 3.7). We have $\gamma(\mathscr{C})=\gamma(\mathscr{K})$.

$\mathrm{(ii)}$ If $n\not\equiv 1$ (mod 3), then let $\mathscr{K}=\mathscr{C}-v_{\lceil\frac{n}{2}\rceil}v_{\lceil\frac{n+2}{2}\rceil}+
v_{\lceil\frac{n}{2}\rceil}v_{\lfloor\frac{n-2}{2}\rfloor}.$ We have $\gamma(\mathscr{C})=\gamma(\mathscr{K})$.
\end{theorem}

\setlength{\unitlength}{0.7pt}
\begin{center}
\begin{picture}(358,313)
\put(6,295){\circle*{4}}
\put(6,218){\circle*{4}}
\qbezier(6,295)(6,257)(6,218)
\put(52,257){\circle*{4}}
\qbezier(6,295)(29,276)(52,257)
\qbezier(6,218)(29,238)(52,257)
\put(190,257){\circle*{4}}
\put(114,257){\circle*{4}}
\put(212,257){\circle*{4}}
\put(201,257){\circle*{4}}
\put(229,257){\circle*{4}}
\put(291,257){\circle*{4}}
\put(173,257){\circle*{4}}
\put(350,257){\circle*{4}}
\put(5,133){\circle*{4}}
\put(5,57){\circle*{4}}
\put(47,95){\circle*{4}}
\put(115,95){\circle*{4}}
\put(212,95){\circle*{4}}
\put(228,95){\circle*{4}}
\put(189,95){\circle*{4}}
\put(201,95){\circle*{4}}
\qbezier(5,133)(5,95)(5,57)
\qbezier(5,133)(26,114)(47,95)
\qbezier(5,57)(26,76)(47,95)
\put(2,304){$v_{\lfloor\frac{n}{2}\rfloor}$}
\put(50,244){$v_{\lfloor\frac{n-2}{2}\rfloor}$}
\put(105,244){$v_{\lfloor\frac{n-4}{2}\rfloor}$}
\put(4,205){$v_{\lceil\frac{n}{2}\rceil}$}
\put(1,145){$v_{\lfloor\frac{n}{2}\rfloor}$}
\put(47,82){$v_{\lfloor\frac{n-2}{2}\rfloor}$}
\put(1,43){$v_{\lceil\frac{n}{2}\rceil}$}
\put(142,-9){Fig. 3.7. $\mathscr{K}$}
\put(172,182){$\mathscr{K}$}
\put(170,18){$\mathscr{K}$}
\put(359,282){$v_{\lceil\frac{n+4}{2}\rceil}$}
\put(338,242){$v_{\lceil\frac{n+6}{2}\rceil}$}
\put(350,95){\circle*{4}}
\put(291,95){\circle*{4}}
\put(357,93){$v_{\lceil\frac{n+2}{2}\rceil}$}
\qbezier(6,218)(6,257)(6,295)
\qbezier(6,218)(6,257)(6,295)
\qbezier(6,218)(6,257)(6,295)
\qbezier(6,218)(6,257)(6,295)
\put(275,80){$v_{\lceil\frac{n+4}{2}\rceil}$}
\put(350,284){\circle*{4}}
\qbezier(350,257)(350,271)(350,284)
\qbezier(291,257)(320,257)(350,257)
\qbezier(291,257)(260,257)(229,257)
\put(291,287){\circle*{4}}
\qbezier(291,257)(291,272)(291,287)
\qbezier(114,257)(143,257)(173,257)
\qbezier(114,257)(83,257)(52,257)
\qbezier(291,95)(320,95)(350,95)
\qbezier(228,95)(259,95)(291,95)
\put(280,244){$v_{\lceil\frac{n+8}{2}\rceil}$}
\put(280,300){$v_{\lceil\frac{n+2}{2}\rceil}$}
\qbezier(47,95)(81,95)(115,95)
\put(174,95){\circle*{4}}
\qbezier(115,95)(144,95)(174,95)
\put(111,82){$v_{\lfloor\frac{n-4}{2}\rfloor}$}
\end{picture}
\end{center}

\begin{proof}
(i) Because $n$ is odd and $n\equiv 1$ (mod 3), then $n\geq 7$. Suppose $n=3k+1$ where $k\geq 2$. By Lemma \ref{le0,09} and Corollary \ref{cl03,06}, it follows that $\gamma(\mathscr{K})=\lceil\frac{n-6}{3}\rceil+2=\lceil\frac{n}{3}\rceil=\gamma(\mathscr{C})$.

(ii) follows from Lemmas \ref{le0,09}, \ref{le0,11}.
\ \ \ \ \ $\Box$
\end{proof}

\section{The $q_{min}$ among uncyclic graphs}

\begin{lemma}\label{le04,01} %------
Let $G$ be a nonbipartite unicyclic graph of order $n$ and with the odd cycle $\mathcal {C}=v_1v_2\cdots v_gv_1$ in it. There is a
unit eigenvector corresponding to $q_{min}(G)$
$X=(\,x_1$, $x_2$, $\ldots$, $x_g$, $x_{g+1}$, $x_{g+2}$, $\ldots$, $x_{n-1}$, $x_{n}\,)^T$, in which suppose $|x_1|=\min\{|x_1|$, $|x_2|$, $\ldots$, $|x_g|\}$,
$|x_s|=\max\{|x_1|$, $|x_2|$, $\ldots$, $|x_g|\}$ where $s\geq 2$, satisfying that

$\mathrm{(i)}$ $|x_1|< |x_s|$;

$\mathrm{(ii)}$ $|x_1|=0$ if and only if $x_{g}=-x_{2}\neq 0$, and for $1\leq i\leq g-1$, if $x_ix_{i+1}\neq 0$, then $x_ix_{i+1}< 0$; moreover, $sgn(x_{j})=(-1)^{d_{H}(v_{1}, v_{j})}$ where $H=G-v_1v_g$.

$\mathrm{(iii)}$ if $|x_1|>0$, then

$\mathrm{(1)}$  if $3\leq s\leq g-1$, then $|x_2|<\cdots<|x_{s-2}|< |x_{s-1}|\leq |x_s|$ and $|x_g|<|x_{g-1}|<\cdots<|x_{s+2}|< |x_{s+1}|\leq |x_s|$;

$\mathrm{(2)}$ if $|x_2|> |x_g|$, then $x_1x_{g}> 0$; for $1\leq i\leq g-1$, $x_ix_{i+1}< 0$; $|x_1|\leq |x_g|$;

$\mathrm{(3)}$ if $|x_2|<|x_g|$,
then $x_1x_{2}> 0$; for $2\leq i\leq g-1$, $x_ix_{i+1}< 0$; $x_gx_{1}< 0$; $|x_1|\leq|x_2|$;

$\mathrm{(4)}$ if $|x_2|=|x_g|$, then $|x_1|\leq|x_2|$, and exactly one of $x_1x_{g}> 0$ and $x_1x_{2}> 0$ holds, where

\ \ \ \ \ $\mathrm{(4.1)}$ if $x_1x_{g}> 0$, then for $1\leq i\leq g-1$, $x_ix_{i+1}< 0$;

\ \ \ \ \ $\mathrm{(4.2)}$ if $x_1x_{2}> 0$, then $x_ix_{i+1}< 0$ for $2\leq i\leq g-1$ and $x_gx_{1}< 0$;

$\mathrm{(5)}$ at least one of  $|x(v_{s+1})|$ and $|x_{s-1}|$ is less than $|x_{s}|$.
\end{lemma}

\begin{proof}
Because $G$ is nonbipartite, it follows that $q_{min}(G)> 0$. Denoted by $T_{v_{1}}$ the tree attached to vertex $v_1$, where $T_{v_{1}}$ is trivial possibly. Denote by $N_{T_{v_{1}}}(v_{1})$ the neighbor set of vertex $v_{1}$ in $T_{v_{1}}$. Suppose
 $N_{T_{v_{1}}}(v_{1})=\{v_{\omega_{1}}$, $v_{\omega_{2}}$, $\ldots$, $v_{\omega_{a}}\}$.

Suppose $|x_1|= 0$. Note that $X$ is an eigenvector. Then $|x_1|< |x_s|$ follows directly. By Lemma \ref{le02,02}, $T_{v_{1}}$ is a zerobranch. If $x_2=0$, then $q_{min}(G)x_1=d_{G}(v_{1})x_1+x_2+x_g+\sum\limits^{a}_{i=1}x_{\omega_{i}}=x_g$. Thus, $x_g=0$. Proceeding on like this, we get that $x_{j}=0$ for any $v_{j}\in V(G)$. So, $X=\mathbf{0}^{T}$. This contradicts $X$ is an eigenvector. Thus, $x_2\neq0$. Similarly, $x_g\neq0$. By $q_{min}(G)x_1=d_{G}(v_{1})x_1+x_2+x_g+\sum\limits^{a}_{i=1}x_{\omega_{i}}=x_2+x_g=0$, $x_2=-x_g$ follows.
Let $H=G-v_1v_g$. If there exists $1\leq t\leq g-1$ such that $x(v_t)x(v_{t+1})> 0$, then we let $Y=(\,y_1$, $y_2$, $\ldots$, $y_g$, $y_{g+1}$, $y_{g+2}$, $\ldots$, $y_{n-1}$, $y_{n}\,)^T$ be a vector satisfying that for any vertex $v_{j}\in V(G)$, $y_{j}=(-1)^{d_{H}(v_{1}, v_{j})}|x_{j}|$. Note that $(x_t+x_{t+1})^{2}> (y_t+y_{t+1})^{2}$. Then $Y^{T}Q(G)Y<X^{T}Q(G)X$, which contradicts that $X$ corresponds to $q_{min}(G)$. This means that for $1\leq i\leq g-1$, if $x(v_i)x(v_{i+1})\neq 0$, then $x_ix_{i+1}< 0$. Therefore, for such defined vector $Y$ as above, we have $Y^{T}Q(G)Y=X^{T}Q(G)X$. Thus, $Y$ is an eigenvector corresponding to $q_{min}(G)$. We can let $X=Y$. Then (ii) follows.

Suppose $|x_1|>0$ next. By Lemmas \ref{le02,02} and
 \ref{le02,03}, we know that for $i=1$, $2$, $\ldots$, $a$, $x_{\omega_{i}}x_1< 0$ and $|x_1|<|x_{\omega_{i}}|$.
 Noting the minimality of $|x_1|$ and $q_{min}(G)> 0$, from $q_{min}(G)x_1=d_{G}(v_1)x_1+x_2+x_g+
 \sum\limits^{a}_{i=1}x_{\omega_{i}}$, we get that one of $x_2$ and $x_g$, say $x_z$ ($z=2$ or $z=g$) for convenience, satisfies that
 $x_zx_1>0$.
As the proof for (ii), we can prove that for any other edge $v_{t}v_{t+1}$ on $\mathcal {C}$ (here, if $t=g$, then we let
$v_{t+1}=v_{1}$)
other than $v_{z}v_{1}$, we have $x_{t}x_{t+1}< 0$.

If $|x_1|= |x_s|$, then $|x_1|= |x_2|=\cdots=|x_g|$. Without loss of generality, suppose that $x_2x_1>0$. Then $x_1x_g<0$ and $x_{g-1}x_g<0$. Denote by $T_{v_{g}}$ the tree attached to vertex $v_{g}$, where $T_{v_{g}}$ is trivial
possibly. Suppose $N_{T_{v_{g}}}(v_{g})=\{v_{o_{1}}$, $v_{o_{2}}$, $\ldots$, $v_{o_{r}}\}$. By Lemmas \ref{le02,02} and \ref{le02,03}, it follows that for $i=1$, $2$, $\ldots$, $r$,
$x_{o_{i}}x_{g}< 0$ and $|x_{g}|<|x_{o_{i}}|$. From $q_{min}(G)x_{g}=d_{G}(v_{g})(x_{g})+
x_{g-1}+x_{1}+\sum^{r}_{i=1}x_{o_{i}}$, we get $q_{min}(G)\leq 0$, which contradicts that $q_{min}(G)> 0$. Therefore, $|x_1|< |x_s|$.

Combining with above narration, we get that $|x_1|< |x_s|$ no matter $|x_1|=0$ or $|x_1|>0$. Then (i) follows. Next, we proceed to prove (iii) with the assumption $|x_1|>0$.

Suppose $4\leq s\leq g-2$. Denote by $T_{v_{s-1}}$ the tree attached to vertex $v_{s-1}$, where $T_{v_{s-1}}$ is trivial
possibly. Suppose $N_{T_{v_{s-1}}}(v_{s-1})=\{v_{\eta_{1}}$, $v_{\eta_{2}}$, $\ldots$, $v_{\eta_{b}}\}$. By Lemmas \ref{le02,02} and \ref{le02,03}, it follows that for $i=1$, $2$, $\ldots$, $b$,
$x_{\eta_{i}}x_{s-1}< 0$ and $|x_{s-1}|<|x_{\eta_{i}}|$. From above discussions, it follows that
$x_{s}x_{s-1}< 0$, $x_{s-1}x_{s-2}< 0$. Noting $q_{min}(G)> 0$, from $q_{min}(G)x_{s-1}=d_{G}(v_{s-1})(x_{s-1})+
x_{s-2}+x_{s}+\sum^{b}_{i=1}x_{\eta_{i}}$, we get $|x_{s-2}|< |x_{s-1}|$. By induction, we get
$|x_1|<|x_2|<\cdots<|x_{s-2}|< |x_{s-1}|\leq |x_s|$. Similarly, we can prove that
$|x_g|<|x_{g-1}|<\cdots<|x_{s+2}|< |x_{s+1}|\leq |x_s|$. For $g=3, 4$, (1) holds clearly. Then (1) follows.

Note that $X$ corresponds to $q_{min}(G)$. If $|x_2|> |x_g|$, because $(|x_1|-|x_2|)^{2}+(|x_1|+|x_g|)^{2}-
(|x_1|+|x_2|)^{2}-(|x_1|-|x_g|)^{2}
=2|x_1|\cdot|x_g|-2|x_1|\cdot|x_2|< 0$, then $x_1x_{g}> 0$. Then (2) follows.

In a same way, it is proved that (3)-(5) hold.
\ \ \ \ \ $\Box$
\end{proof}

\begin{theorem}\label{th04,02} %------
If $\mathcal {G}$ is a nonbipartite $\mathcal {F}^{\circ}_{g, l}$-graph with $g\geq 5$, $n\geq g+1$, then there is a graph $\mathbb{H}$ with girth $3$ and order $n$ such that $\gamma(\mathcal {G})\leq\gamma(\mathbb{H})$ and $q_{min}(\mathbb{H})<q_{min}(\mathcal {G})$.
\end{theorem}

\begin{proof}
Because $\mathcal {G}$ is nonbipartite, $g$ is odd. Because $n\geq g+1$, there are pendant vertices in $\mathcal {G}$. This means $\delta =1$. By Lemma \ref{le0,06}, it follows that $q_{min}(\mathcal {G})<1$. Let $X=(x_1$, $x_2$, $x_3$, $\ldots$, $x_n)^T$ be a unit
eigenvector corresponding to $q_{min}(\mathcal {G})$. Among the vertices of $\mathbb{C}$, suppose $x_{a+1}=\min\{|x(v_1)|$, $|x(v_2)|$, $\ldots$, $|x(v_g)|\}$, and without loss of generality, suppose $1\leq a\leq r_{1}$ and  $x_{a+1}\geq 0$. By Lemma \ref{le04,01}, we see that $x_{a}x_{a+1}\geq 0$ or $x_{a+2}x_{a+1}\geq 0$; without loss of generality, suppose $|x_{a}|\leq |x_{a+2}|$, $x_{a}x_{a+1}\geq 0$. Using Lemma \ref{le04,01} again follows that $x_{a+2}\neq 0$ and $|x_{a+2}|>|x_{a+1}|$. Suppose that $|x_{a-1}|\leq |x_{a+2}|$. Let $H=G-v_av_{a+1}$. By Lemma \ref{le04,01}, suppose that for any $j\neq a, a+1$, $\mathrm{sgn}(x_{j})=(-1)^{d_{H}(v_{a+1}, v_{j})}$. It follows that $x_{a-1}\leq 0$, $x_{a+2}< 0$. In $\mathcal {G}$, denote by $T_{v_{a+2}}$ attached to $v_{a+2}$, and let $S=V(T_{v_{a+2}})\setminus \{v_{a+2}\}$ (Here, if $T_{v_{a+2}}$ is trivial, then $S=\emptyset$). Let $\mathcal {G}_{1}=\mathcal {G}-v_{a+1}v_{a+2}+v_{a+1}v_{a-1}$, $\mathcal {G}_{2}=\mathcal {G}_{1}-v_{a-1}v_{a-2}+v_{a-1}v_{a+2}$ (if $|x_{a-1}|> |x_{a+2}|$, then we let $\mathcal {G}_{1}=\mathcal {G}-v_{a}v_{a-1}+v_{a}v_{a+2}$, $\mathcal {G}_{2}=\mathcal {G}_{1}-v_{a+2}v_{a+3}+v_{a+2}v_{a-1}$).

Firstly, we prove that $q_{min}(\mathcal {G}_{1})<q_{min}(\mathcal {G})$. Note that $(x_{a+1}+x_{a-1})^{2}\leq (x_{a+1}+x_{a+2})^{2}$. Then $X^{T}Q(\mathcal {G}_{1})X\leq X^{T}Q(\mathcal {G})X$. This means that $q_{min}(\mathcal {G}_{1})\leq q_{min}(\mathcal {G})$. If $q_{min}(\mathcal {G}_{1})= q_{min}(\mathcal {G})$, then $X^{T}Q(\mathcal {G}_{1})X= X^{T}Q(\mathcal {G})X$, and $X$ is also an eigenvector of $\mathcal {G}_{1}$. Note that $q_{min}(\mathcal {G}_{1})x_{a+2}=d_{\mathcal {G}_{1}}(v_{a+2})x_{a+2}+x_{a+3}+\sum\limits_{v_{j}\sim v_{a+2}, v_{j}\in S} x_{j}$, $q_{min}(\mathcal {G})x_{a+2}=d_{\mathcal {G}}(v_{a+2})x_{a+2}+x_{a+1}+x_{a+3}+\sum\limits_{v_{j}\sim v_{a+2}, v_{j}\in S} x_{j}$, $d_{\mathcal {G}_{1}}(v_{a+2})=d_{\mathcal {G}}(v_{a+2})-1$ and $|x_{a+1}|< |x_{a+2}|$. If $q_{min}(\mathcal {G}_{1})x_{a+2}= q_{min}(\mathcal {G})x_{a+2}$, then $|x_{a+1}|= |x_{a+2}|$. This contradicts $|x_{a+2}|>|x_{a+1}|$. This implies that  $q_{min}(\mathcal {G}_{1})x_{a+2}\neq q_{min}(\mathcal {G})x_{a+2}$. Then it follows that $q_{min}(\mathcal {G}_{1})<q_{min}(\mathcal {G})$.

Secondly, we prove that $q_{min}(\mathcal {G}_{2})<q_{min}(\mathcal {G})$. In $\mathcal {G}_{2}-v_{a-1}$, Denote by $G_{1}$ the component containing $v_{a}$ and denote by $G_{2}$ the component containing $v_{a+2}$. Let vector $Y$ satisfy  $$\left\{\begin{array}{cc}
                                y_{i}=x_{i},  &  v_{i}\in V(G_{1}); \\
                                y_{a-1}=x_{a-1}; \\
                                y_{i}=-x_{i}, & v_{i}\in V(G_{2}).
                              \end{array}\right.$$
Note that $|x_{a+1}|\leq |x_{a-1}|\leq |x_{a+2}|$, $y_{a+1}y_{a-1}\leq 0$, $y_{a+2}y_{a-1}\leq 0$. Then $(y_{a+1}+y_{a-1})^{2}+(y_{a+2}+y_{a-1})^{2}\leq (x_{a+1}+x_{a+2})^{2}$. So, $Y^{T}Q(\mathcal {G}_{2})Y\leq X^{T}Q(\mathcal {G})X$, and $q_{min}(\mathcal {G}_{2})\leq q_{min}(\mathcal {G})$. If $q_{min}(\mathcal {G}_{2})= q_{min}(\mathcal {G})$, then $Y^{T}Q(\mathcal {G}_{2})Y= X^{T}Q(\mathcal {G})X$, and $Y$ is an eigenvector of $\mathcal {G}_{2}$. Note that $q_{min}(\mathcal {G}_{2})y_{a+2}=-q_{min}(\mathcal {G}_{2})x_{a+2}=-(d_{\mathcal {G}_{2}}(v_{a+2})x_{a+2}-x_{a-1}+x_{a+3}+\sum\limits_{v_{j}\sim v_{a+2}, v_{j}\in S} x_{j})$, $q_{min}(\mathcal {G})x_{a+2}=d_{\mathcal {G}}(v_{a+2})x_{a+2}+x_{a+1}+x_{a+3}+\sum\limits_{v_{j}\sim v_{a+2}, v_{j}\in S} x_{j}$, and $d_{\mathcal {G}_{2}}(v_{a+2})=d_{\mathcal {G}}(v_{a+2})$. If $q_{min}(\mathcal {G}_{2})= q_{min}(\mathcal {G})$, then $q_{min}(\mathcal {G}_{2})x_{a+2}=q_{min}(\mathcal {G})x_{a+2}$, and then
$-x_{a-1}=x_{a+1}$. This means that $y_{a+1}=-y_{a-1}$. Note that $Y$ is an eigenvector of $\mathcal {G}_{2}$ now. If $x_{a+1}=-x_{a-1}=0$, then $y_{a+1}=y_{a-1}=0$, which contradicts Lemma \ref{le04,01}. So, $x_{a+1}\neq 0$ now. If $|x_{a}|>|x_{a+1}|$, then by Lemma \ref{le04,01}, it follows that $y_{a+1}y_{a-1}>0$, which contradicts $y_{a+1}=-y_{a-1}$; if $|x_{a}|=|x_{a+1}|$, then $|y_{a}|=|y_{a+1}|=|y_{a-1}|$ which contradicts (1) of (iii) in Lemma \ref{le04,01}. Therefore, $x_{a+1}\neq -x_{a-1}$ (in fact, $|x_{a+1}|<|x_{a-1}|$), $q_{min}(\mathcal {G}_{2})x_{a+2}\neq q_{min}(\mathcal {G})x_{a+2}$, and then $q_{min}(\mathcal {G}_{2})\neq q_{min}(\mathcal {G})$. So, $q_{min}(\mathcal {G}_{2})< q_{min}(\mathcal {G})$.

Thirdly, noting that $x_{a+2}\neq 0$ and $|x_{a+2}|>|x_{a+1}|$, as proven that $q_{min}(\mathcal {G}_{1})<q_{min}(\mathcal {G})$ above, we can prove a case that if $r_{1}\geq 4$, $a+1= r_{1}$, then $q_{min}(\mathscr{G}_{1})<q_{min}(\mathcal {G})$ where $\mathscr{G}_{1}$ is expressed in section 3.

Finally, if $a=1$ and $r_{1}\geq 4$, noting that $x_{3}\neq 0$, as above proof, we can prove that then $q_{min}(\mathscr{X}_{t})<q_{min}(\mathcal {G})$ where $\mathscr{X}_{t}$ is expressed in section 3.

For the above cases, by Lemma \ref{le03,03}, we get that $\gamma(\mathcal {G})\leq\gamma(\mathcal {G}_{1})$, $\gamma(\mathcal {G})\leq\gamma(\mathcal {G}_{2})$, $\gamma(\mathcal {G})\leq\gamma(\mathscr{G}_{1})$ if $r_{1}\geq 4$, $a+1= r_{1}$, and $\gamma(\mathcal {G})=\gamma(\mathscr{X}_{t})$ if $a=1$ and $r_{1}\geq 4$, respectively.

Similarly, for other cases, we get that there is a graph $\mathbb{H}$ with girth $3$ such that $\gamma(\mathcal {G})\leq\gamma(\mathbb{H})$ and $q_{min}(\mathbb{H})<q_{min}(\mathcal {G})$. This completes the proof. \ \ \ \ \ $\Box$
\end{proof}

\begin{lemma}\label{le04,03} %------
Suppose that $G$ is a nonbipartite $\mathcal {F}_{3, l}$-graph of order $n$ where $\mathbb{C}=v_{1}v_{2}v_{3}v_{1}$. $X=(\,x_{1}$, $x_{2}$, $\ldots$, $x_{n}\,)^T$ is a unit eigenvector corresponding to $q_{min}(G)$. Then $|x_{3}|=\max\{|x_{1}|$, $|x_{2}|$, $|x_{3}|\}$.
\end{lemma}

\begin{proof}
Because $G$ is nonbipartite, $g$ is odd. Without loss of generality, suppose $v_{2}$ is a $p$-dominator and $v_{2}v_{u}$ is the only pendant edge. If $|x_{2}|>|x_{3}|$. Let $\mathcal {H}=G-\sum\limits_{v_{j}\sim v_{3}, j\neq 1, 2} v_{3}v_{j}+\sum\limits_{v_{j}\sim v_{3}, j\neq 1, 2} v_{2}v_{j}-v_{2}v_{u}+v_{3}v_{u}$.
Let $a=|x_{2}|-|x_{3}|$, let $S$ be the set of vertices $v_{1}$, $v_{2}$, $v_{3}$ and the pendant vertices attached to $v_{1}$, $v_{2}$ in $G$. Let $Y=(\,y_{1}$, $y_{2}$, $\ldots$, $y_{n-1}$, $y_{n}\,)^T$ satisfy that $$\left\{\begin{array}{cc}
                                y_{i}=x_{i},  & v_{i}\in (S\setminus \{v_{u}\}); \\
                                y_{u}=-sgn(x_{3})(|x_{u}|-a),  & v_{u}\\
                                y_{i}=(-)^{d_{\mathcal {H}}(v_{i}, v_{3})}sgn(x_{2})(|x_{i}|+a), & others.
                              \end{array}\right.$$
By  Theorem \ref{th03,02}, we know that $|x_{l+2}|\geq|x_{2}|>0$. Note that both $v_{l+3}$ and $v_{u}$ are pendant vertices. From $x_{l+3}=\displaystyle\frac{x_{l+2}}{q_{min}(G)-1}$ and $x_{u}=\displaystyle\frac{x_{2}}{q_{min}(G)-1}$, it follows that $|x_{l+3}|\geq|x_{u}|$ and $|y_{l+3}|^{2}-|x_{l+3}|^{2}=(|x_{l+3}|+a)^{2}-|x_{l+3}|^{2}\geq|x_{u}|^{2}-|y_{u}|^{2}=|x_{u}|^{2}-(|x_{u}|-a)^{2}$. This leads to that $Y^{T}Y\geq X^{T}X$. Note that $Y^{T}Q(\mathcal {H})Y=X^{T}Q(G)X$. This means that $q_{min}(\mathcal {H})\leq q_{min}(G)$. If $q_{min}(\mathcal {H})= q_{min}(G)$, then $Y$ is an eigenvalue corresponding to $q_{min}(\mathcal {H})$. Combined with Lemmas \ref{le02,02}, \ref{le02,03}, it follows that $y_{l+3}y_{l+2}<0$ and $|y_{l+3}|>|y_{l+2}|>0$. Noting that $|y_{l+3}|=|x_{l+3}|+a$ and $|y_{l+2}|=|x_{l+2}|+a$, combining with Lemmas \ref{le02,02}, \ref{le02,03} again, we get that $x_{l+3}x_{l+2}<0$ and $|x_{l+3}|>|x_{l+2}|>0$. From $q_{min}(\mathcal {H})y_{l+3}=y_{l+3}+y_{l+2}$, it follows that $q_{min}(\mathcal {H})=\frac{|x_{l+3}|-|x_{l+2}|}{|x_{l+3}|+a}<\frac{|x_{l+3}|-|x_{l+2}|}{|x_{l+3}|}=q_{min}(G)$, which contradicts $q_{min}(\mathcal {H})= q_{min}(G)$. Consequently, $q_{min}(\mathcal {H})< q_{min}(G)$. It is a contradiction because $\mathcal {H}\cong G$. It follows that $|x_{2}|\leq|x_{3}|$. Similarly, it is proved that $|x_{1}|\leq|x_{3}|$. The result follows.
\ \ \ \ \ $\Box$
\end{proof}

\begin{theorem}\label{th04,04} %------
Among all nonbipartite unicyclic graphs of order $n\geq 4$ with girth $3$ and domination number at least $\frac{n+1}{3}<\gamma\leq \frac{n}{2}$, we have

$\mathrm{(i)}$ if $n=4$, the $q_{min}$ attains the minimum uniquely at $\mathscr{H}_{3,1}$;

$\mathrm{(ii)}$ if $n\geq 5$, $\gamma=\frac{n-1}{2}$, then the $q_{min}$ attains the minimum uniquely at $\mathscr{H}_{3,\frac{n-3}{2}}$;

$\mathrm{(iii)}$ if $n\geq 6$, $\gamma=\frac{n}{2}$, then the $q_{min}$ attains the minimum uniquely at $\mathscr{H}_{3,\frac{n}{2}}$.

$\mathrm{(iv)}$ if $n\geq 5$ and $n-2\gamma\geq 2$, then the $q_{min}$ attains the minimum uniquely at a $\mathscr{H}_{3,\alpha}$ where $\alpha\leq\frac{n-3}{2}$ is the least integer such that $\lceil\frac{n-2\alpha-2}{3}\rceil+\alpha=\gamma$;

\end{theorem}

\begin{proof}
If $n=4$, then the graph is isomorphic to $S^{+}_{4}=\mathscr{H}_{3,1}$, and then $\mathrm{(ii)}$ follows.
For $n\geq 5$, the $\mathrm{(ii)}$, $\mathrm{(iii)}$ and $\mathrm{(iv)}$ follow from Lemmas \ref{le0,04}, \ref{le0,10}, Theorems \ref{th03,02}, \ref{th03,04}, \ref{th03,05}, Corollary \ref{cl03,06} and Lemma \ref{le04,03}.
 \ \ \ \ \ $\Box$
\end{proof}

Let $\mathcal {K}=\{G|\,\ G\ \mathrm{be}\ \mathrm{a}\ \mathrm{nonbipartite}\ \mathcal {F}^{\circ}_{g, l}$-graph of order $n\geq 4$ and domination number at least $\frac{n+1}{3}<\gamma\leq \frac{n}{2}$, where $g$ is any odd number at least $3$ and $l$ is any positive integral number$\}$ and $q_{\mathcal {K}}=\min\{q_{min}(G)|\,\ G\in \mathcal {K}\}$.

\begin{theorem}\label{th04,05} %------
\

$\mathrm{(i)}$ If $n=4$, the $q_{\mathcal {K}}$ attains uniquely at $\mathscr{H}_{3,1}$;

$\mathrm{(ii)}$ If $n\geq 5$ and $n-2\gamma\geq 2$, then the least $q_{\mathcal {K}}> q_{min}(\mathscr{H}_{3,\alpha})$ where $\alpha\leq\frac{n-3}{2}$ is the least integer such that $\lceil\frac{n-2\alpha-2}{3}\rceil+\alpha=\gamma$.
\end{theorem}

\begin{proof}
As Theorem \ref{th04,04}, the theorem holds for $n=4$.
For $n\geq 5$, $\mathrm{(ii)}$  follows Theorems \ref{th04,02}, \ref{th04,04}. \ \ \ \ \ $\Box$
\end{proof}

\begin{lemma}\label{le04,09} %------
For  a nonbipartite $\mathcal {F}_{g, l}$-graph graph $G$ of order $n\geq 6$, if $g =5$, then there exists a graph $\mathbb{H}$ such that $g(\mathbb{H})=3$, $\gamma(G)\leq \gamma(\mathbb{H})$ and $q_{min}(\mathbb{H})< q_{min}(G)$.
\end{lemma}

\begin{proof}
Because $G$ is nonbipartite, $g$ is odd. Note that $\delta(G)=1$. By Lemma \ref{le0,06}, we get that $q_{min}(G)<1$.

{\bf Case 1} There is no $p$-dominator on $\mathbb{C}$. Then $G$ is like $G_{1}$ (see $G_{1}$ in Fig. 4.1). Suppose $X=(\,x_1$, $x_2$, $\ldots$, $x_k$, $x_{k+1}$, $x_{k+2}$, $\ldots$, $x_{n-1}$, $x_{n}\,)^T$ is a unit eigenvector corresponding to $q_{min}(G)$. Note that by Lemmas \ref{le0,05}, \ref{th03,02}, it follows that $|x_5|=\max\{|x_1|$, $|x_2|$, $|x_3|$, $|x_4|$, $|x_5|\}> 0$. Let $Y=(\,y_1$, $y_2$, $\ldots$, $y_k$, $y_{k+1}$, $y_{k+2}$, $\ldots$, $y_{n-1}$, $y_{n}\,)^T$ satisfy that $$\left\{\begin{array}{cc}
                                y_{1}=x_{4},  \\
                                y_{2}=x_{3},  \\
                                y_{4}=x_{1},  \\
                                y_{3}=x_{2},  \\
                                y_{i}=x_{i}, & \ \ others.
                              \end{array}\right.$$
Let $Z=X+Y=(\,z_1$, $z_2$, $\ldots$, $z_k$, $z_{k+1}$, $z_{k+2}$, $\ldots$, $z_{n-1}$, $z_{n}\,)^T$. Then $Z=X+Y$ is an eigenvector corresponding to $q_{min}(G)$ that $z_{1}=z_{4}$, $z_{2}=z_{3}$. Also by Lemmas \ref{le0,05}, \ref{th03,02}, it follows that $|z_5|=\max\{|z_1|$, $|z_2|$, $|z_3|$, $|z_4|$, $|z_5|\}> 0$. By Lemma \ref{le04,01}, we get that $|z_2|>0$, $|z_2|< |z_1|$ and $z_2 z_1< 0$. Let $\mathbb{H}=G-v_{3}v_{4}+v_{3}v_{1}$. Then $Z^{T}Q(\mathbb{H})Z\leq Z^{T}Q(G)Z$. This means $q_{min}(\mathbb{H})\leq q_{min}(G)$. As Theorem \ref{th04,02}, we get that $q_{min}(\mathbb{H})< q_{min}(G)$. Let $B_{1}=\mathbb{H}[v_{1}$, $v_{2}$, $v_{3}]$, $B_{2}=\mathbb{H}- \{v_{1}$, $v_{2}$, $v_{3}\}$. As Lemma \ref{le03,03}, we can get a minimal dominating set $D$ of $\mathbb{H}$
which contains all $p$-dominators but no any pendant vertex and no $v_{1}$ that $D=\{v_{2}\}\cup D_{2}$, where $\{v_{2}\}$ is a dominating set of $B_{1}$, $D_{2}$ is a dominating set of $B_{2}$. Note that $D$ is also a dominating set of $G$. So, $\gamma(G)\leq \gamma(\mathbb{H})$.

{\bf Case 2} There is only $1$ $p$-dominator on $\mathbb{C}$ (see $G_{2}-G_{4}$ in Fig. 4.1).

{\bf Subcase 2.1} For $G_{2}$, let $\mathbb{H}=G_{2}-v_{3}v_{4}+v_{3}v_{1}$. As Case 1, it is proved that $\gamma(G_{2})\leq \gamma(\mathbb{H})$ and $q_{min}(\mathbb{H})< q_{min}(G_{2})$.

{\bf Subcase 2.2} For $G_{3}$, suppose $X=(\,x_1$, $x_2$, $\ldots$, $x_{n-1}$, $x_{n}\,)^T$ is a unit eigenvector corresponding to $q_{min}(G_{3})$.

{\bf Claim} $|x_{4}|>|x_{1}|$, $|x_{5}|>|x_{3}|$. Denote $v_{k}$ the pendant vertex attached to $v_{4}$. Suppose $0<|x_{4}|\leq|x_{1}|$. Let $G^{'}_{3}=G_{3}-v_{4}v_{k}+v_{1}v_{k}$. By Lemma \ref{le0,04}, then $q_{min}(G^{'}_{3})<q_{min}(G_{3})$. This is a contradiction because $G^{'}_{3}\cong G_{3}$. Suppose $|x_{4}|=|x_{1}|=0$. By Lemma \ref{le04,01}, we get that $x_{2}\neq 0$, $x_{3}\neq 0$. By $q_{min}(G_{3})x_{2}=2x_{2}+x_{3}$, $q_{min}(G_{3})x_{3}=2x_{3}+x_{2}$, we get $x^{2}_{2}=x^{2}_{3}$. Suppose $x_{2}>0$. Then we get $q_{min}(G_{3})x_{2}=2x_{2}+x_{3}\geq x_{2}$. This means that $q_{min}(G_{3})\geq 1$ which contradicts $q_{min}(G_{3})<1$. Thus, $|x_{4}|>|x_{1}|$. Similarly, we get $|x_{5}|>|x_{3}|$. Then the claim holds.

Suppose $|x_{1}|=\min\{|x_{1}|$, $|x_{2}|$, $|x_{3}|\}$ and $x_{1}\geq 0$. If $|x_{2}|> |x_{5}|$, by Lemma \ref{le04,01}, suppose $x_1x_5\geq 0$. Let $H=G_{3}-v_{1}v_{5}$. Also by Lemma \ref{le04,01}, suppose for any $j\neq 1, 5$, $\mathrm{sgn}x_j= (-1)^{d_{H}(v_{j}, v_1)}$. Let  $\mathbb{H}=G_{3}-v_{1}v_{5}+v_{3}v_{1}$. Because $|x_{5}|>|x_{3}|$, it follows that $q_{min}(\mathbb{H})\leq X^{T}Q(\mathbb{H})X<X^{T}Q(G_{3})X =q_{min}(G_{3})$. Let $B_{1}=\mathbb{H}[v_{1}, v_{2}]$, $B_{2}=\mathbb{H}- \{v_{1}, v_{2}\}$. As Lemma \ref{le03,03}, we can get a minimal dominating set $D$ of $\mathbb{H}$ which contains all $p$-dominators but no any pendant vertex and no $v_{3}$ that $D=\{v_{1}\}\cup D_{2}$, where $D_{2}$ is a dominating set of $B_{2}$. Note that $D$ is also a dominating set of $G_{3}$. So, $\gamma(G_{3})\leq \gamma(\mathbb{H})$. If $|x_{2}|< |x_{5}|$, by Lemma \ref{le04,01}, $x_{1}x_{2}\geq 0$. Let $H=G_{3}-v_{1}v_{2}$. Also by Lemma \ref{le04,01}, suppose for any $j\neq 1, 2$, $\mathrm{sgn}x_j= (-1)^{d_{H}(v_{j}, v_1)}$. Let  $\mathbb{H}=G_{3}-v_{1}v_{5}+v_{3}v_{1}$. Because $|x_{5}|>|x_{3}|$, it follows that $q_{min}(\mathbb{H})< q_{min}(G_{3})$ similarly. As $|x_{2}|> |x_{5}|$, it is proved that $\gamma(G_{3})\leq \gamma(\mathbb{H})$. If $|x_{2}|= |x_{5}|$, by Lemma \ref{le04,01}, suppose $x_{1}x_{5}\geq 0$. Let  $\mathbb{H}=G_{3}-v_{1}v_{5}+v_{3}v_{1}$. As $|x_{2}|> |x_{5}|$, it is proved that $q_{min}(\mathbb{H})< q_{min}(G_{3})$, $\gamma(G_{3})\leq \gamma(\mathbb{H})$.

\setlength{\unitlength}{0.7pt}
\begin{center}
\begin{picture}(636,599)
\put(38,577){\circle*{4}}
\put(19,541){\circle*{4}}
\qbezier(38,577)(28,559)(19,541)
\put(79,546){\circle*{4}}
\qbezier(38,577)(58,562)(79,546)
\put(19,507){\circle*{4}}
\qbezier(19,541)(19,524)(19,507)
\put(79,507){\circle*{4}}
\qbezier(79,546)(79,527)(79,507)
\qbezier(19,507)(49,507)(79,507)
\put(86,577){\circle*{4}}
\qbezier(38,577)(62,577)(86,577)
\put(127,577){\circle*{4}}
\qbezier(86,577)(106,577)(127,577)
\put(86,598){\circle*{4}}
\qbezier(86,577)(86,588)(86,598)
\put(127,598){\circle*{4}}
\qbezier(127,577)(127,588)(127,598)
\put(1,540){$v_{1}$}
\put(3,500){$v_{2}$}
\put(85,502){$v_{3}$}
\put(85,543){$v_{4}$}
\put(20,582){$v_{5}$}
\put(78,565){$v_{5+s}$}
\put(36,473){$G_{1}$}
\put(253,577){\circle*{4}}
\put(270,544){\circle*{4}}
\qbezier(253,577)(261,561)(270,544)
\put(209,541){\circle*{4}}
\qbezier(253,577)(231,559)(209,541)
\put(209,507){\circle*{4}}
\qbezier(209,541)(209,524)(209,507)
\put(270,507){\circle*{4}}
\qbezier(270,544)(270,526)(270,507)
\qbezier(209,507)(239,507)(270,507)
\put(163,577){\circle*{4}}
\qbezier(253,577)(208,577)(163,577)
\put(163,598){\circle*{4}}
\qbezier(163,577)(163,588)(163,598)
\put(205,577){\circle*{4}}
\put(205,598){\circle*{4}}
\qbezier(205,577)(205,588)(205,598)
\put(253,598){\circle*{4}}
\qbezier(253,577)(253,588)(253,598)
\put(193,541){$v_{1}$}
\put(190,504){$v_{2}$}
\put(275,504){$v_{3}$}
\put(275,544){$v_{4}$}
\put(258,576){$v_{5}$}
\put(230,473){$G_{2}$}
\put(363,577){\circle*{4}}
\put(349,549){\circle*{4}}
\qbezier(363,577)(356,563)(349,549)
\put(405,549){\circle*{4}}
\qbezier(363,577)(384,563)(405,549)
\put(349,508){\circle*{4}}
\qbezier(349,549)(349,529)(349,508)
\put(405,508){\circle*{4}}
\qbezier(405,549)(405,529)(405,508)
\qbezier(349,508)(377,508)(405,508)
\put(456,577){\circle*{4}}
\qbezier(363,577)(409,577)(456,577)
\put(456,598){\circle*{4}}
\qbezier(456,577)(456,588)(456,598)
\put(414,577){\circle*{4}}
\put(414,598){\circle*{4}}
\qbezier(414,577)(414,588)(414,598)
\put(442,549){\circle*{4}}
\qbezier(405,549)(423,549)(442,549)
\put(331,547){$v_{1}$}
\put(331,504){$v_{2}$}
\put(409,504){$v_{3}$}
\put(389,541){$v_{4}$}
\put(353,585){$v_{5}$}
\put(402,566){$v_{5+s}$}
\put(366,473){$G_{3}$}
\put(592,577){\circle*{4}}
\put(608,548){\circle*{4}}
\qbezier(592,577)(600,563)(608,548)
\put(548,546){\circle*{4}}
\qbezier(592,577)(570,562)(548,546)
\put(548,505){\circle*{4}}
\qbezier(548,546)(548,526)(548,505)
\put(608,505){\circle*{4}}
\qbezier(608,548)(608,527)(608,505)
\qbezier(548,505)(578,505)(608,505)
\put(497,577){\circle*{4}}
\qbezier(497,577)(544,577)(592,577)
\put(497,597){\circle*{4}}
\qbezier(497,577)(497,587)(497,597)
\put(541,577){\circle*{4}}
\put(541,598){\circle*{4}}
\qbezier(541,577)(541,588)(541,598)
\put(512,505){\circle*{4}}
\qbezier(548,505)(530,505)(512,505)
\put(530,545){$v_{1}$}
\put(550,510){$v_{2}$}
\put(613,502){$v_{3}$}
\put(614,548){$v_{4}$}
\put(592,582){$v_{5}$}
\put(523,567){$v_{s+5}$}
\put(568,473){$G_{4}$}
\put(35,415){\circle*{4}}
\put(18,381){\circle*{4}}
\qbezier(35,415)(26,398)(18,381)
\put(73,385){\circle*{4}}
\qbezier(35,415)(54,400)(73,385)
\put(18,343){\circle*{4}}
\qbezier(18,381)(18,362)(18,343)
\put(73,343){\circle*{4}}
\qbezier(73,385)(73,364)(73,343)
\qbezier(18,343)(45,343)(73,343)
\put(101,415){\circle*{4}}
\qbezier(35,415)(68,415)(101,415)
\put(101,441){\circle*{4}}
\qbezier(101,415)(101,428)(101,441)
\put(71,415){\circle*{4}}
\put(71,440){\circle*{4}}
\qbezier(71,415)(71,428)(71,440)
\put(35,440){\circle*{4}}
\qbezier(35,415)(35,428)(35,440)
\put(1,381){$v_{1}$}
\put(8,331){$v_{2}$}
\put(78,338){$v_{3}$}
\put(56,381){$v_{4}$}
\put(16,417){$v_{5}$}
\put(102,385){\circle*{4}}
\qbezier(73,385)(87,385)(102,385)
\put(40,306){$G_{5}$}
\put(190,415){\circle*{4}}
\put(207,383){\circle*{4}}
\qbezier(190,415)(198,399)(207,383)
\put(151,381){\circle*{4}}
\qbezier(190,415)(170,398)(151,381)
\put(151,343){\circle*{4}}
\qbezier(151,381)(151,362)(151,343)
\put(207,343){\circle*{4}}
\qbezier(151,343)(179,343)(207,343)
\qbezier(207,383)(207,363)(207,343)
\put(125,415){\circle*{4}}
\qbezier(190,415)(157,415)(125,415)
\put(125,440){\circle*{4}}
\qbezier(125,415)(125,428)(125,440)
\put(153,415){\circle*{4}}
\put(153,441){\circle*{4}}
\qbezier(153,415)(153,428)(153,441)
\put(190,441){\circle*{4}}
\qbezier(190,415)(190,428)(190,441)
\put(123,343){\circle*{4}}
\qbezier(151,343)(137,343)(123,343)
\put(155,375){$v_{1}$}
\put(140,332){$v_{2}$}
\put(212,340){$v_{3}$}
\put(212,382){$v_{4}$}
\put(194,415){$v_{5}$}
\put(160,306){$G_{6}$}
\put(276,413){\circle*{4}}
\put(261,386){\circle*{4}}
\qbezier(276,413)(268,400)(261,386)
\put(261,341){\circle*{4}}
\qbezier(261,386)(261,364)(261,341)
\put(314,383){\circle*{4}}
\qbezier(276,413)(295,398)(314,383)
\put(314,341){\circle*{4}}
\qbezier(314,383)(314,362)(314,341)
\qbezier(261,341)(287,341)(314,341)
\put(344,413){\circle*{4}}
\qbezier(276,413)(310,413)(344,413)
\put(344,440){\circle*{4}}
\qbezier(344,413)(344,427)(344,440)
\put(314,413){\circle*{4}}
\put(314,440){\circle*{4}}
\qbezier(314,413)(314,427)(314,440)
\put(342,383){\circle*{4}}
\qbezier(314,383)(328,383)(342,383)
\put(343,341){\circle*{4}}
\qbezier(314,341)(328,341)(343,341)
\put(242,383){$v_{1}$}
\put(242,338){$v_{2}$}
\put(312,329){$v_{3}$}
\put(296,378){$v_{4}$}
\put(258,416){$v_{5}$}
\put(301,402){$v_{s+5}$}
\put(283,306){$G_{7}$}
\put(441,414){\circle*{4}}
\put(402,381){\circle*{4}}
\qbezier(441,414)(421,398)(402,381)
\put(456,383){\circle*{4}}
\qbezier(441,414)(448,399)(456,383)
\put(402,343){\circle*{4}}
\qbezier(402,381)(402,362)(402,343)
\put(456,343){\circle*{4}}
\qbezier(456,383)(456,363)(456,343)
\qbezier(402,343)(429,343)(456,343)
\put(372,414){\circle*{4}}
\qbezier(441,414)(406,414)(372,414)
\put(372,441){\circle*{4}}
\qbezier(372,414)(372,428)(372,441)
\put(403,414){\circle*{4}}
\put(403,442){\circle*{4}}
\qbezier(403,414)(403,428)(403,442)
\put(375,381){\circle*{4}}
\qbezier(402,381)(388,381)(375,381)
\put(484,343){\circle*{4}}
\qbezier(456,343)(470,343)(484,343)
\put(407,375){$v_{1}$}
\put(384,338){$v_{2}$}
\put(449,332){$v_{3}$}
\put(460,380){$v_{4}$}
\put(440,418){$v_{5}$}
\put(390,403){$v_{s+5}$}
\put(423,306){$G_{8}$}
\put(554,415){\circle*{4}}
\put(538,383){\circle*{4}}
\qbezier(554,415)(546,399)(538,383)
\put(590,384){\circle*{4}}
\qbezier(554,415)(572,400)(590,384)
\put(538,343){\circle*{4}}
\qbezier(538,383)(538,363)(538,343)
\put(590,343){\circle*{4}}
\qbezier(590,384)(590,364)(590,343)
\qbezier(538,343)(564,343)(590,343)
\put(627,415){\circle*{4}}
\qbezier(554,415)(590,415)(627,415)
\put(627,444){\circle*{4}}
\qbezier(627,415)(627,430)(627,444)
\put(594,415){\circle*{4}}
\put(594,443){\circle*{4}}
\qbezier(594,415)(594,429)(594,443)
\put(512,383){\circle*{4}}
\qbezier(538,383)(525,383)(512,383)
\put(617,384){\circle*{4}}
\qbezier(590,384)(603,384)(617,384)
\put(521,373){$v_{1}$}
\put(520,341){$v_{2}$}
\put(594,343){$v_{3}$}
\put(572,378){$v_{4}$}
\put(536,418){$v_{5}$}
\put(585,404){$v_{s+5}$}
\put(556,306){$G_{9}$}
\put(72,253){\circle*{4}}
\put(38,219){\circle*{4}}
\qbezier(72,253)(55,236)(38,219)
\put(92,222){\circle*{4}}
\qbezier(72,253)(82,238)(92,222)
\put(38,181){\circle*{4}}
\qbezier(38,219)(38,200)(38,181)
\put(92,181){\circle*{4}}
\qbezier(92,222)(92,202)(92,181)
\qbezier(38,181)(65,181)(92,181)
\put(8,253){\circle*{4}}
\qbezier(72,253)(40,253)(8,253)
\put(8,282){\circle*{4}}
\qbezier(8,253)(8,268)(8,282)
\put(37,253){\circle*{4}}
\put(37,282){\circle*{4}}
\qbezier(37,253)(37,268)(37,282)
\put(14,181){\circle*{4}}
\qbezier(38,181)(26,181)(14,181)
\put(119,181){\circle*{4}}
\qbezier(92,181)(105,181)(119,181)
\put(22,219){$v_{1}$}
\put(30,169){$v_{2}$}
\put(83,170){$v_{3}$}
\put(96,219){$v_{4}$}
\put(75,255){$v_{5}$}
\put(46,144){$G_{10}$}
\put(22,243){$v_{s+5}$}
\put(167,253){\circle*{4}}
\put(148,221){\circle*{4}}
\qbezier(167,253)(157,237)(148,221)
\put(148,181){\circle*{4}}
\qbezier(148,221)(148,201)(148,181)
\put(200,220){\circle*{4}}
\qbezier(167,253)(183,237)(200,220)
\put(200,181){\circle*{4}}
\qbezier(200,220)(200,201)(200,181)
\qbezier(148,181)(174,181)(200,181)
\put(236,253){\circle*{4}}
\qbezier(167,253)(201,253)(236,253)
\put(236,282){\circle*{4}}
\qbezier(236,253)(236,268)(236,282)
\put(208,253){\circle*{4}}
\put(208,283){\circle*{4}}
\qbezier(208,253)(208,268)(208,283)
\put(135,253){\circle*{4}}
\qbezier(167,253)(151,253)(135,253)
\put(229,220){\circle*{4}}
\qbezier(200,220)(214,220)(229,220)
\put(230,181){\circle*{4}}
\qbezier(200,181)(215,181)(230,181)
\put(131,219){$v_{1}$}
\put(141,168){$v_{2}$}
\put(198,168){$v_{3}$}
\put(184,216){$v_{4}$}
\put(159,260){$v_{5}$}
\put(163,144){$G_{11}$}
\put(335,253){\circle*{4}}
\put(295,219){\circle*{4}}
\qbezier(335,253)(315,236)(295,219)
\put(295,181){\circle*{4}}
\qbezier(295,219)(295,200)(295,181)
\put(349,220){\circle*{4}}
\qbezier(335,253)(342,237)(349,220)
\put(349,181){\circle*{4}}
\qbezier(349,220)(349,201)(349,181)
\qbezier(295,181)(322,181)(349,181)
\put(261,253){\circle*{4}}
\qbezier(335,253)(298,253)(261,253)
\put(261,280){\circle*{4}}
\qbezier(261,253)(261,267)(261,280)
\put(296,253){\circle*{4}}
\put(296,281){\circle*{4}}
\qbezier(296,253)(296,267)(296,281)
\put(368,253){\circle*{4}}
\qbezier(335,253)(351,253)(368,253)
\put(269,219){\circle*{4}}
\qbezier(295,219)(282,219)(269,219)
\put(376,220){\circle*{4}}
\qbezier(349,220)(362,220)(376,220)
\put(282,225){$v_{1}$}
\put(278,178){$v_{2}$}
\put(353,177){$v_{3}$}
\put(331,219){$v_{4}$}
\put(329,258){$v_{5}$}
\put(310,144){$G_{12}$}
\put(428,255){\circle*{4}}
\put(411,221){\circle*{4}}
\qbezier(428,255)(419,238)(411,221)
\put(467,222){\circle*{4}}
\qbezier(428,255)(447,239)(467,222)
\put(411,181){\circle*{4}}
\qbezier(411,221)(411,201)(411,181)
\put(467,181){\circle*{4}}
\qbezier(467,222)(467,202)(467,181)
\qbezier(411,181)(439,181)(467,181)
\put(498,255){\circle*{4}}
\qbezier(428,255)(463,255)(498,255)
\put(469,255){\circle*{4}}
\put(469,278){\circle*{4}}
\qbezier(469,255)(469,267)(469,278)
\put(498,279){\circle*{4}}
\qbezier(498,255)(498,267)(498,279)
\put(496,222){\circle*{4}}
\qbezier(467,222)(481,222)(496,222)
\put(496,181){\circle*{4}}
\qbezier(467,181)(481,181)(496,181)
\put(399,245){\circle*{4}}
\qbezier(411,221)(405,233)(399,245)
\put(424,144){$G_{13}$}
\put(414,212){$v_{1}$}
\put(394,176){$v_{2}$}
\put(463,170){$v_{3}$}
\put(450,216){$v_{4}$}
\put(421,261){$v_{5}$}
\put(458,244){$v_{5+s}$}
\put(590,255){\circle*{4}}
\put(551,221){\circle*{4}}
\qbezier(590,255)(570,238)(551,221)
\put(606,222){\circle*{4}}
\qbezier(590,255)(598,239)(606,222)
\put(551,181){\circle*{4}}
\qbezier(551,221)(551,201)(551,181)
\put(606,181){\circle*{4}}
\qbezier(606,222)(606,202)(606,181)
\qbezier(551,181)(578,181)(606,181)
\put(524,181){\circle*{4}}
\qbezier(551,181)(537,181)(524,181)
\put(524,221){\circle*{4}}
\qbezier(551,221)(537,221)(524,221)
\put(636,181){\circle*{4}}
\qbezier(606,181)(621,181)(636,181)
\put(522,255){\circle*{4}}
\qbezier(590,255)(556,255)(522,255)
\put(522,279){\circle*{4}}
\qbezier(522,255)(522,267)(522,279)
\put(555,255){\circle*{4}}
\put(555,280){\circle*{4}}
\qbezier(555,255)(555,268)(555,280)
\put(554,214){$v_{1}$}
\put(543,170){$v_{2}$}
\put(603,169){$v_{3}$}
\put(610,219){$v_{4}$}
\put(587,260){$v_{5}$}
\put(540,244){$v_{5+s}$}
\put(567,144){$G_{14}$}
\put(47,97){\circle*{4}}
\put(25,64){\circle*{4}}
\qbezier(47,97)(36,81)(25,64)
\put(80,66){\circle*{4}}
\qbezier(47,97)(63,82)(80,66)
\put(25,23){\circle*{4}}
\qbezier(25,64)(25,44)(25,23)
\put(80,23){\circle*{4}}
\qbezier(80,66)(80,45)(80,23)
\qbezier(25,23)(52,23)(80,23)
\put(118,97){\circle*{4}}
\qbezier(47,97)(82,97)(118,97)
\put(118,122){\circle*{4}}
\qbezier(118,97)(118,110)(118,122)
\put(90,97){\circle*{4}}
\put(90,122){\circle*{4}}
\qbezier(90,97)(90,110)(90,122)
\put(17,97){\circle*{4}}
\qbezier(47,97)(32,97)(17,97)
\put(111,66){\circle*{4}}
\qbezier(80,66)(95,66)(111,66)
\put(111,23){\circle*{4}}
\qbezier(80,23)(95,23)(111,23)
\put(9,34){\circle*{4}}
\qbezier(25,23)(17,29)(9,34)
\put(29,58){$v_{1}$}
\put(18,12){$v_{2}$}
\put(78,12){$v_{3}$}
\put(62,60){$v_{4}$}
\put(42,103){$v_{5}$}
\put(43,-12){$G_{15}$}
\put(176,97){\circle*{4}}
\put(154,65){\circle*{4}}
\qbezier(176,97)(165,81)(154,65)
\put(154,23){\circle*{4}}
\qbezier(154,65)(154,44)(154,23)
\put(205,66){\circle*{4}}
\qbezier(176,97)(190,82)(205,66)
\put(205,23){\circle*{4}}
\qbezier(205,66)(205,45)(205,23)
\qbezier(154,23)(179,23)(205,23)
\put(129,34){\circle*{4}}
\qbezier(154,23)(141,29)(129,34)
\put(129,55){\circle*{4}}
\qbezier(154,65)(141,60)(129,55)
\put(243,97){\circle*{4}}
\qbezier(176,97)(209,97)(243,97)
\put(243,122){\circle*{4}}
\qbezier(243,97)(243,110)(243,122)
\put(215,97){\circle*{4}}
\put(215,121){\circle*{4}}
\qbezier(215,97)(215,109)(215,121)
\put(232,66){\circle*{4}}
\qbezier(205,66)(218,66)(232,66)
\put(232,23){\circle*{4}}
\qbezier(205,23)(218,23)(232,23)
\put(140,70){$v_{1}$}
\put(143,12){$v_{2}$}
\put(197,12){$v_{3}$}
\put(187,62){$v_{4}$}
\put(169,104){$v_{5}$}
\put(166,-12){$G_{16}$}
\put(200,86){$v_{s+5}$}
\put(565,97){\circle*{4}}
\put(546,66){\circle*{4}}
\qbezier(565,97)(555,82)(546,66)
\put(600,68){\circle*{4}}
\qbezier(565,97)(582,83)(600,68)
\put(546,24){\circle*{4}}
\qbezier(546,66)(546,45)(546,24)
\put(600,24){\circle*{4}}
\qbezier(600,68)(600,46)(600,24)
\qbezier(546,24)(573,24)(600,24)
\put(527,80){\circle*{4}}
\qbezier(546,66)(536,73)(527,80)
\put(526,44){\circle*{4}}
\qbezier(546,24)(536,34)(526,44)
\put(628,36){\circle*{4}}
\qbezier(600,24)(614,30)(628,36)
\put(628,75){\circle*{4}}
\qbezier(600,68)(614,72)(628,75)
\put(537,97){\circle*{4}}
\qbezier(565,97)(551,97)(537,97)
\put(633,97){\circle*{4}}
\qbezier(565,97)(599,97)(633,97)
\put(633,122){\circle*{4}}
\qbezier(633,97)(633,110)(633,122)
\put(605,97){\circle*{4}}
\put(605,121){\circle*{4}}
\qbezier(605,97)(605,109)(605,121)
\put(550,59){$v_{1}$}
\put(543,13){$v_{2}$}
\put(591,13){$v_{3}$}
\put(585,61){$v_{4}$}
\put(558,104){$v_{5}$}
\put(561,-12){$G_{19}$}
\put(251,-36){Fig. 4.1. $G_{1}-G_{19}$}
\put(304,97){\circle*{4}}
\put(282,63){\circle*{4}}
\qbezier(304,97)(293,80)(282,63)
\put(336,65){\circle*{4}}
\qbezier(304,97)(320,81)(336,65)
\put(282,22){\circle*{4}}
\qbezier(282,63)(282,43)(282,22)
\put(336,22){\circle*{4}}
\qbezier(336,65)(336,44)(336,22)
\qbezier(282,22)(309,22)(336,22)
\put(367,97){\circle*{4}}
\qbezier(304,97)(335,97)(367,97)
\put(340,97){\circle*{4}}
\put(340,119){\circle*{4}}
\qbezier(340,97)(340,108)(340,119)
\put(367,119){\circle*{4}}
\qbezier(367,97)(367,108)(367,119)
\put(276,97){\circle*{4}}
\qbezier(304,97)(290,97)(276,97)
\put(258,63){\circle*{4}}
\qbezier(282,63)(270,63)(258,63)
\put(257,22){\circle*{4}}
\qbezier(282,22)(269,22)(257,22)
\put(362,22){\circle*{4}}
\qbezier(336,22)(349,22)(362,22)
\put(286,58){$v_{1}$}
\put(277,10){$v_{2}$}
\put(333,10){$v_{3}$}
\put(341,61){$v_{4}$}
\put(300,102){$v_{5}$}
\put(298,-12){$G_{17}$}
\put(435,97){\circle*{4}}
\put(410,66){\circle*{4}}
\qbezier(435,97)(422,82)(410,66)
\put(410,22){\circle*{4}}
\qbezier(410,66)(410,44)(410,22)
\put(466,66){\circle*{4}}
\qbezier(435,97)(450,82)(466,66)
\put(466,22){\circle*{4}}
\qbezier(410,22)(438,22)(466,22)
\qbezier(466,66)(466,44)(466,22)
\put(385,22){\circle*{4}}
\qbezier(410,22)(397,22)(385,22)
\put(386,66){\circle*{4}}
\qbezier(410,66)(398,66)(386,66)
\put(408,97){\circle*{4}}
\qbezier(435,97)(421,97)(408,97)
\put(506,97){\circle*{4}}
\qbezier(435,97)(470,97)(506,97)
\put(473,97){\circle*{4}}
\put(473,119){\circle*{4}}
\qbezier(473,97)(473,108)(473,119)
\put(506,119){\circle*{4}}
\qbezier(506,97)(506,108)(506,119)
\put(497,66){\circle*{4}}
\qbezier(466,66)(481,66)(497,66)
\put(414,59){$v_{1}$}
\put(406,11){$v_{2}$}
\put(469,14){$v_{3}$}
\put(449,59){$v_{4}$}
\put(430,102){$v_{5}$}
\put(425,-12){$G_{18}$}
\end{picture}
\end{center}

\

For the both cases that $|x_{2}|=\min\{|x_{1}|$, $|x_{2}|$, $|x_{3}|\}$ and $|x_{3}|=\min\{|x_{1}|$, $|x_{2}|$, $|x_{3}|\}$. As the case that $|x_{1}|=\min\{|x_{1}|$, $|x_{2}|$, $|x_{3}|\}$, it is proved that there exists a graph $\mathbb{H}$ such that $g(\mathbb{H})=3$, $\gamma(G_{3})\leq \gamma(\mathbb{H})$ and $q_{min}(\mathbb{H})< q_{min}(G_{3})$.

In a same way, for $G_{4}$, it is proved that there exists a graph $\mathbb{H}$ such that $g(\mathbb{H})=3$, $\gamma(G_{4})\leq \gamma(\mathbb{H})$ and $q_{min}(\mathbb{H})< q_{min}(G_{4})$.

For the cases that {\bf Case 3} there is exactly $2$ $p$-dominators on $\mathbb{C}$ (see $G_{5}-G_{10}$ in Fig. 4.1); {\bf Case 4} there is exactly $3$ $p$-dominators on $\mathbb{C}$ (see $G_{11}-G_{15}$ in Fig. 4.1); {\bf Case 5} there is exactly $4$ $p$-dominators on $\mathbb{C}$ (see $G_{16}-G_{18}$ in Fig. 4.1); {\bf Case 6} there is exactly $5$ $p$-dominators on $\mathbb{C}$ (see $G_{19}$ in Fig. 4.1), in a same way, it is proved that the exists a a graph $\mathbb{H}$ such that $g(\mathbb{H})=3$, $\gamma(G)\leq \gamma(\mathbb{H})$ and $q_{min}(\mathbb{H})\leq q_{min}(G)$.
Thus, the result follows as desired. \ \ \ \ \ $\Box$
\end{proof}

\begin{lemma}\label{le04,11} %------
For odd cycle $\mathscr{C}=v_{1}v_{2}v_{3}\cdots v_{n-1}v_{n}v_{1}$, $\gamma(\mathscr{K})=\gamma(\mathscr{C})$, and
$q_{min}(\mathscr{K})\leq q_{min}(\mathscr{C})$ with equality if and only if $n=3$ ($\mathscr{K}$ is shown in Fig. 3.7).
\end{lemma}

\begin{proof}
The theorem holds for $n=3$ clearly. Next, suppose that $n\geq 5$.

{\bf Claim} For $\mathscr{C}$, there is an eigenvector $X=(\,x_1$, $x_2$, $\ldots$, $x_{n-1}$, $x_{n}\,)^{T}$ such that $x_{i}=x_{n-i}$ for $1\leq i\leq \lfloor\frac{n}{2}\rfloor$, $|x_{n}|=\max\{|x_{i}||\, 1\leq i\leq n\}$. Let $Y=(\,y_1$, $y_2$, $\ldots$,  $y_{n-1}$, $y_{n}\,)^{T}$ be an eigenvector corresponding to $q_{min}(\mathscr{C})$ and $|y_{n}|=\max\{|y_{i}||\, 1\leq i\leq n\}$. Clearly, $|y_{n}|> 0$. Let $Y^{'}=(\,y^{'}_1$, $y^{'}_2$, $\ldots$,  $y^{'}_{n-1}$, $y^{'}_{n}\,)^{T}$ satisfy that $y^{'}_{n}=y_{n}$, and satisfy that for $1\leq i\leq \lfloor\frac{n}{2}\rfloor$, $y^{'}_{i}=y_{n-i}$, $y^{'}_{n-i}=y_{i}$. Then $Y^{'}$ is also an eigenvector corresponding to $q_{min}(\mathscr{C})$. Let $X=Y+Y^{'}$. Then $X$ makes the claim holding.

For such $X$, $x_{\lfloor\frac{n}{2}\rfloor}=x_{n-\lfloor\frac{n}{2}\rfloor}$. Then by Lemma \ref{le04,01} it follows that $x_{\lfloor\frac{n}{2}\rfloor}\neq0$, and moreover, $|x_{\lfloor\frac{n}{2}\rfloor}|=\min\{x_{i}|\, 1\leq i\leq n\}$. Note that $q_{min}(\mathscr{C})x_{n}=2x_{n}+x_{1}+x_{n-1}$. Combined with Lemma \ref{le04,01}, it follows that $|x_{1}|<|x_{n}|$. Using Lemma \ref{le04,01} again gets that $|x_{\lfloor\frac{n}{2}\rfloor}|<|x_{\lfloor\frac{n-2}{2}\rfloor}|<\cdots<|x_{2}|<|x_{1}|<|x_{n}|$.

(i) $n\equiv 1$ (mod 3). Then $n\geq 7$. Suppose $n=3k+1$ where $k\geq 2$. Because $n\geq 7$, it follows that if $\lceil\frac{n+8}{2}\rceil\leq n$, then $|x_{\lceil\frac{n+8}{2}\rceil}|> |x_{\lceil\frac{n+4}{2}\rceil}|$; if $\lceil\frac{n+8}{2}\rceil> n$, then $|x_{t}|\geq|x_{\lceil\frac{n+4}{2}\rceil}|$ ($t\equiv \lceil\frac{n+8}{2}\rceil$ (mod n), $1\leq t\leq n-1$). Let $Z=(\,z_1$, $z_2$, $\ldots$, $z_{n-1}$, $z_{n}\,)^{T}$ satisfy that if $\lceil\frac{n+8}{2}\rceil\leq n$, then
$z_{\lceil\frac{n+2}{2}\rceil}|=-sgn(z_{\lceil\frac{n+8}{2}\rceil})(|x_{\lceil\frac{n+2}{2}\rceil}|+
|x_{\lceil\frac{n+8}{2}\rceil}|-|x_{\lceil\frac{n+4}{2}\rceil}|)$, $z_{i}=x_{i}$ for $i\neq \lceil\frac{n+2}{2}\rceil$; if $\lceil\frac{n+8}{2}\rceil> n$, then
$z_{\lceil\frac{n+2}{2}\rceil}|=-sgn(z_{t})(|x_{\lceil\frac{n+2}{2}\rceil}|+
|x_{t}|-|x_{\lceil\frac{n+4}{2}\rceil}|)$, $z_{i}=x_{i}$ for $i\neq \lceil\frac{n+2}{2}\rceil$. Thus $Z^{T}Z\geq X^{T}X$ and $\displaystyle q_{min}(\mathscr{K})\leq \frac{Z^{T}Q(\mathscr{K})Z}{Z^{T}Z}\leq \frac{X^{T}Q(\mathscr{C})X}{X^{T}X}=q_{min}(\mathscr{C})$. As Theorem \ref{th04,02}, it is proved that $q_{min}(\mathscr{K})<q_{min}(\mathscr{C})$. Combined with Theorem \ref{th03,09}, the theorem holds for this case.

(ii) $n\not\equiv 1$ (mod 3) and $n\geq 5$. As (i), it is proved that the theorem holds.
\ \ \ \ \ $\Box$
\end{proof}

\setlength{\unitlength}{0.7pt}
\begin{center}
\begin{picture}(390,313)
\put(38,295){\circle*{4}}
\put(38,218){\circle*{4}}
\qbezier(38,295)(38,257)(38,218)
\put(84,257){\circle*{4}}
\qbezier(38,295)(61,276)(84,257)
\qbezier(38,218)(61,238)(84,257)
\put(222,257){\circle*{4}}
\put(146,257){\circle*{4}}
\put(244,257){\circle*{4}}
\put(233,257){\circle*{4}}
\put(261,257){\circle*{4}}
\put(261,287){\circle*{4}}
\put(323,257){\circle*{4}}
\qbezier(261,257)(261,272)(261,287)
\put(205,257){\circle*{4}}
\put(382,257){\circle*{4}}
\put(37,133){\circle*{4}}
\put(37,57){\circle*{4}}
\put(79,95){\circle*{4}}
\put(147,95){\circle*{4}}
\put(244,95){\circle*{4}}
\put(260,95){\circle*{4}}
\put(221,95){\circle*{4}}
\put(233,95){\circle*{4}}
\qbezier(37,133)(37,95)(37,57)
\qbezier(37,133)(58,114)(79,95)
\qbezier(37,57)(58,76)(79,95)
\put(36,302){$v_{2}$}
\put(82,244){$v_{3}$}
\put(137,244){$v_{4}$}
\put(36,205){$v_{1}$}
\put(33,141){$v_{2}$}
\put(79,80){$v_{3}$}
\put(33,43){$v_{1}$}
\put(130,-9){Fig. 4.2. $\mathscr{H}_{3,\frac{n}{2}}$, $\mathscr{H}_{3,\frac{n-3}{2}}$}
\put(204,182){$\mathscr{H}_{3,\frac{n}{2}}$}
\put(202,23){$\mathscr{H}_{3,\frac{n-3}{2}}$}
\put(389,285){$v_{\frac{n+2}{2}}$}
\put(389,252){$v_{\frac{n}{2}}$}
\put(382,95){\circle*{4}}
\put(323,95){\circle*{4}}
\put(374,79){$v_{\frac{n+3}{2}}$}
\put(389,126){$v_{\frac{n+5}{2}}$}
\qbezier(38,218)(38,257)(38,295)
\qbezier(38,218)(38,257)(38,295)
\qbezier(38,218)(38,257)(38,295)
\qbezier(38,218)(38,257)(38,295)
\put(382,127){\circle*{4}}
\qbezier(382,95)(382,111)(382,127)
\put(323,127){\circle*{4}}
\qbezier(323,95)(323,111)(323,127)
\put(260,126){\circle*{4}}
\qbezier(260,95)(260,111)(260,126)
\put(307,80){$v_{\frac{n+1}{2}}$}
\put(382,284){\circle*{4}}
\qbezier(382,257)(382,271)(382,284)
\qbezier(323,257)(352,257)(382,257)
\qbezier(323,257)(292,257)(261,257)
\put(323,287){\circle*{4}}
\qbezier(323,257)(323,272)(323,287)
\put(84,287){\circle*{4}}
\qbezier(84,257)(84,272)(84,287)
\put(146,286){\circle*{4}}
\qbezier(146,257)(146,272)(146,286)
\qbezier(146,257)(175,257)(205,257)
\qbezier(146,257)(115,257)(84,257)
\put(205,286){\circle*{4}}
\qbezier(205,257)(205,272)(205,286)
\put(1,218){\circle*{4}}
\qbezier(38,218)(19,218)(1,218)
\put(0,295){\circle*{4}}
\qbezier(38,295)(19,295)(0,295)
\qbezier(323,95)(352,95)(382,95)
\qbezier(260,95)(291,95)(323,95)
\put(312,243){$v_{\frac{n-2}{2}}$}
\put(206,128){\circle*{4}}
\qbezier(79,95)(113,95)(147,95)
\put(147,127){\circle*{4}}
\qbezier(147,95)(147,111)(147,127)
\put(206,95){\circle*{4}}
\qbezier(147,95)(176,95)(206,95)
\qbezier(206,95)(206,112)(206,128)
\put(143,79){$v_{4}$}
\end{picture}
\end{center}

\begin{lemma}\label{le04,06} %------
Let $G$ be a nonbipartite $\mathcal {F}_{g, l}$-graph of order $n$ for some $l$ and with domination number $\frac{n-1}{2}$. Then $q_{min}(G)\geq q_{min}(\mathscr{H}_{3,\frac{n-3}{2}})$ with equality if and only if $G\cong \mathscr{H}_{3,\frac{n-3}{2}}$ (see Fig. 4.2).
\end{lemma}

\begin{proof}
Because $G$ is nonbipartite, $g$ is odd. If $G$ is a $\mathcal {F}^{\circ}_{g, l}$-graph, then the theorem follows from Theorem \ref{th04,05}. If $g=3$, then the theorem follows from Theorem \ref{th04,04}. For $g=5$, the theorem follows from Lemmas \ref{le04,09}, \ref{le04,11}. Next we consider the case that $G$ is not a $\mathcal {F}^{\circ}_{g, l}$-graph and suppose $g\geq 7$.

Note that by Lemma \ref{le03,08}, in $G$, there are at most $3$ consecutive vertices of $\mathbb{C}$ such that each of them is not $p$-dominator, and there are $2$ cases as follows to consider.  Let $X=(\,x_{1}$, $x_{2}$, $\ldots$, $x_{n}\,)^T$ is a unit eigenvector corresponding to $q_{min}(G)$. Suppose $x_{a}=\min\{|x_{1}|$, $|x_{2}|$, $\ldots$, $|x_{g}|\}$.

{\bf Case 1} In $G$, there is exactly one vertex of $\mathbb{C}$ which is not $p$-dominator. Note that $G$ is not a $\mathcal {F}^{\circ}_{g, l}$-graph. Then $n\geq g+2$ and $v_{g}$ is the only one vertex which is not $p$-dominator on $\mathbb{C}$. As Lemma \ref{le04,03}, it is proved that $x_{g}=\max\{|x_{1}|$, $|x_{2}|$, $\ldots$, $|x_{g-1}|$, $|x_{g}|\}$. Then we suppose $a\leq g-1$. By Lemma \ref{le04,01}, if $a\leq g-3$, without loss of generality, suppose $x_{a+1}\leq x_{a-1}$, $x_{a+1}x_{a}\geq 0$, $|x_{a-1}|\geq |x_{a+2}|$. Let $G_{1}=G-v_{a}v_{a-1}+v_{a}v_{a+2}$ (if $x_{a-1}\leq x_{a+2}$ and $a\geq 2$, let $G_{1}=G-v_{a+1}v_{a+2}+v_{a+1}v_{a-1}$; if $a=1$, let $G_{1}=G-v_{1}v_{g}+v_{1}v_{3}$). If $a= g-2$, suppose $x_{g-1}\leq x_{g-3}$, $x_{g-1}x_{g-2}\geq 0$, and then let $G_{1}=G-v_{g-1}v_{g}+v_{g-1}v_{g-3}$. If $a= g-1$, because $|x_{g}|\geq |x_{g-2}|$, then suppose $x_{g-1}x_{g-2}\geq 0$. Let $G_{1}=G-v_{g-1}v_{g}+v_{g-1}v_{g-3}$. Note that $\gamma(G_{1})\leq \frac{n-1}{2}$. As Lemma \ref{le03,03}, it is proved that $\gamma(G)\leq \gamma(G_{1})$. It follows that $\gamma(G_{1})= \frac{n-1}{2}$. As Theorem \ref{th04,02}, we get that $q_{min}(G_{1})< q_{min}(G)$. Note that $g(G_{1})=3$. Then the theorem follows from Theorem \ref{th04,04}.

{\bf Case 2} In $G$, there are exactly $3$ consecutive vertices of $\mathbb{C}$ such that each of them is not $p$-dominator.
Note that $G$ is not a $\mathcal {F}^{\circ}_{g, l}$-graph. Combined with Lemma \ref{le03,08}, the $3$ vertices of $\mathbb{C}$ such that each of them is not $p$-dominator are $v_{g-2}$, $v_{g-1}$, $v_{g}$ or $v_{g}$, $v_{1}$, $v_{2}$. Without loss of generality, we suppose the $3$ vertices are $v_{g-2}$, $v_{g-1}$, $v_{g}$. By Lemma \ref{th03,02}, $|x_{g}|>0$. We say that $|x_{g}|>|x_{g-2}|$. Otherwise, suppose $|x_{g}|\leq|x_{g-2}|$. Let $G^{'}=G-v_{g}v_{g+1}+v_{g+1}v_{g-2}$. Then by Lemma \ref{le0,04}, $q_{min}(G^{'})< q_{min}(G)$. This is a contradiction because $G^{'}\cong G$. Hence $|x_{g}|>|x_{g-2}|$. And then $a\leq g-1$.

{\bf Subcase 2.1} $a\leq g-4$. By Lemma \ref{le04,01}, without loss of generality, suppose $x_{a+1}\leq x_{a-1}$, $x_{a+1}x_{a}\geq 0$.
As Case 1, it is proved the theorem holds.

{\bf Subcase 2.2} $a= g-3$. By Lemma \ref{le04,01}, suppose $x_{g-2}\leq x_{g-4}$, $x_{g-2}x_{g-3}\geq 0$; suppose $|x_{g-4}|\geq |x_{g-1}|$. Denote by $v_{\tau_{g-3}}$ the pendant vertex attached to $v_{g-3}$. Let $G_{1}=G-v_{g-3}v_{g-4}+v_{g-3}v_{g-1}-v_{g-3}v_{\tau_{g-3}}+v_{g}v_{\tau_{g-3}}$ (if $x_{g-4}\leq x_{g-1}$, let $G_{1}=G-v_{g-2}v_{g-1}+x_{g-2}x_{g-4}$). As Case 1, it is proved the theorem holds.

{\bf Subcase 2.3} $a= g-2$. By Lemma \ref{le04,01}, suppose $x_{g-1}\leq x_{g-3}$, $x_{g-1}x_{g-2}\geq 0$; suppose $|x_{g-3}|\geq |x_{g}|$. Denote by $v_{\tau_{g-3}}$ the pendant vertex attached to $v_{g-3}$. Let $G_{1}=G-v_{g-2}v_{g-3}+v_{g-2}v_{g}$ (if $x_{g-3}\leq x_{g}$, let $G_{1}=G-v_{g-1}v_{g}+x_{g-1}x_{g-3}-v_{g-3}v_{\tau_{g-3}}+v_{g}v_{\tau_{g-3}}$). As Case 1, it is proved the theorem holds.

{\bf Subcase 2.4} $a= g-1$. Note $|x_{g}|>|x_{g-2}|$. By Lemma \ref{le04,01}, $x_{g-2}x_{g-1}\geq 0$. Without loss of generality, suppose $x_{g-3}\geq x_{g}$, let $G_{1}=G-v_{g-2}v_{g-3}+v_{g-2}v_{g}$ (if $x_{g-3}\leq x_{g}$, let $G_{1}=G-v_{g-1}v_{g}+x_{g-1}x_{g-3}-v_{g}v_{g+1}+v_{g-3}v_{g+1}$). As Case 1, it is proved the theorem holds.
This completes the proof. \ \ \ \ \ $\Box$
\end{proof}

By Theorem \ref{th03,02}, Lemma \ref{le04,06}, we get the following Theorem \ref{th04,07}.

\begin{theorem}\label{th04,07} %------
Let $G$ be a nonbipartite connected unicyclic graph of order $n\geq 3$ and with domination number $\frac{n-1}{2}$. Then $q_{min}(G)\geq q_{min}(\mathscr{H}_{3,\frac{n-3}{2}})$ with equality if and only if $G\cong \mathscr{H}_{3,\frac{n-3}{2}}$.
\end{theorem}

\begin{theorem}\label{th04,08} %------
Let $G$ be a nonbipartite unicyclic graph of order $n\geq 6$ with domination number $\frac{n}{2}$. Then $q_{min}(G)\geq q_{min}(\mathscr{H}_{3,\frac{n}{2}})$ with equality if and only if $G\cong \mathscr{H}_{3,\frac{n}{2}}$ (see Fig. 4.2).
\end{theorem}

\begin{proof}
This result follows from Lemma \ref{le0,10}, Theorems \ref{th04,02} and \ref{th04,04}.
 \ \ \ \ \ $\Box$
\end{proof}

By Theorems \ref{th04,04}, \ref{th04,07}, \ref{th04,08} and Lemma \ref{le04,09}, we get the following theorem \ref{th04,10}.

\begin{theorem}\label{th04,10} %------
Among all nonbipartite unicyclic graphs of order $n\geq 4$ and with girth $g\leq5$, and domination number at least $\frac{n+1}{3}<\gamma\leq \frac{n}{2}$, we have

$\mathrm{(i)}$ if $n=4$, the least $q_{min}$ attains the minimum uniquely at $\mathscr{H}_{3,1}$;

$\mathrm{(ii)}$ if $n\geq 5$ and $\gamma=\frac{n-1}{2}$, the least $q_{min}$ attains the minimum uniquely at a $\mathscr{H}_{3,\frac{n-3}{2}}$;

$\mathrm{(iii)}$ if $n\geq 6$ and $\gamma=\frac{n}{2}$, the least $q_{min}$ attains the minimum uniquely at a $\mathscr{H}_{3,\frac{n}{2}}$;

$\mathrm{(iv)}$ if $n\geq 5$ and $n-2\gamma\geq 2$, then the least $q_{min}$ attains the minimum uniquely at a $\mathscr{H}_{3,\alpha}$ where $\alpha\leq\frac{n-3}{2}$ is the least integer such that $\lceil\frac{n-2\alpha-2}{3}\rceil+\alpha=\gamma$.

\end{theorem}

\section{The $q_{min}$ among general graphs}

Let $\mathcal {M}=\{G|\,\ G$ be a nonbipartite graph of order  of order $n\geq 4$ and domination number at least $\frac{n+1}{3}<\gamma\leq \frac{n}{2}$, and $G$ have\ a nonbipartite\ spanning subgraph\ which is a $\mathcal {F}^{\circ}_{g, l}$-graph $\bbbeta$, where $g$ is any odd number at least $3$ and $l$ is any positive integral number$\}$ and $q_{\mathcal {M}}=\min\{q_{min}(G)|\,\ G\in \mathcal {M}\}$.

\begin{theorem}\label{th05,1} %------
\

$\mathrm{(i)}$ if $n=4$, $q_{\mathcal {M}}$ attains uniquely at $\mathscr{H}_{3,1}$;

$\mathrm{(ii)}$ if $n\geq 5$ and $n-2\gamma\geq 2$, then $q_{\mathcal {M}}> q_{min}(\mathscr{H}_{3,\alpha})$ where $\alpha\leq\frac{n-3}{2}$ is the least integer such that $\lceil\frac{n-2\alpha-2}{3}\rceil+\alpha=\gamma$.
\end{theorem}

\begin{proof}

If $G$ is a nonbipartite graph of order $4$, then $G$ contains $S^{+}_{4}$ as a subgraph. Note that $S^{+}_{4}=\mathscr{H}_{3,1}$. Suppose the vertices of $G$ are labeled as $\mathscr{H}_{3,1}$ (see Fig. 3.6). By Lemma \ref{le02,01}, we see that $q_{min}(S^{+}_{4})\leq q_{min}(G)$. For any vector $X$, note that $q_{min}(S^{+}_{4})\leq X^{T}Q(S^{+}_{4})X\leq X^{T}Q(G)X$. Suppose that $Y$ is an eigenvector corresponding to $q_{min}(G)$. If $q_{min}(S^{+}_{4})= q_{min}(G)$, then $q_{min}(S^{+}_{4})= Y^{T}Q(S^{+}_{4})Y= Y^{T}Q(G)Y=q_{min}(G)$, and then $Y$ is an eigenvector corresponding to $q_{min}(S^{+}_{4})$ too. If $G\neq S^{+}_{4}$, by Lemmas \ref{le02,03}, \ref{le04,03}, then for any edge $v_{i}v_{j}\not\in E(S^{+}_{4})$, it follows that $x_{i}+x_{j}\neq 0$, and then $Y^{T}Q(S^{+}_{4})Y< Y^{T}Q(G)Y$, which contradicts $Y^{T}Q(S^{+}_{4})Y= Y^{T}Q(G)Y$. Therefore, if $q_{min}(S^{+}_{4})= q_{min}(G)$ if and only if $G= S^{+}_{4}$. Then (i) follows.

In a same way as proved for (i), noting that $\gamma(G)\leq \gamma(\bbbeta)$, combining with Theorems \ref{th03,05} and \ref{th04,05}, it is proved that (ii) holds.
\ \ \ \ \ $\Box$
\end{proof}

\begin{theorem}\label{th05,2} %------
Let $G$ be a nonbipartite connected graph of order $n\geq 6$ with domination number $\frac{n}{2}$. Then $q_{min}(G)\geq q_{min}(\mathscr{H}_{3,\frac{n}{2}})$ with equality if and only if $G\cong \mathscr{H}_{3,\frac{n}{2}}$.
\end{theorem}

\begin{proof}
By Lemma \ref{le0,10}, $G=H\circ K_{1}$ for some connected nonbipartite graph $H$ of order $\frac{n}{2}$. By deleting edges from $H$, we can get a nonbipartite unicyclic spanning subgraph $\mathbb{G}=H^{'}\circ K_{1}$ of $G$, where $H^{'}$ is a connected nonbipartite unicyclic spanning graph of $H$. Combining with  Theorem \ref{th04,08}, we get $q_{min}(G)\geq q_{min}(\mathbb{G})\geq q_{min}(\mathscr{H}_{3,\frac{n}{2}})$ where $q_{min}(\mathbb{G})= q_{min}(\mathscr{H}_{3,\frac{n}{2}})$ if and only if $\mathbb{G}\cong \mathscr{H}_{3,\frac{n}{2}}$.

If $q_{min}(G)= q_{min}(\mathscr{H}_{3,\frac{n}{2}})$, then $\mathbb{G}\cong \mathscr{H}_{3,\frac{n}{2}}$. Suppose $\mathbb{G}= \mathscr{H}_{3,\frac{n}{2}}$ (see Fig. 4.2). Let $X=(\,x_1$, $x_2$, $\ldots$, $x_{n-1}$, $x_{n}\,)^{T}$ be an eigenvector of $Q(G)$ corresponding to $q_{min}(G)$. Then $X$ is also an eigenvector of $Q(\mathbb{G})$ corresponding to $q_{min}(\mathbb{G})$. If $G\neq \mathbb{G}$, then there exists an edge $v_{i}v_{j}\not\in E(\mathbb{G})$ where both $v_{i}$ and $v_{j}$ are $p$-dominators in $\mathbb{G}$. Combined with Lemmas \ref{le02,03}, \ref{le04,03}, as the proof of Theorem \ref{th05,1}, it is proved that $q_{min}(\mathscr{H}_{3,\frac{n}{2}})\leq X^{T}Q(\mathbb{G})X<X^{T}Q(\mathbb{G}+v_{i}v_{j})X\leq X^{T}Q(G)X=q_{min}(G)$ which contradicts the supposition that $q_{min}(G)= q_{min}(\mathscr{H}_{3,\frac{n}{2}})$. Thus, it follows that $G= \mathbb{G}$. Then the theorem follows.
 \ \ \ \ \ $\Box$
\end{proof}

Using Lemmas \ref{le02,03}, \ref{le0,07}, \ref{le04,03} and Theorem \ref{th04,07}, similar to the proof of Theorem \ref{th05,2}, we get the following Theorem \ref{th05,3}.

\begin{theorem}\label{th05,3} %------
Let $G$ be a nonbipartite connected graph of order $n$ and with domination number $\frac{n-1}{2}$. Then $q_{min}(G)\geq q_{min}(\mathscr{H}_{3,\frac{n-1}{2}})$ with equality if and only if $G\cong \mathscr{H}_{3,\frac{n-1}{2}}$.
\end{theorem}

In a same way, we get the following theorem.

\begin{theorem}\label{th05,4} %------
Among all nonbipartite graphs of order $n\geq 4$ and with odd-girth $g_{o}\leq5$, and domination number at least $\frac{n+1}{3}<\gamma\leq \frac{n}{2}$, we have

$\mathrm{(i)}$ if $n=4$, the least $q_{min}$ attains the minimum uniquely at $\mathscr{H}_{3,1}$;

$\mathrm{(ii)}$ if $n\geq 5$ and $\gamma=\frac{n-1}{2}$, the least $q_{min}$ attains the minimum uniquely at a $\mathscr{H}_{3,\frac{n-3}{2}}$;

$\mathrm{(iii)}$ if $n\geq 6$ and $\gamma=\frac{n}{2}$, the least $q_{min}$ attains the minimum uniquely at a $\mathscr{H}_{3,\frac{n}{2}}$;

$\mathrm{(iv)}$ if $n\geq 5$ and $n-2\gamma\geq 2$, then the least $q_{min}$ attains the minimum uniquely at a $\mathscr{H}_{3,\alpha}$ where $\alpha\leq\frac{n-3}{2}$ is the least integer such that $\lceil\frac{n-2\alpha-2}{3}\rceil+\alpha=\gamma$.

\end{theorem}

\section{Open problem}

{\bf Question } Let $G$ be a nonbipartite graph of order $n\geq 4$ and with domination number $\frac{n+1}{3}<\gamma\leq \frac{n}{2}$, and Let $\mathbb{S}=\mathscr{H}_{3,\alpha}$ be of order $n$ where $\alpha$ is the least integer such that $\lceil\frac{n-2\alpha-2}{3}\rceil+\alpha=\gamma$. How about the relation between their $q_{min}s$?

{\bf Remark} From Section 5, we know that for such $n\geq 4$ and $\frac{n+1}{3}<\gamma\leq \frac{n}{2}$, for a nonbipartite graph $G$ of order $n$ and with domination number $\gamma$, if $g_{o}(G)\leq 5$ or $\gamma=\frac{n-1}{2}, \frac{n}{2}$, then $q_{min}(\mathbb{S})\leq q_{min}(G)$ with equality if and only if $G\cong \mathbb{S}$. By some computations for some graphs, for a nonbipartite graph $G$ of order $n\geq 4$ and with domination number $\frac{n+1}{3}<\gamma\leq \frac{n}{2}$, it seems that $q_{min}(\mathbb{S})\leq q_{min}(G)$ with equality if and only if $G\cong \mathbb{S}$. So, we conjecture that such $\mathbb{S}$ has the smallest $q_{min}$ among the nonbipartite graphs of order $n\geq 4$ and with domination number $\frac{n+1}{3}<\gamma\leq \frac{n}{2}$.

\small {

}

\end{document}